\documentclass[twoside]{article}

\usepackage[accepted]{aistats2025}
\usepackage[utf8]{inputenc} 
\usepackage[T1]{fontenc}    
\usepackage{hyperref}       
\usepackage{url}            
\usepackage{booktabs}       
\usepackage{amsfonts}       
\usepackage{nicefrac}       
\usepackage{microtype}      
\usepackage[table]{xcolor}         
\usepackage{graphicx}

\usepackage{subcaption}
\usepackage{placeins}

\usepackage{amsmath, amssymb, amsthm}
\usepackage{algorithm}
\usepackage{algpseudocode}
\usepackage{dsfont}
\usepackage{enumitem}
\usepackage{subcaption}
\usepackage{authblk}

\algrenewcommand\algorithmicrequire{\textbf{Input:}}
\algrenewcommand\algorithmicensure{\textbf{Output:}}

\DeclareMathOperator*{\argmin}{arg\;min}
\newcommand{\R}{\mathbb{R}}

\renewcommand{\d}{\mathrm{d}}
\newcommand{\bs}{\boldsymbol}
\newcommand{\I}{\bl{I}}
\newcommand{\N}{\mathsf{N}}
\renewcommand{\P}{\mathbb{P}}
\newcommand{\E}{\mathbb{E}}
\newcommand{\bl}{\mathbf}
\newcommand{\tbl}[1]{\widetilde{\mathbf{#1}}}
\newcommand{\DtD}{\bl{D}^\top \bl{D}}
\newcommand{\XtX}{\bl{X}^\top \bl{X}}
\newcommand{\Tr}{\mathsf{Tr}}
\newcommand{\new}{\mathsf{new}}

\newcommand{\indep}{\perp\!\!\!\perp}
\newcommand{\LOOCV}{\mathsf{LOOCV}}
\newcommand{\GCV}{\mathsf{GCV}}
\newcommand{\nGCV}{\mathsf{ROTI}\text{-}\mathsf{GCV}}
\newcommand{\aGCV}{\mathsf{aROTI}\text{-}\mathsf{GCV}}
\newcommand{\bulk}{\mathsf{bulk}}

\newcommand{\op}{\mathsf{op}}

\newcommand{\red}{\mathsf{proj}}

\newcommand{\resid}{\mathsf{resid}}

\renewcommand{\O}{\mathcal{O}}

\newcommand{\1}{\mathds{1}}

\newtheorem{lemma}{Lemma}
\newtheorem{corollary}[lemma]{Corollary}
\newtheorem{theorem}{Theorem}
\theoremstyle{definition}
\newtheorem{definition}{Definition}
\theoremstyle{remark}
\newtheorem{remark}{Remark}

\usepackage{array}

\newcommand{\PreserveBackslash}[1]{\let\temp=\\#1\let\\=\temp}

\newcolumntype{C}[1]{>{\PreserveBackslash\centering}p{#1}}
\newcolumntype{R}[1]{>{\PreserveBackslash\raggedleft}p{#1}}
\newcolumntype{L}[1]{>{\PreserveBackslash\raggedright}p{#1}}
\usepackage[round]{natbib}

\begin{document}

\twocolumn[

\aistatstitle{ROTI-GCV: Generalized Cross-Validation for Right-Rotationally Invariant Data }

\aistatsauthor{ Kevin Luo \And Yufan Li \And  Pragya Sur }

\aistatsaddress{ Harvard University \And Harvard University \And Harvard University } ]

\begin{abstract}
Two key tasks in high-dimensional regularized regression are tuning the regularization strength for accurate predictions and estimating the out-of-sample risk. It is known that the standard approach --- $k$-fold cross-validation --- is inconsistent in modern high-dimensional settings. While leave-one-out and generalized cross-validation remain consistent in some high-dimensional cases, they become inconsistent when samples are dependent or contain heavy-tailed covariates. As a first step towards modeling structured sample dependence and heavy tails, we use right-rotationally invariant covariate distributions --- a crucial concept from compressed sensing. In the proportional asymptotics regime where the number of features and samples grow comparably, which is known to better reflect the empirical behavior in moderately sized datasets, we introduce a new framework, ROTI-GCV, for reliably performing cross-validation under these challenging conditions. Along the way, we propose new estimators for the signal-to-noise ratio and noise variance. We conduct experiments that demonstrate the accuracy of our approach in a variety of synthetic and semi-synthetic settings.
\end{abstract}

\section{INTRODUCTION}\label{sec:intro}

Regularized estimators are fundamental for modern high-dimensional statistics. In this context, the prototypical problem of ridge regression has gained renewed attention, with multiple works characterizing its out-of-sample risk in high dimensions, c.f.,~\cite{ dobriban2015highdimensional,hastie2022surprises}. Additionally, consistent estimates for this out-of-sample risk, crucial for optimizing the regularization parameter $\lambda$, have been established in the form of leave-one-out cross-validation (LOOCV) and generalized cross-validation (GCV) \citep{xu2021consistent,pmlr-v130-patil21a, atanasov2024scaling}. 

Crucially, all aforementioned analyses assume that samples are independent and identically distributed (i.i.d.) draws from some distribution with sufficient regularity, such as light tails. These assumptions often fail in practice, such as in financial returns, neuroscience, and climate data \citep{GROBYS2021101891,esd-12-1-2021, Tagliazucchi2013}. In such cases, cross-validation techniques can become biased (Theorem~\ref{thm:orig-gcv}, Remark~\ref{rem:gcv-consis}), necessitating alternatives. Only recently, \cite{bigot:hal-04559313} considered independent but non-identically distributed samples and \cite{atanasov2024riskcrossvalidationridge} explored Gaussian designs with both sample and feature dependence.
Understanding  cross-validation in high dimensions under more general sample dependencies and for designs with heavier tails than sub-Gaussians remains a  significant challenge. This paper takes the first step in addressing this issue.

We assess the impact of sample dependence and heavy tails in ridge regression and cross-validation using an alternative random matrix ensemble for the covariate distribution.  Instead of assuming a Gaussian matrix or i.i.d. rows from a fixed distribution, 
we require the singular value decomposition of the design $\bl{X}$ to be \emph{right-rotationally invariant} (see Definition~\ref{def:riri}). We characterize the ridge regression risk for these designs and introduce ROTI-GCV, a new GCV-inspired framework for consistently estimating the out-of-sample risk. Simulations show that ROTI-GCV outperforms traditional GCV and LOOCV in situations with correlated and heavy-tailed observations. 

Right-rotationally invariant designs facilitate tractable analysis while capturing structured forms of sample dependencies and heavy-tailed 
covariates. 
These designs have gained significant attention as alternatives to i.i.d.~designs for theoretical analysis in numerous studies 
\citep{takeda2006analysis,ma2017orthogonal,rangan2019vector,10.1214/21-AOS2101,gerbelot2020asymptotic,takeuchi2019rigorous,liu2022memory,li2023random,fan2021replica,xu2022approximatemessagepassingmultilayer}.
Recent works \citep{dudeja2023spectral, wang2022universality} showed that, under mild conditions, properties of convex estimators across various designs are well-approximated by those under a right-rotationally invariant design with the same spectrum, indicating a broad universality class for these distributions.

Formally, we operate in the proportional asymptotics regime, where the dimension $p$ scales proportionally with the sample size $n$. This regime, rooted in probability theory and statistical physics, has received significant attention due to its ability to accurately capture empirical phenomena 
\citep{hastie2022surprises,sur2019modern,sur2019likelihood,liang2022precise}. Furthermore,  results derived under this regime require minimal assumptions on the signal structure, resulting in theory and methods with broad practical utility \cite{song2024hede}. In regards to cross-validation for i.i.d.~samples, \cite{hastie2022surprises,rad2020scalable, wang2018approximate} established  that classical $k$-fold cross-validation inconsistently estimates the out-of-sample error in this high-dimensional regime, whereas LOOCV and GCV remain consistent, prompting our study of GCV.  However, our work focuses on non-i.i.d.~samples and non-Gaussian tailed designs, areas not sufficiently addressed in prior work.

We provide proofs for ridge regression, but emphasize that our core idea can extend to other penalties. Our proof hinges on characterizing the limit of GCV (Theorem \ref{thm:orig-gcv}). This reveals that the original GCV is inconsistent for the out-of-sample risk in our setting. We next observe that though GCV is inconsistent, its asymptotic limit can serve as an estimating equation for the unknown parameters parameterizing the risk. This estimating equation approach directly inspires a new cross-validation method that works for our challenging right-rotationally invariant designs. A similar strategy may be invoked for any other penalty where an analogue of Theorem \ref{thm:orig-gcv} is available. Therefore, we anticipate our approach to be broadly generalizable beyond ridge regression.

\section{PRELIMINARIES AND SETUP}

\subsection{Right-rotationally invariant designs}\label{ssec:riri-def}

\begin{definition}[Right-rotationally invariant design]\label{def:riri}
A random design matrix $\bl{X} \in \R^{n \times p}$ is \emph{right-rotationally invariant} if for any $\bl{O} \in \mathbb{O}(p)$, one has $\bl{XO} \stackrel{d}{=} \bl{X}$, where $\mathbb{O}(p)$ denotes $p \times p$ orthogonal matrices. Equivalently, $\bl{X}$ is right-rotationally invariant if and only if the singular value decomposition $\bl{X} = \bl{Q}^\top \bl{D} \bl{O}$ satisfies that $\bl{O}$ is independent of $(\bl{Q}, \bl{D})$ and drawn from the Haar measure on $\mathbb{O}(p)$, i.e. ``uniformly" distributed on $\mathbb{O}(p)$. 
\end{definition}

Some examples of right-rotationally invariant designs are as follows (note that an i.i.d.~entries standard Gaussian design is a member of this class, but the class includes many non-i.i.d.~and non-Gaussian designs as well; see Appendix~\ref{app:riri-proof} for proofs of invariance): 

\textit{(i) Autocorrelated data}: The rows of $\bl{X}$ can be drawn from an autocorrelated series, where $\bl{x}_i = \rho \bl{x}_{i - 1} + \sqrt{1 - \rho^2} \bl{z}_i$, with $\bl{z}_i$ being i.i.d. draws from $\N(0, \I_n)$.

\textit{(ii) $t$-distributed data}: The rows of $\bl{X}$ can be drawn from a multivariate $t$ distribution with as few as $3$ degrees of freedom, capturing heavy-tailed data such as financial returns \citep{GROBYS2021101891}.

\textit{(iii) Products of Gaussian matrices}: $\bl{X} = \bl{X}_1 \bl{X}_2 \cdot \cdots \cdot \bl{X}_k$, where $\bl{X}_1$ has $n$ rows and $\bl{X}_k$ has $p$ columns, while the remaining dimensions are arbitrary. See \cite{hanin2020products, hanin2021non} for connections to linear neural networks.

\textit{(iv) Matrix-Normals}: $\bl{X} \sim \mathsf{MN}(\bl{0}, \bs\Sigma^{\mathrm{row}}, \bs\Sigma^{\mathrm{col}})$, where $\bs\Sigma^{\mathrm{row}}$ can be arbitrary while $\bs\Sigma^{\mathrm{col}} \sim \mathsf{InvWishart}(\bl{I}_p, (1 + \delta)p)$, for any $\delta > 0$. See Appendix~\ref{app:notation} for details on distribution notation.  

\textit{(v) Equicorrelated data}: $\bl{X} \in \R^{n \times p}$ has independent columns, but each column follows a multivariate Gaussian distribution with covariance matrix $\bs\Sigma$, where $\Sigma_{ij} = \rho$ if $i \neq j$, and $\Sigma_{ii} = 1$. 

\textit{(vi) Spiked matrices}: $\bl{X} = \lambda \bl{V}\bl{W}^\top + \bl{G}$, where $\bl{V} \in \R^{n \times r}$ and $\bl{W} \in \R^{p \times r}$ are the first $r$ columns of two Haar matrices, and $G_{ij} \sim \N(0, 1)$.

See \cite[Figure~1]{li2023spectrumaware} for more such examples. Setting (ii) has heavier tails than supported by existing GCV theory; settings (i, iii, v, vi) have sample dependence; setting (iv) has both.

Right-rotationally invariant ensembles allow us enormous flexibility---they can simultaneously capture dependent rows and heavy-tails by allowing a rather general class of singular value distributions. 
The spectrum of right-rotationally invariant designs can essentially be arbitrary, and allowing for significant generalization from prior i.i.d.~designs used to study the behavior of LOOCV and GCV in high dimensions. Furthermore, the  universality class of these designs is extremely broad, as established in \cite{dudeja2023spectral,wang2022universality}.

To simplify some statements, we introduce some notation. Let $m_\bl{D}(z) = \frac{1}{p} \sum_{i = 1} ^p \frac{1}{D_{ii}^2 - z}$ and $v_\bl{D}(z) = \frac{1}{n}\sum_{i = 1} ^n \frac{1}{D_{ii}^2 - z}$, with $D_{ii} = 0$ if $i > \min(n, p)$. 
Also of use will be the derivatives $m'_{\bl{D}}$ and $v'_{\bl{D}}$. Note $m_\bl{D}$ is the Stieltjes transform of the empirical measure of $(\DtD_{ii})_{i = 1} ^p$; $v_\bl{D}$ is the Stieltjes transform of $(\mathbf{D}\mathbf{D}^\top_{ii})_{i = 1} ^n$.

\subsection{Problem Setup}

We study high-dimensional ridge regression and GCV with right-rotationally invariant designs. Concretely, 
we observe $(\bl{X}, \bl{y})$, such that
\begin{equation}\label{eq:riri-gen}
    \bl{y} = \bl{X}\bs\beta + \bs\epsilon,
\end{equation}
with $\bl{y} \in \R^n$, $\bl{X} \in \R^{n \times p}$, and $\bs\epsilon, \bs\beta \in \R^{p}$. Hence $n$ is the sample size, $p$ is  both the data dimension and the parameter count, $\bs\beta$ is the signal, and $\bs\epsilon$ is a component-wise independent noise vector, where each entry has mean $0$ and variance $\sigma^2$. 
We study the expected out-of-sample risk on an independent draw from a new, potentially different right-rotationally invariant design, $\tbl{X} = \tbl{Q}^\top \tbl{D} \tbl{O}$, with $\tbl{X} \in \R^{n' \times p}$. This corresponds to learning a model on one interdependent population, and using it on a separate interdependent population. We analyze the conditional risk
\begin{equation}\label{eq:risk-def}
    R_{\bl{X}, \bl{y}}\left(\hat{\bs\beta}(\bl{X}, \bl{y}), \bs\beta\right) = \frac{1}{n'} \E\left[\left.\left\|\tbl{X}\hat{\bs\beta} - \tbl{X}\bs\beta\right\|^2 \right| \bl{X}, \bl{y}\right].
\end{equation}
Such conditional risks have been examined in prior works \citep{hastie2022surprises, pmlr-v130-patil21a}.
We focus on the risk of the ridge estimator
\begin{equation}
    \hat{\bs\beta}_\lambda = \argmin_{\bl{b} \in \R^p} \left\{\|\bl{y} - \bl{X}\bl{b}\|_2^2 + \lambda \|\bl{b}\|^2_2\right\},
\end{equation}
which has closed form $\hat{\bs\beta} = (\XtX + \lambda \I)^{-1} \bl{X}^\top \bl{y}$. As such, sometimes we will write $R_{\bl{X}, \bl{y}}(\lambda) := R_{\bl{X}, \bl{y}}(\hat{\bs\beta}_\lambda, \bs\beta)$ for convenience. Before stating our results, we note the running assumptions used in what follows. 
\begin{enumerate}[label = A\arabic*]
    \item \textcolor{black}{We consider a sequence of right-rotationally invariant designs $\bl{X}_1, \bl{X}_2, \dots, \bl{X}_{n}, \dots$ where each $\bl{X}_i \sim \bl{Q}_i^\top \bl{D}_i \bl{O}_i$, and a sequence of signal vectors $\bs\beta_1, \bs\beta_2, \dots, \bs\beta_n, \dots$ such that for all $n$, $\bs\beta_n \in \R^p$ and $\|\bs\beta_n\|= r\sqrt{n} $; $\bl{y}_i$ is generated by Eq.~\eqref{eq:riri-gen}. In the sequel, we drop the dependence on $n$ and write $\bl{X}, \bl{y}$ and $\bs\beta$ with the indexing implicit. We refer to $(\bl{X}, \bl{y})$ as the training data.} \label{item:scaling-matrix}
    \item $\limsup \lambda_{\max}(\bl{D}_n) < C$ a.s.~for some constant $C$.\label{item:bound-op}
    \item $\bl{X}_n \in \mathbb{R}^{n \times p(n)}$, $n\to\infty$, $p(n) / n \rightarrow \gamma  \in (0, \infty)$. 
    \item $\bs\epsilon$ is independent of $\bl{X}$; $\epsilon_i$'s are independent; $\E[\epsilon_i] = 0$, $\E[\epsilon_i^2] = \sigma^2$, and $\E[\epsilon_i^{4 + \eta}] <\infty$ for some $\eta > 0$.
    \item To fix scaling and simplify our results, we assume that $\frac{n \E[\Tr(\tbl{D}^\top \tbl{D})]}{n' p} = 1$; this ensures the risk is defined and maintains the scale of the training set relative to the test set (See Remark~\ref{rmk:scaling-explanation} of Appendix). \label{item:scaling} 
    \item $\bs\beta_n$ is fixed or random independent of $(\bl{X}_n, \bs\epsilon)$.  \label{item:indep-beta}
    \item $\tbl{X}$ is an independent draw from a potentially different right-rotationally invariant ensemble.  \label{item:indep-test}
\end{enumerate}
Of crucial importance will be the two parameters $r^2$ and $\sigma^2$, which dictate the behavior of the problem.
\begin{remark}\label{rem:extra-noise}
To clarify why the risk in Eq.~\ref{eq:risk-def} is relevant, if one samples $(\tbl{X}, \tbl{y})$ according to Eq.~\eqref{eq:riri-gen}, then $R_{\bl{X}, \bl{y}}(\lambda) = \E[ \|\tbl{y} - \tbl{X}\hat{\bs\beta}_\lambda \|^2 \mid \bl{X}, \bl{y}] - \sigma^2$, meaning the risk defined corresponds, up to a factor of $\sigma^2$, to the mean-squared error on the test set.
\end{remark}

\begin{remark}\label{rmk:scaling-1}
    One should think of $\bl{X}$ in our setting as akin to the normalized matrix $\bl{X}/\sqrt{n}$ in the i.i.d. setting. This further ensures that $\|\XtX\|_{\op} = \|\DtD\|_{\op} = \O(1)$. To maintain the signal to noise ratio under this scaling, $\|\bs\beta\| = \O(\sqrt{n})$ (Assumption~\ref{item:scaling-matrix}); the test set is appropriately scaled through Assumption~\ref{item:scaling}.
\end{remark}

\section{GENERALIZED CROSS-VALIDATION}\label{sec:gcv-vanilla}
Prior work \citep{xu2021consistent, pmlr-v130-patil21a} established that for many  designs with i.i.d. samples, LOOCV and GCV estimate the out-of-sample error consistently in the proportional asymptotics regime, while $k$-fold CV does not. In our setting, LOOCV is intractable without additional assumptions on $\bl{Q}$ (see Appendix~\ref{app:loocv-bad}). However, GCV is tractable, motivating our study here.   Initially introduced as a technique for selecting the regularization parameter in smoothing splines  \cite{craven1978smoothing}, GCV's optimality and consistency has been shown in various tasks \citep{li1986asymptotic, gu2005optimal, zhang2015divide, xu2018optimal, xu2019distributed}. 
The GCV cross-validation metric for ridge regression is as follows \citep{golubgcv}:
\begin{equation}\label{eq:gcv}
    \GCV_n(\lambda) = \frac{1}{n} \sum_{i = 1} ^n \left(\frac{y_i - \bl{x}_i^\top \hat{\bs\beta}_{\lambda}}{1 - \Tr(\bl{S}_{\lambda})/n}\right)^2. 
\end{equation}
Here $\bl{S}_\lambda = \bl{X}^\top (\XtX + \lambda \I)^{-1} \bl{X}^\top$ is the smoother matrix. Our first result is as follows: 
\begin{theorem}\label{thm:orig-gcv}
    Under the stated assumptions,
    \begin{equation}
        \GCV_n(\lambda) - \frac{r^2 (v_{\bl{D}}(-\lambda) - \lambda v'_{\bl{D}}(-\lambda)) + \sigma^2\gamma v'_{\bl{D}}(-\lambda)}{\gamma v_{\bl{D}}(-\lambda)^2} \xrightarrow{a.s.} 0.
    \end{equation}
\end{theorem}
Further, the risk takes the following form:
\begin{theorem}\label{thm:ridge}
    Fix $0 < \lambda_1 < \lambda_2 < \infty$. Define
    \begin{multline}\label{Rdef}
        \mathsf{R}_{\bl{D}}(r^2, \sigma^2) = r^2 \left(\frac{\lambda^2}{\gamma} v'_{\bl{D}}(-\lambda) + \frac{\gamma - 1}{\gamma}\right) \\
        + \sigma^2 ( v_{\bl{D}}(-\lambda) - \lambda v_{\bl{D}}'(-\lambda))
    \end{multline}
    Under the stated assumptions,
    \begin{equation}
        \sup_{\lambda \in [\lambda_1, \lambda_2]} \left| R_{\bl{X}, \bl{y}} \left( \hat{\bs\beta}_\lambda, \bs\beta \right) - \mathsf{R}_{\bl{D}}(r^2, \sigma^2) \right| \xrightarrow{a.s.} 0.
    \end{equation}
    where we recall that $r^2=\|\bs\beta _n\|^2/n$ and $\sigma^2=\mathbb{E}{\epsilon_i^2}$.
\end{theorem}
\textcolor{black}{In Theorem~\ref{thm:ridge}, one can interpret the term involving $r^2$ as the bias of the estimator, and the term with $\sigma^2$ as its variance. As such, one can then interpret Theorem~\ref{thm:orig-gcv} as showing how GCV reconstructs the bias and variance from the data.}
\begin{remark}\label{rem:gcv-consis}
\textcolor{black}{It is shown in \cite{pmlr-v130-patil21a} that for i.i.d. designs, $\Delta_n(\lambda) := \GCV_n(\lambda) - R_{\bl{X}, \bl{y}}(\lambda) - \sigma^2$ approaches zero uniformly over compact intervals\footnote{The extra $\sigma^2$ term is discussed in Remark~\ref{rem:extra-noise}.}. Their proof of this relies explicitly on identities of the Stieltjes transform given by the Silverstein equation \citep{SILVERSTEIN1995331}. These identities do not exist for right-rotationally invariant designs, and thus the GCV becomes biased for these designs, meaning $\Delta_n(\lambda) \not \to 0$. The more the spectrum departs from that of the i.i.d. setting, the more biased the estimates.}
\end{remark}

\subsection{A Modified GCV}

The original GCV (in Eq.~\eqref{eq:gcv}) can be seen as starting with the training error and finding a rescaling that recovers the out-of-sample risk. To construct our alternative GCV metric, which is provably consistent for right-rotationally invariant designs, we likewise begin with the training error, but instead use it to produce estimating equations for $r^2$ and $\sigma^2$. 
We then use these to compute the risk formula given by Theorem~\ref{thm:ridge}. 

\subsubsection{Estimators for $r^2, \sigma^2$}
Consistent estimation of $\sigma^2$  has been studied in \cite{li2023spectrumaware}. We propose a different approach since we observe it performs better in finite samples. In particular, we establish the following result: 
\begin{lemma}\label{lem:gcv-as} Under our assumptions, for any $\lambda > 0$,
\begin{multline}\label{eq:gcv-as}
    \frac{1}{n} \sum_{i = 1} ^n \left(y_i - \bl{x}_i^\top \hat{\bs\beta}_\lambda \right)^2 -
    \left[r^2 \frac{\lambda^2}{\gamma} \left(v_{\bl{D}}(-\lambda) - \lambda v'_{\bl{D}}(-\lambda)\right)\right. \\
    \left. \phantom{\frac{\lambda^2}{\gamma}} + \sigma^2\lambda^2 v'_{\bl{D}}(-\lambda)\right] \xrightarrow{a.s.} 0.
\end{multline}
\end{lemma}
Observe that Eq.~\eqref{eq:gcv-as} can be used as an estimating equation, meaning that, if one treats \eqref{eq:gcv-as} as an equality to 0 (i.e. replacing ``$\xrightarrow{a.s.}$'' with ``$=$''), then we can solve for the unknown quantities $r^2$ and $\sigma^2$.
Recall that $\gamma = p/n$, $v_\bl{D}$, and $v'_\bl{D}$, defined in Section~\ref{ssec:riri-def}, are all data dependent quantities. 
Now, Eq.~\eqref{eq:gcv-as} yields a distinct estimating equation for \textit{each} value of $\lambda$. While it takes only two equations to solve for $r^2$ and $\sigma^2$, we instead use a grid of $\lambda$ and find OLS estimates for $r^2$ and $\sigma^2$ to improve robustness. Explicitly, the scheme proceeds as shown in steps 1-6 of Algorithm~\ref{alg:roti-gcv}. 

Using Lemma~\ref{lem:gcv-as}, we can establish that the following:
\begin{corollary}\label{cor:consistent}
    Under Assumptions \ref{item:scaling-matrix}-\ref{item:indep-test} and a non-degeneracy condition on $\bl{D}$ (see Appendix~\ref{app:nondegen} for details), $\hat{r}^2$ and $\hat{\sigma}^2$ from Algorithm \ref{alg:roti-gcv} are strongly consistent for $r^2$ and $\sigma^2$, respectively.
\end{corollary}
Empirically, we find that normalizing the spectrum of $\XtX$ and taking $\{\lambda_i\}_{i=1}^L$ logarithmically spaced between $1$ and $10^{2.5}$ works well. 
\begin{algorithm}[h]    
    \caption{ROTI-GCV($\lambda$)}\label{alg:roti-gcv}
    \begin{algorithmic}[1]
        \Require{$\bl{X}, \bl{y}$,  $\lambda$; for use in estimation, grid of regularization strengths $\lambda_1 < \dots < \lambda_L$;  }
        \Ensure{ Estimate for ${R}_{\bl{X}, \bl{y}}(\hat{\bs\beta}, \bs\beta)$}
        \For{$i = 1, \dots, L$}
            \State $a_i \gets \frac{\lambda^2_i}{\gamma} \left(v_{\bl{D}}(-\lambda_i) - \lambda_i v'_{\bl{D}}(-\lambda_i)\right)$
            \State $b_i \gets \lambda_i^2 v'_{\bl{D}}(-\lambda_i)$
            \State $t_i \gets \|\bl{y} - \bl{X}\hat{\bs\beta}_{\lambda_i}\|_2^2. $
        \EndFor
        \State $\hat{\sigma}^2_{\{\lambda_i\}_{i = 1} ^L}(\bl{X}, \bl{y}) \gets \frac{\sum_{i = 1} ^L \left(\frac{t_i}{a_i} - \frac{t_1}{a_1} \right) \left(\frac{b_i}{a_i} - \frac{b_1}{a_1}\right)}{\sum_{i = 1} ^L\left(\frac{b_i}{a_i} - \frac{b_1}{a_1}\right)^2}$, \qquad $\hat{r}^2_{\{\lambda_i\}_{i = 1} ^L}(\bl{X}, \bl{y}) \gets \frac{\sum_{i = 1} ^L \left(\frac{t_i}{b_i} - \frac{t_1}{b_1} \right) \left(\frac{a_i}{b_i} - \frac{a_1}{b_1}\right)}{\sum_{i = 1} ^L\left(\frac{a_i}{b_i} - \frac{a_1}{b_1}\right)^2}.$
        \State \textbf{return } $\mathsf{R}_{\bl{D}}(\hat{r}^2, \hat{\sigma}^2)$ for $\mathsf{R}_{\mathbf{D}}(\cdot, \cdot)$ defined in \eqref{Rdef}
    \end{algorithmic}
\end{algorithm}

\subsubsection{Plug-in estimation}
We now define $\nGCV(\lambda) := \mathsf{R}_\bl{D}(\hat{r}^2, \hat{\sigma}^2)$ (recall $\mathsf{R}_\bl{D}$ from Theorem~\ref{thm:ridge}) and tune $\lambda$ using this metric (Algorithm~\ref{alg:roti-gcv}); below we establish that this quantity is uniformly consistent for the out-of-sample risk over compact intervals.
\begin{corollary}\label{cor:empirical-unif}
    Fix any $0 < \lambda_1 < \lambda_2 < \infty$. Under Assumptions~\ref{item:scaling-matrix}-\ref{item:indep-test}, 
    one has
    \begin{equation}
        \sup_{\lambda \in  [\lambda_1, \lambda_2]}\left|\nGCV(\lambda) - R_{\bl{X}, \bl{y}}(\hat{\bs\beta}_\lambda, \bs\beta)\right| \xrightarrow{a.s.} 0.
    \end{equation}
\end{corollary}
\begin{remark}
    \textcolor{black}{Uniform convergence, unlike pointwise convergence, ensures the risk attained by the minimizer (over a compact interval) of our cross-validation metric is close to optimal.} See Appendix~\ref{app:unif} for a short proof.
\end{remark}

\section{GCV UNDER SIGNAL-PC ALIGNMENT}\label{sec:gcv-aligned}
An issue that emerges when attempting to apply ROTI-GCV in practice is that the signal often tends to align with the top eigenvectors of the data, violating the independence assumption imposed on $\bs\beta$ and $\bl{O}$. The right-rotationally invariant assumption, for fixed $\bs\beta$, inherently assumes that the signal is incoherent with respect to the eigenbasis. However, this assumption is generally not true, and in fact in the i.i.d. anisotropic case, the geometry of $\bs\beta$ with respect to the top eigenvectors of $\bs\Sigma$ is known to significantly influence the behavior of ridge regression (see e.g. \cite{wu2020optimal}).

Adapting the approach given in \cite{li2023spectrumaware}, 
we model this scenario using two index sets $\mathcal{J}_a$ and $\mathcal{J}_c$ of distinguished eigenvectors, which are \textit{aligned} and \textit{coupled} eigenvectors, respectively. 
We replace Assumptions~\ref{item:indep-beta}~and~\ref{item:indep-test} in Section~\ref{sec:gcv-vanilla} with the following:
\begin{enumerate}[label = A\arabic*, resume]
    \item The true signal $\bs\beta$ is given as $\bs\beta = \bs\beta' + \sum_{i \in \mathcal{J}_a}\sqrt{n}\alpha_i\bl{o}_i$, where $\|\bs\beta'\|/\sqrt{n} = r$; $\bs\beta'$ is independent of $(\bl{X}, \bs\epsilon)$. Here $\bl{o}_i$ refers to the $i$th row of $\bl{O}$, which is an eigenvector of $\XtX$. An eigenvector $\bl{o}_i$ is an \textit{aligned} eigenvector if $\bl{o}_i \in \mathcal{J}_a$, as then $\bl{o}_i$ aligns with the signal $\bs\beta$; \label{item:aligned} 
    \item The test data $\tbl{X} \sim \tbl{Q}^\top \tbl{D} \tbl{O}$ is drawn from a separate right-rotationally invariant ensemble; $\tbl{O}$ is distributed according to a Haar matrix with rows conditioned to satisfy $\bl{o}_i = \tbl{o}_i$ for all $i \in \mathcal{J}_c$. An eigenvector $\bl{o}_i$ is then \emph{coupled} if $i \in \mathcal{J}_c$, as $\bl{o}_i$ is shared between the test and train sets. \label{item:coupled}
\end{enumerate}
We require two more assumptions:
\begin{enumerate}[label = A\arabic*, resume]
    \item  $\mathcal{J}_a, \mathcal{J}_c$, $\bs\alpha = (\alpha_i)_{i \in \mathcal{J}_a}$ are fixed, not changing with $n, p$; this ensures consistent estimation of these distinguished directions.
    \item $\limsup \|\E[\tbl{D}^\top \tbl{D}]\|_\op < C$ for some constant $C$. This differs from Assumption~\ref{item:bound-op}, as it is on the test data. 
    \label{item:j-bounded-spec} 
\end{enumerate}

The main difficulty that arises in our setting, compared to \cite{li2023spectrumaware}, is that we must account for how the geometry of the eigenvectors of the test set relates to those of the training set. One should expect, for example, in some types of structured data that the top eigenvectors of the training data and test data are closely aligned. We account for this through the coupled eigenvectors condition (\ref{item:coupled}). Furthermore, without this coupling, any alignment between the eigenvectors of the training set and the signal $\bs\beta$ will not exist in the test set, since then the eigenvectors of the test set would be independent of this alignment. We are still interested in the same test risk $R_{\bl{X}, \bl{y}}$ as in Eq.~\eqref{eq:risk-def}, except now note that $\tbl{X}$ and $\bl{X}$ are dependent.

\begin{theorem}\label{thm:align-gcv}
    Under Assumptions~\ref{item:scaling-matrix}-\ref{item:scaling} and \ref{item:aligned}-\ref{item:j-bounded-spec}, for any $0 < \lambda_1 < \lambda_2 < \infty$, we have
    \begin{equation}
        \sup_{\lambda \in [\lambda_1, \lambda_2]} \left| R_{\bl{X}, \bl{y}} \left( \hat{\bs\beta}_\lambda, \bs\beta \right) - \mathcal{R}_{\bl{X}, \bl{y}}(r^2, \sigma^2, \bs\alpha) \right| \xrightarrow{a.s.} 0,
    \end{equation}
    where
    \begin{equation}
        \mathcal{R}_{\bl{X}, \bl{y}}(r^2, \sigma^2, \bs\alpha) = \mathcal{B}_{\bl X}(\hat{\bs\beta}_\lambda, \bs\beta) + \mathcal{V}_{\bl X}(\hat{\bs\beta}_\lambda, \bs\beta).
    \end{equation}
    Let $\mathfrak{d}_i^2 = \E\left[\left(\tbl{D}^\top \tbl{D}\right)_{ii}\right]$~and~$\mathfrak{d}_\bulk^2 = \frac{1}{p - |\mathcal{J}_c|} \sum_{i \not \in \mathcal{J}_c} \mathfrak{d}_i^2$. 
    \begin{align*}
        \mathcal{B}_\bl{X}(\hat{\bs\beta}_\lambda, \bs\beta) &= \frac{\lambda^2}{n} \sum_{i = 1} ^p \left[\frac{\left(r^2/\gamma  + \1(i \in \mathcal{J}_a)n\alpha_i^2\right)}{(D_{ii}^2 + \lambda)^2} \right. \\
        &\qquad \cdot \left. \phantom{\frac{a^2}{b}}\left(\mathfrak{d}_\bulk^2 + \1(i \in \mathcal{J}_c)(\mathfrak{d}_i^2 - \mathfrak{d}^2_\bulk)\right)\right] \\
        \mathcal{V}_\bl{X}(\hat{\bs\beta}_\lambda, \bs\beta) &= \frac{\sigma^2}{n} \sum_{i = 1} ^n \frac{D_{ii}^2\left(\mathfrak{d}_\bulk^2 + \1(i \in \mathcal{J}_c)(\mathfrak{d}_i^2 - \mathfrak{d}^2_\bulk)\right)}{(D_{ii}^2 + \lambda)^2} 
    \end{align*}
    \begin{remark}
    \textcolor{black}{Comparing the two terms here with those in Theorem~\ref{thm:orig-gcv}, observe that they are equal when $\mathcal{J}_a$ and $\mathcal{J}_c$ are empty. Furthermore, observe that each element of $\mathcal{J}_a$ adds a large contribution to the bias, scaling with the size of the alignment $\alpha_i$; similarly, elements of $\mathcal{J}_c$ contribute additional bias and variance. }
    \end{remark}
    \begin{remark}
        The assumption that the top eigenvectors of training and test sample covariance are exactly equal (Assumption~\ref{item:coupled}) is not expected to hold perfectly in practice. However, the idea is that the approximation can lead to more robust procedures. We thus expect this approximation to hold when the training and test data have ``spiky'' spectra, where the top few eigenvalues are heavily separated from the bulk, and these top eigenvectors of the test and training data are close.
    \end{remark}
\end{theorem}

\subsection{Estimating $r^2$, $\sigma^2$, and $\alpha_i$}

Our approach is simply to estimate the $\alpha_i$'s using the classical principal components regression (PCR);
see \cite{jolliffe1982note, hubert2003robust} for details, also
described in steps 1 and 2 of Algorithm~\ref{alg:aroti-gcv}.
We then compute a transformed model which is right-rotationally invariant, from which Algorithm~\ref{alg:roti-gcv} steps 1-6 can be used to find $r^2$ and $\sigma^2$. 

We require some notation; we focus on subsets $S$ of the indices of all nonzero singular values $\mathcal{N} = \{i \in [p] : (\bl{D}^\top \bl{D})_{ii} > 0\}$. Let $\bl{O}_S \in \mathbb{R}^{|S| \times p}$ for any $S$ denote the rows of $\bl{O}$ indexed by $S$. Let $\bl{D}_S \in \R^{(n - |S|) \times p}$ denote the rows of $\bl{D}$ indexed by $S$. Finally, for a set $S \subset \mathcal{N}$, denote $\overline{S} = \mathcal{N} \setminus S$. The estimation for $r^2, \sigma^2$, and $\bs\alpha$ then proceeds as in steps 1-5 of Algorithm~\ref{alg:aroti-gcv}.
\begin{lemma}\label{lem:align-consistent}
Under the additional assumption that $\bs\epsilon$ has Gaussian entries\footnote{The assumption of Gaussian entries is technical, and we believe it can be removed with some work.}, $\tilde{r}^2$, $\tilde{\sigma}^2$, and $\hat{\bs\alpha}$, from Algorithm~\ref{alg:aroti-gcv}, are all strongly consistent. 
\end{lemma}

\subsection{Cross-validation metric}
Our approach for cross-validation is to then again explicitly compute the risk formulas $\mathcal{B}_{\bl{X}}(\hat{\bs\beta}_\lambda, \bs\beta)$ and $\mathcal{V}_{\bl{X}}(\hat{\bs\beta}_\lambda, \bs\beta)$,
where we estimate the unknown parameters from data. To apply our method, we therefore need to consistently estimate $r^2, \sigma^2, \bs\alpha$, and $\{\mathfrak{d}_i^2\}_{i \in \mathcal{J}_c}$ and $\mathfrak{d}_\bulk^2$. Lemma~\ref{lem:align-consistent} gives us most of these parameters, but we still require consistent estimates of the $\mathfrak{d}_i$'s. This is a nonissue in cases where $\bl{D}$ and $\tbl{D}$ are modeled as deterministic. If they are random, we expect that if the test data is drawn from the same right-rotationally invariant ensemble as the training data, then, in practice, one can use the eigenvalues of the training data to do this. However, our framework further allows for the test distribution to be drawn from a different right-rotationally invariant ensemble. In such cases, it is possible to then use test data to do this estimation because it only depends on the test covariates $\tbl{X}$ (which the practitioner may have access to), but not its labels; this is done for experiments in Section~\ref{sec:numericals}. 
We therefore make one final assumption:
\begin{enumerate}[label = A\arabic*]
\setcounter{enumi}{10}
\item We assume there exist estimators $\hat{\mathfrak{d}}_i$ and $\hat{\mathfrak{d}}_\bulk$ where we have $\sup_{i \in \mathcal{J}_a \cup \mathcal{J}_c} |\hat{\mathfrak{d}}_i - \mathfrak{d}_i| \xrightarrow{a.s.} 0$, and  $|\hat{\mathfrak{d}}_\bulk - \mathfrak{d}_\bulk| \xrightarrow{a.s.} 0$.\label{item:consist-eigen}
\end{enumerate}
We then take our validation metric as $\aGCV(\lambda) = \mathcal{R}_{\bl{X}, \bl{y}}(\tilde{r}^2, \tilde{\sigma}^2, \hat{\bs\alpha}, \{\hat{\mathfrak{d}}_i\}_{i \in \mathcal{J}_c, \bulk})$, where the latter expression denotes substituting every parameter with its estimated version. Algorithm \ref{alg:aroti-gcv} summarizes our entire procedure\footnote{The algorithm requires a choice of $\mathcal{J}_a, \mathcal{J}_c$; we give a brief overview of selecting $\mathcal{J}_a$ and $\mathcal{J}_c$ in Section~\ref{sec:semi-real-numericals}, followed by a detailed discussion in Appendix~\ref{app:verify}.}.

\begin{algorithm}
    \caption{aROTI-GCV($\lambda$)}\label{alg:aroti-gcv}
    \begin{algorithmic}[1]
        \Require{$\bl{X}, \bl{y}$, regularization strength $\lambda$; index sets $\mathcal{J}_a, \mathcal{J}_c$; consistent estimators $\hat{\mathfrak{d}}_i^2$ and $\hat{\mathfrak{d}}_\bulk^2$; for use in estimation of $r^2, \sigma^2$, grid of regularization strengths $\lambda_1 < \dots < \lambda_L$;}
        \Ensure{ Estimate for ${R}_{\bl{X}, \bl{y}}(\hat{\bs\beta}, \bs\beta)$}
        \State $\bl{X}_\red \gets \bl{X} \bl{O}_{\mathcal{J}_a}^\top$
        \State $\hat{\bs\alpha} \gets (\bl{X}_\red^\top \bl{X}_\red)^{-1} \bl{X}_\red^\top \bl{y}/\sqrt{n}$
        \State $\bl{y}_\new = (\bl{D}_{\bar{\mathcal{J}_a}}^\top \bl{D}_{\bar{\mathcal{J}_a}})^{1/2} ( \bl{X}_\resid^\top \bl{X}_\resid)^{-1} \bl{X}_\resid^\top \bl{y}$
        \State $\bl{X}_\new = (\bl{D}_{\bar{\mathcal{J}_a}}^\top \bl{D}_{\bar{\mathcal{J}_a}})^{1/2} \bl{O}_{\bar{\mathcal{J}_a}}$
        \State \text{Compute } $\tilde{r} = \hat{r}^2_{\{\lambda_i\}_{i = 1} ^L}(\bl{X}_\new, \bl{y}_\new), \tilde{\sigma}^2 = \hat{\sigma}^2_{\{\lambda_i\}_{i = 1} ^L}(\bl{X}_\new, \bl{y}_\new)$
        \State \textbf{return } $\mathcal{R}_{\bl{X}, \bl{y}}\left(\tilde{r}^2, \tilde{\sigma}^2, \hat{\bs\alpha}, \{\hat{\mathfrak{d}}\}_{i \in \mathcal{J}_c, \bulk}\right)$ 
    \end{algorithmic}
\end{algorithm}

\begin{corollary}\label{cor:align-unif}
    Under Assumptions~\ref{item:scaling-matrix}-\ref{item:scaling}, \ref{item:aligned}-\ref{item:consist-eigen}, and the additional assumption that the entries of $\bs\epsilon$ are Gaussian,
    $\aGCV$ satisfies 
    \begin{equation}
        \sup_{\lambda \in [\lambda_1, \lambda_2]} |\aGCV(\lambda) - R_{\bl{X}, \bl{y}}(\hat{\bs\beta}_\lambda, \bs\beta)| \xrightarrow{a.s.} 0.
    \end{equation}
\end{corollary}
\begin{remark}
    This states that aROTI-GCV is uniformly consistent on compact intervals, ensuring the minimizer of aROTI-GCV attains risk that is close to optimal. 
\end{remark}
\section{NUMERICAL EXPERIMENTS}\label{sec:numericals}

In this section, we demonstrate the efficacy of our method via numerical experiments. In producing each plot below, we generate a training set according to a prescribed distribution, as well as a signal vector $\bs\beta$. We then repeatedly resample the training noise $\bs\epsilon$ and run each cross-validation method, in addition to computing the true risk on the test set.

\subsection{Experiments on right-rotationally invariant data}\label{ssec:rri-plot}

\begin{figure*}[t]
    \centering
    \begin{subfigure}[t]{0.49\textwidth}
        \centering
        \includegraphics[width=\textwidth]{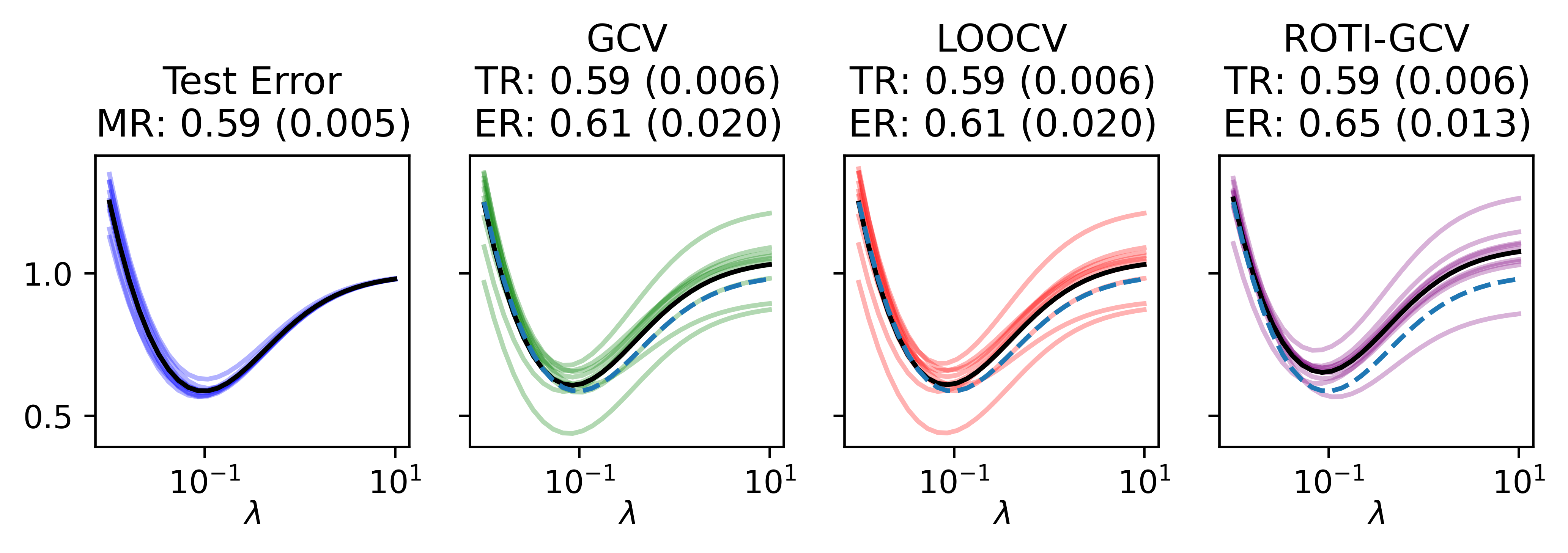}
        \caption{Gaussian data}
        \label{fig:gaussian-gcv}
    \end{subfigure}
    \begin{subfigure}[t]{0.49\textwidth}
        \centering
        \includegraphics[width=\textwidth]{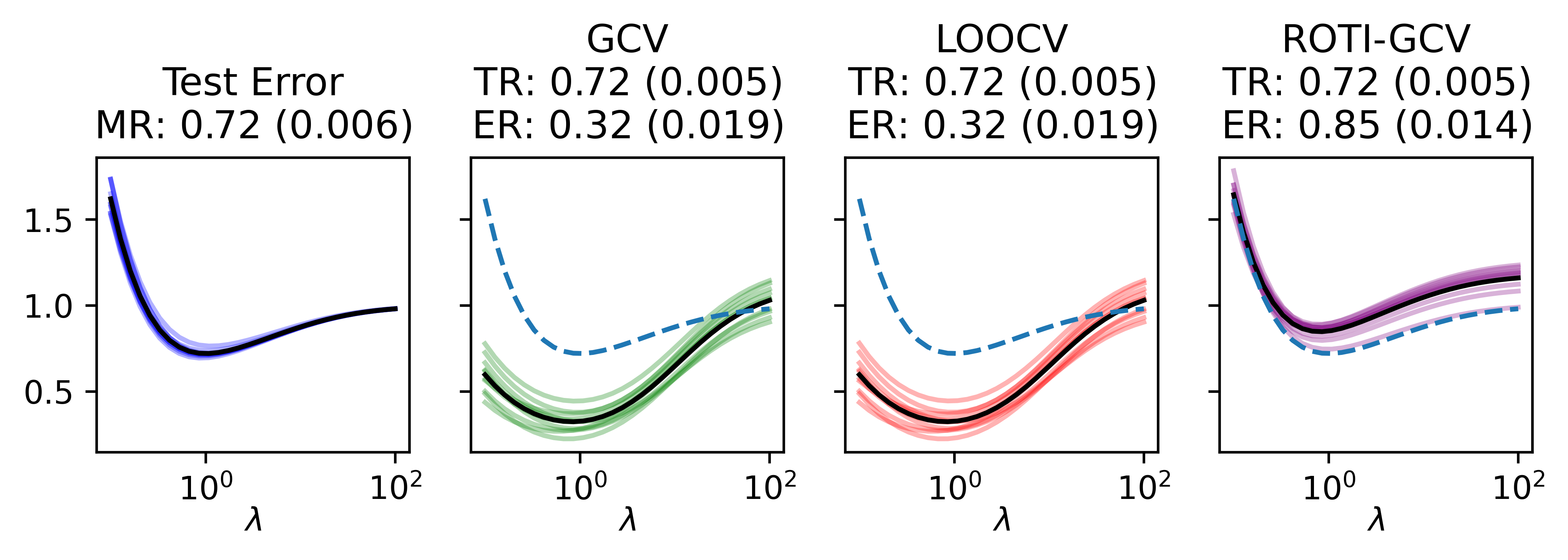}
        \caption{Autocorrelated data, $\rho = 0.8$}
        \label{fig:auto-gcv}
    \end{subfigure}
    \begin{subfigure}[t]{0.49\textwidth}
        \centering
        \includegraphics[width=\textwidth]{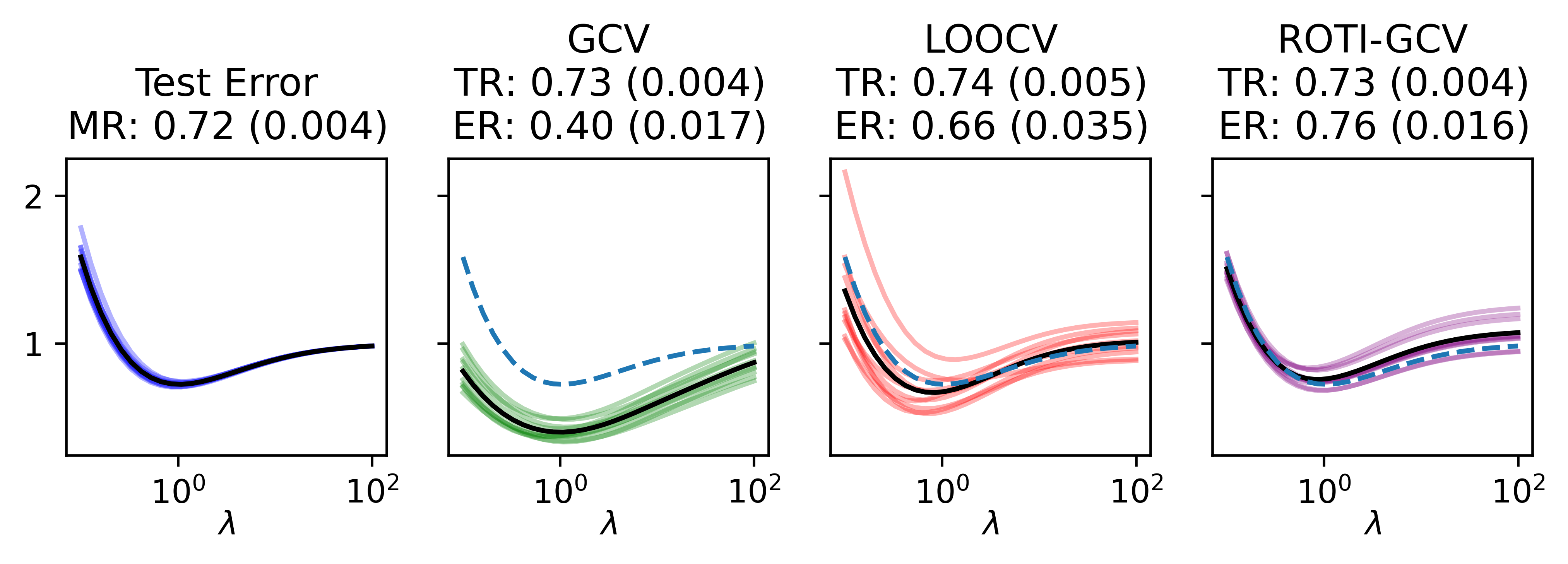}
        \caption{$t$-distributed data}
        \label{fig:t-gcv}
    \end{subfigure}
    \begin{subfigure}[t]{0.49\textwidth}
        \centering
        \includegraphics[width=\textwidth]{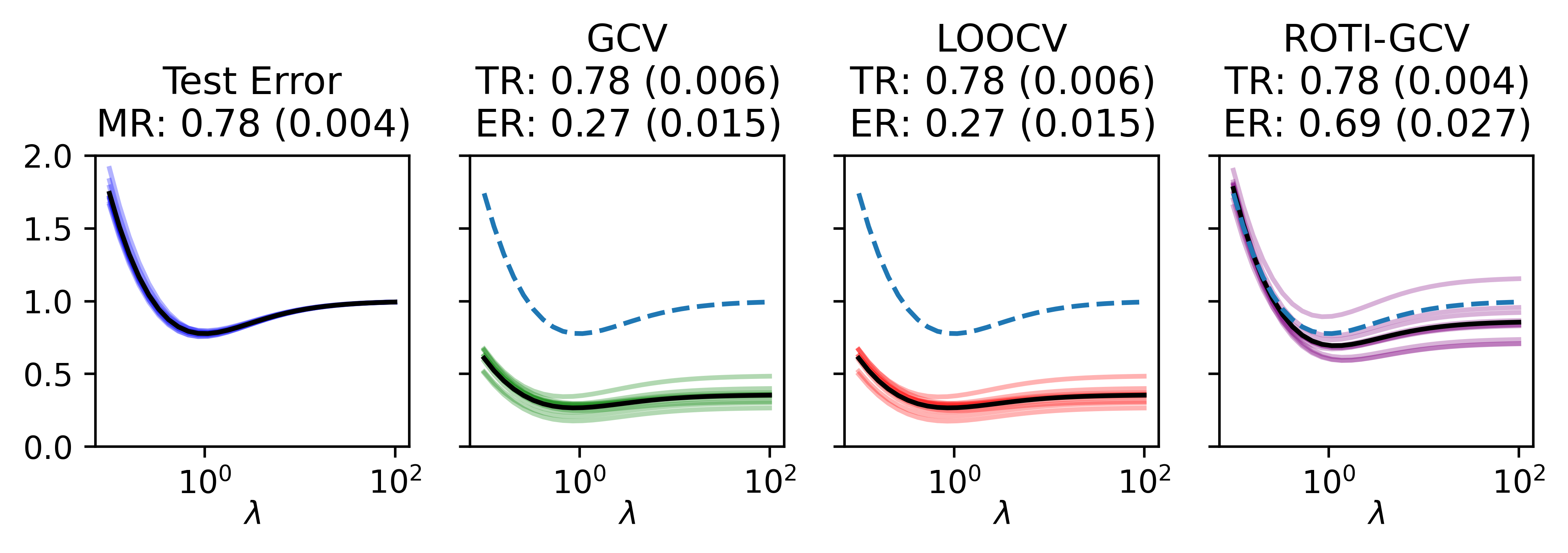}
        \caption{Equicorrelated data, $\rho = 0.8$}
        \label{fig:equi-gcv}
    \end{subfigure}
    \caption{Risk curves produced by the cross-validation methods in addition to the true risk curve. Each plot uses a different form of right-rotationally invariant data, as indicated in each caption. 
    (\subref{fig:auto-gcv}):~Autocorrelated data: rows of $\bl{X}$ are drawn according to $\bl{x}_i = \rho \bl{x}_{i - 1} + \sqrt{1 - \rho^2} \bl{z}_i$, with $\bl{z}_i$ being i.i.d. draws from $\N(0, \I_n)$. We set $\rho = 0.8$. {(\subref{fig:t-gcv}):~$t$-distributed data:} each row is drawn from a multivariate $t$ distribution with $3$ degrees of freedom. (\subref{fig:equi-gcv}):~Equicorrelated data $\bl{X} \in \R^{n \times p}$ has independent columns, but each column follows a multivariate Gaussian distribution with covariance matrix $\bs\Sigma$, where $\Sigma_{ij} = \rho$ if $i \neq j$, and $\Sigma_{ii} = 1$.
    All simulations have $n = p = 1000$ and $r^2 = \sigma^2 = 1$. The $x$-axis for every plot is $\lambda$, the regularization parameter. The colored lines (blue, green, red, purple) are one of 10 iterations. In each iteration, we compute the cross-validation metric over a range of $\lambda$ to produce the line, which reflects the estimated out-of-sample risk. The black line of each plot shows the average result for that method. The dashed blue line shows the average expected MSE curve as a benchmark.} 
    \label{fig:gcv-synth-perf}
\end{figure*}

Figure~\ref{fig:gcv-synth-perf} illustrates the performance of ROTI-GCV and compares it to LOOCV and GCV.  In these figures, the Tuned Risk (abbreviated TR) refers to the out-of-sample risk obtained when we use the value of $\lambda$ that optimizes the given cross-validation method. Minimum risk (MR), denotes the minimum of the true expected out-of-sample risk curve. Standard errors are shown in parentheses for all tuned risks. Estimated Risk (ER) refers to the value of the out-of-sample risk that would be estimated using the given cross-validation method. As expected, in the Gaussian setting all three methods perform well. However, for dependent data, such as in the autocorrelated and equicorrelated cases, we observe that the risk curves produced by GCV and LOOCV are wildly off from the truth, whereas ROTI-GCV accurately captures the true risk curve. Surprisingly, the tuned risk values from LOOCV, GCV are still decently close to that of ROTI-GCV. 
As discussed in Remark~\ref{rem:gcv-consis}, GCV (and also LOOCV) estimate a quantity with an additional factor of $\sigma^2$; thus, to make the evaluation of estimated risks fair, these two curves are plotted with this term removed.

\subsection{Experiments under Signal-PC alignment}

\begin{figure*}[t]%
    \centering%
    \begin{subfigure}[t]{0.49\textwidth}%
        \centering%
        \captionsetup{width=.9\textwidth}%
        \includegraphics[width=\textwidth]{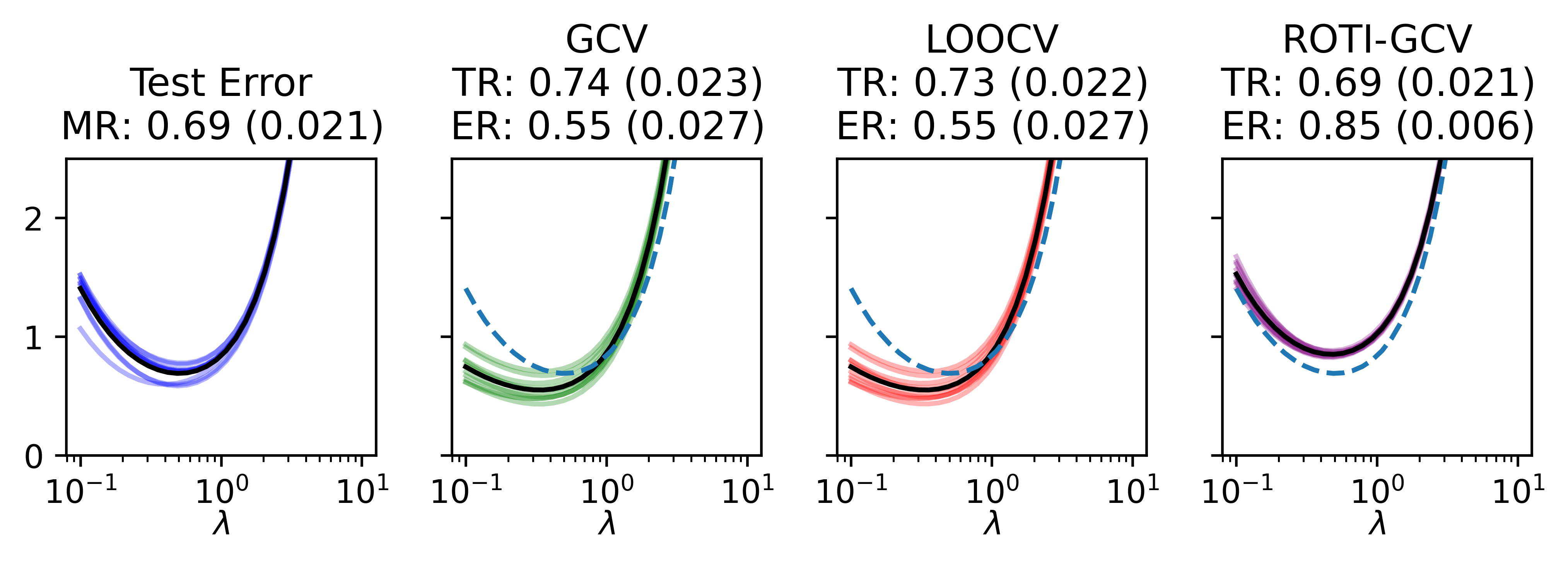}%
        \caption{Signal-PC alignment}%
        \label{fig:aligned-gcv}%
    \end{subfigure}%
    \hfill%
    \begin{subfigure}[t]{0.49\textwidth}%
        \centering%
        \captionsetup{width=.9\textwidth}%
        \includegraphics[width=\textwidth]{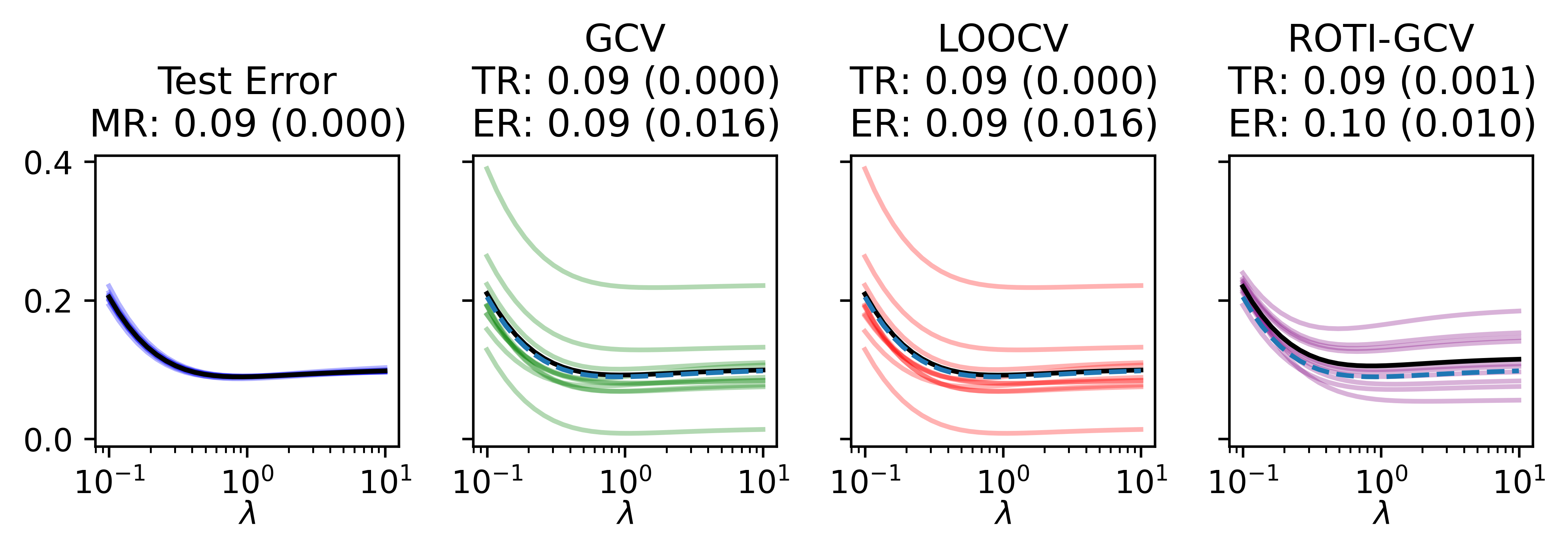}%
        \caption{Gaussian mixture}
        \label{fig:mixture}%
    \end{subfigure}%
    \hfill%
    \begin{subfigure}[t]{0.49\textwidth}
        \centering
        \captionsetup{width=.9\textwidth}%
        \includegraphics[width=\textwidth]{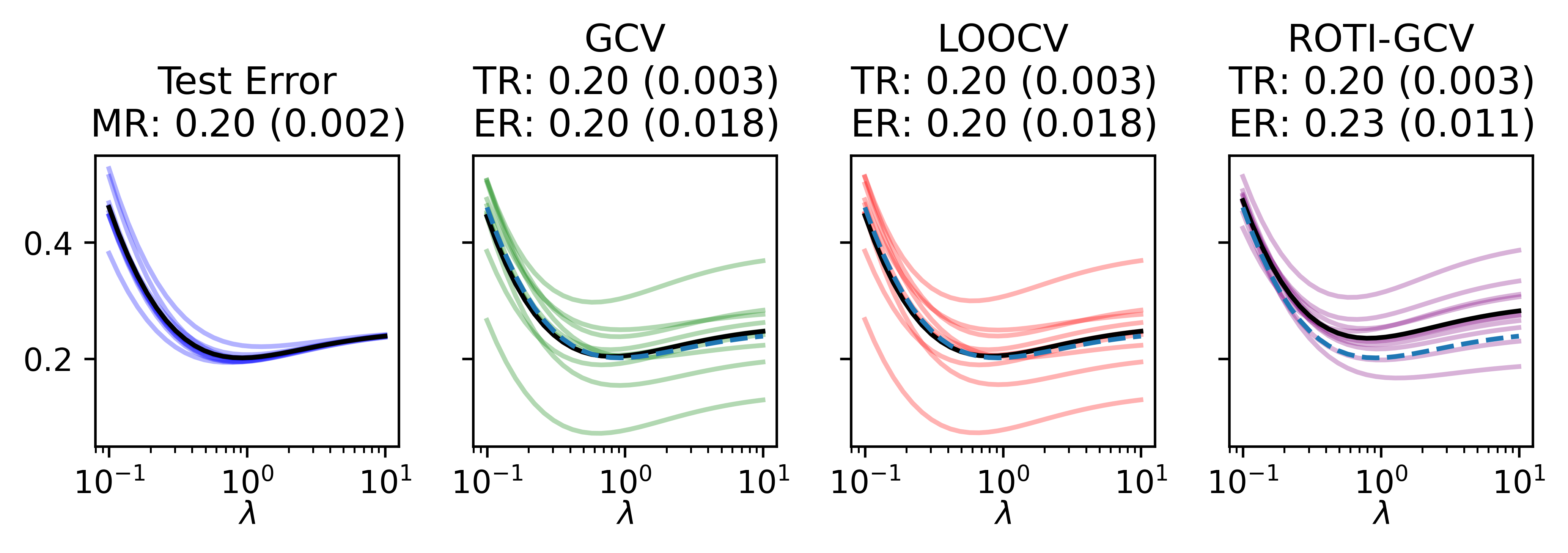}%
        \caption{Equicorrelated rows, $\rho = 0.5$}
        \label{fig:row-equi}
    \end{subfigure}
    \hfill%
    \begin{subfigure}[t]{0.49\textwidth}
        \centering
        \captionsetup{width=.9\textwidth}%
        \includegraphics[width=\textwidth]{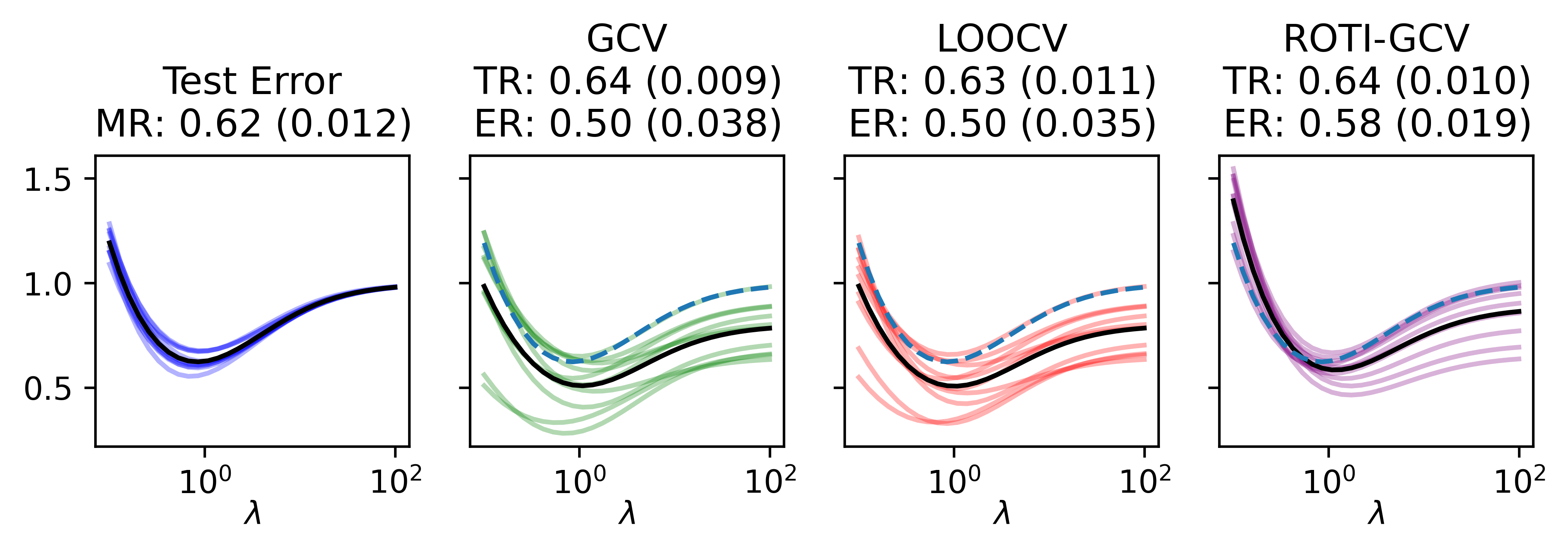}
        \caption{Speech Data without Signal-PC alignment}
        \label{fig:speech-data-result}
    \end{subfigure}
    \hfill%
    \begin{subfigure}[t]{0.49\textwidth}
        \centering
        \captionsetup{width=.9\textwidth}%
        \includegraphics[width=\linewidth]{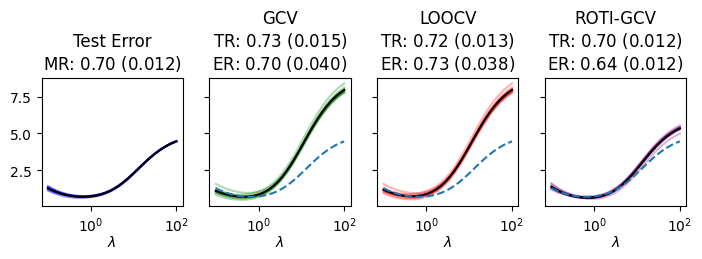}%
        \caption{Speech Data with Signal-PC alignment}
        \label{fig:speech-data-align}
    \end{subfigure}
    \hfill%
    \begin{subfigure}[t]{0.49\textwidth}
            \centering%
        \captionsetup{width=.9\textwidth}%
        \includegraphics[width=\textwidth]{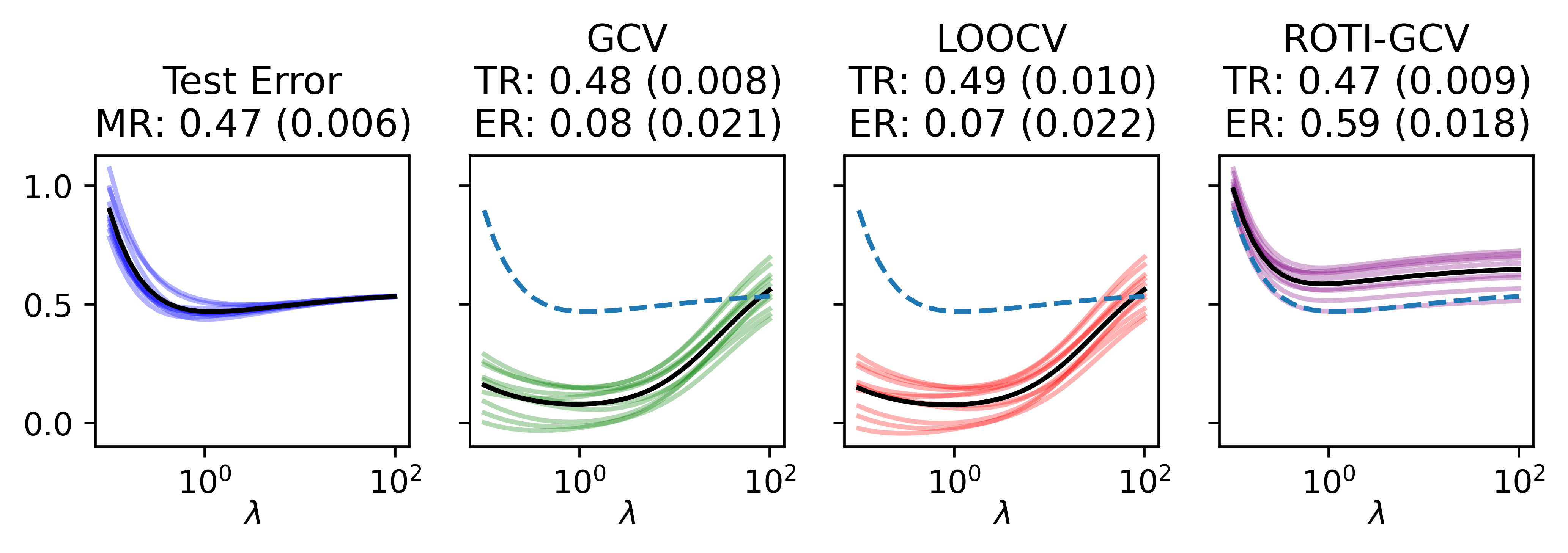}
        \caption{Residualized returns}%
        \label{fig:sp500}%
    \end{subfigure}
    \caption{As in Figure~\ref{fig:gcv-synth-perf}, curves show the risk and estimated risk as a function of $\lambda$. (\subref{fig:aligned-gcv}): $\bl{X}$ has autocorrelated rows, with $\rho = 0.8$. The top 10 eigenvectors ($\mathcal{J}_c = [10]$) are coupled, so that the top 10 eigenvectors of the test set are equal to those of $\bl{X}$. $\bs\beta = \bs\beta' + \sqrt{n}\sum_{i = 1} ^{10} \frac{i}{10} \bl{o}_i$, with $\|\bs\beta'\|^2 = \sqrt{n}$, $\sigma^2 = 1$, $n = p = 1000$. (\subref{fig:mixture}) i.i.d. rows drawn from a Gaussian mixture, i.e. $\bl{x}_i$ is drawn from $\frac{1}{2} \N(\bl{3}, \I_p) + \frac{1}{2} \N(-\bl{3}, \I_p)$. When computing ROTI-GCV, we set $\mathcal{J}_c = \{1\}$. (\subref{fig:row-equi}): i.i.d. rows drawn from an equicorrelated Gaussian, i.e. $\bl{x}_i \sim \N(0, \bs\Sigma)$, where $\Sigma_{ij} = \1(i = j) + \rho\1(i \neq j)$. We set $\rho = 0.5$. When computing ROTI-GCV, we set $\mathcal{J}_c = [10]$. \textcolor{black}{(\subref{fig:speech-data-result}): speech data with $\bs\beta$ sampled uniformly from sphere, with $r^2 = \sigma^2 = 1$, $n = p = 400$; we choose $\mathcal{J}_c = [3]$. (\subref{fig:speech-data-align}): speech data with $\bs\beta = \bs\beta' + \tfrac{\sqrt{n}}{2}\sum_{i = 1} ^5 \bl{o}_i$; again $r^2 = \sigma^2 = 1$ and $n = p = 400$. We choose $\mathcal{J}_c = [3]$ and $\mathcal{J}_a = [5]$ (discussion in Section~\ref{sec:ex-verify}).} %
    (\subref{fig:sp500}): 30 minute residualized returns sampled every 1 minute; $\bs\beta$ sampled uniformly from sphere.
    }%
    \label{fig:align-real}
\end{figure*}
We now consider settings where the signal $\bs\beta$ aligns with the eigenvectors $\bl{O}$. The relevant method here is aROTI-GCV (Algorithm \ref{alg:aroti-gcv}). In Figure \ref{fig:aligned-gcv}, we show the performance of aROTI-GCV when we set the top eigenvectors of the test set equal to those of the training set, matching the coupled eigenvectors condition (Assumption~\ref{item:coupled}). We further construct the signal such that it aligns with the top 10 eigenvectors. As the conditions match our theory, we see aROTI-GCV performs well. The setting of Figure \ref{fig:aligned-gcv} might look contrived, but we next provide examples of well-known distributions where such structures naturally arise (Figures~\ref{fig:mixture}~and~\ref{fig:row-equi}). 
In Figure~\ref{fig:mixture}, the data is drawn from a mixture of two Gaussians, one centered at $\bl{3} = [3, 3, \dots, 3]$ and the other centered at $-\bl{3}$, each with identity covariance. In this setting, the top eigenvector is very close to  $\bl{3}$ and emerges from the bulk. Furthermore, projected into the subspace $\mathrm{span}(\bl{3})^\perp$, the data is a standard Gaussian and thus right-rotationally invariant. As seen, our method works under this type of covariance structure. Shown in Figure~\ref{fig:row-equi}, we find similar results when each row of $\bl{X}$ is drawn from an equicorrelated Gaussian, i.e., $\bl{x}_i \sim \N(0, \bs\Sigma)$ where $\Sigma_{ij} = \rho^{1-\mathds{1}(i=j)}$. Both examples mirror the setting of \ref{fig:aligned-gcv} in that they are not right-rotationally invariant, as the test and train tests vary in one direction more than the others. However, we see that the modifications in aROTI-GCV enable the method to still succeed. Note that in both of these examples, the top eigenvalue diverges at the rate $O(n)$, which violates our Assumption~\ref{item:j-bounded-spec} as well as the bounded spectrum condition \cite[Assumption~3]{pmlr-v130-patil21a}, meaning neither cross-validation method is provably valid in this setting. However, we observe that ROTI-GCV still performs reasonably well.
\subsection{Experiments on semi-real data}\label{sec:semi-real-numericals}%
In our semi-real experiments, we use real data for the designs, but generate the signals and responses ourselves. Knowing the true signal enables us to calculate the population risk of any estimator, allowing us to evaluate the performance of our GCV metrics accurately. For such experiments, we must verify whether the assumptions of our theory are met. To this end, we provide a series of diagnostics to check whether the most important assumptions hold, including to select the index sets $\mathcal{J}_a$ and $\mathcal{J}_c$. We briefly list these here, and defer more details to Appendix~\ref{app:verify}, including examples of running these diagnostics for multiple datasets.

\textit{(i) regularity of singular value distributions (Assumptions~\ref{item:bound-op}, \ref{item:j-bounded-spec})}. Plot histogram of singular values to check for outliers; see Figure~\ref{fig:speech-data-spectrum}.

\textit{(ii) signal-PC alignment: relationship between $\bs\beta$ and $\{\bl{o}_i\}_{i = 1} ^p$ (Assumptions~\ref{item:indep-beta}, \ref{item:aligned})}. \cite{li2023spectrumaware} provide a hypothesis testing framework for identifying $\mathcal{J}_a$, wherein one computes a $p$-value per eigenvector to test whether or not it is aligned; the Benjamini-Hochberg procedure can then be used to control the false discovery rate. See Table~\ref{table:final}.

\textit{(iii) coupling: relationship between $\{\bl{o}_i\}_{i = 1} ^p$ and $\{\tbl{o}_j\}_{i=1} ^p$ (Assumptions~\ref{item:indep-test}, \ref{item:coupled})}. We compute overlaps $\sqrt{p}\langle \bl{o}_i, \tbl{o}_j\rangle$ for the top singular vectors. Note if {$\bl{o}_i~\indep~\tbl{o}_j$} (as under \ref{item:indep-test}), then $\sqrt{p}\langle \bl{o}_i, \tbl{o}_j \rangle \to \N(0, 1)$. If $\sqrt{p}\langle \bl{o}_i, \tbl{o}_j \rangle$ is large relative to this null distribution, $\bl{o}_i$, $\tbl{o}_j$ should be coupled. See Table~\ref{table:speech-overlaps}. We leave development of a formal hypothesis test to future work.

\subsubsection{Speech data}\label{sec:ex-verify}
The designs we use consist of speech data sourced from OpenML \citep{OpenMLSpeech}. Each row has dimension 400 and consists of an i-vector (a type of featurization for audio data, see e.g. \cite{ibrahim2018vector}) of a speech segment. In this first example, we fully illustrate all the proposed diagonstics; as such, we consider both when $\bs\beta$ is uniformly drawn on the sphere, so there is no signal-PC alignment, and when signal-PC alignment is present, to illustrate how to choose $\mathcal{J}_a$. 

We first examine the distribution of singular values in Figure~\ref{fig:speech-data-spectrum}, finding there are no outlier values. Next, we compute overlaps of test and train eigenvectors, shown in Table~\ref{table:speech-overlaps}, finding that the top 3 eigenvectors of each have strong overlaps. This motivates setting $\mathcal{J}_c = [3]$\footnote{One can further couple the next few eigenvectors (e.g. $\bl{o}_6$ with $\tbl{o}_4$ and so on), but we observe this has no impact on the results.}. The results for this setting when $\bs\beta$ is drawn uniformly from the sphere are shown in Figure~\ref{fig:speech-data-result}, where we see successfully capture the true loss curve and that the estimated risk values are superior to those given by GCV and LOOCV.

Next, in the setting with added signal-PC alignment, we instead take $\bs\beta = \bs\beta' + \tfrac{\sqrt{n}}{2}\sum_{i = 1} ^5 \bl{o}_i$, meaning the aligned set is $\mathcal{J}_a = [5]$. We use the hypothesis testing framework of \cite{li2023spectrumaware} and compute Benjamini-Hochberg-adjusted $p$-values for alignment, shown in Table~\ref{table:final}; here, we clearly identify the set of aligned eigenvectors. The resulting loss curves for this setting are then shown in Figure~\ref{fig:speech-data-align}. 
\begin{figure}[t]%
    \centering%
    \includegraphics[width=0.8\linewidth]{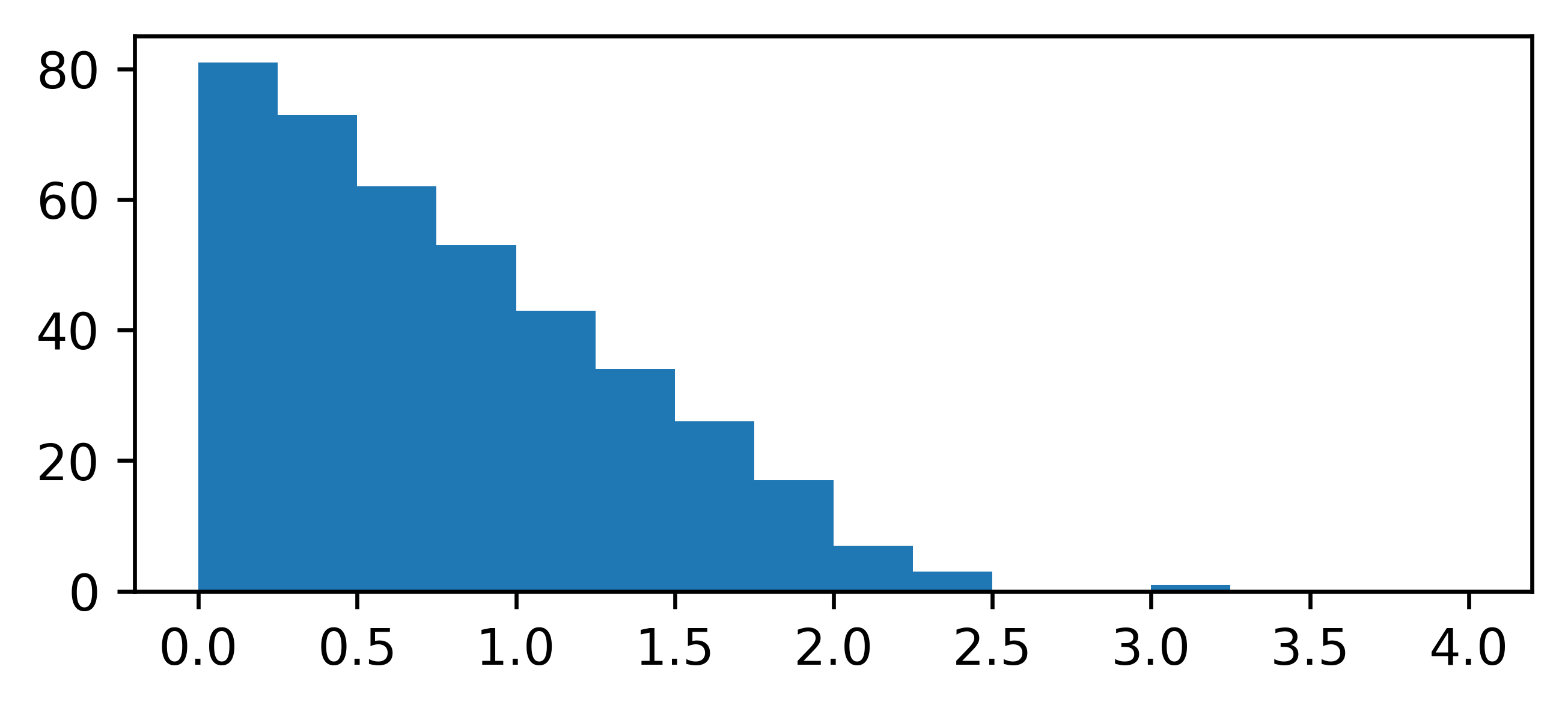}%
    \caption{Singular value distribution for speech data.}%
    \label{fig:speech-data-spectrum}%
\end{figure}%
\begin{table}%
    \caption{Overlaps for speech data setting; Cell $(i, j)$ contains $\sqrt{p}|\langle \bl{o}_i, \tbl{o}_j \rangle|$.} \label{table:speech-overlaps}
    \centering
    \resizebox{0.9\linewidth}{!}{\tiny%
    \setlength\tabcolsep{1.5pt}\renewcommand{\arraystretch}{1.5}%
    \sffamily%
    \begin{tabular}{p{1.5em}C{2.5em}C{2.5em}C{2.5em}C{2.5em}C{2.5em}C{2.5em}C{2.5em}C{2.5em}C{2.5em}C{2.5em}}%
     & $\tbl{o}_{1}$ & $\tbl{o}_{2}$ & $\tbl{o}_{3}$ & $\tbl{o}_{4}$ & $\tbl{o}_{5}$ & $\tbl{o}_{6}$ & $\tbl{o}_{7}$ & $\tbl{o}_{8}$ & $\tbl{o}_{9}$ & $\tbl{o}_{10}$ \\
    $\bl o_{1}$ & \cellcolor[HTML]{67000d} \textcolor[HTML]{f1f1f1}{17.30} & \cellcolor[HTML]{fdd0bc} \textcolor[HTML]{000000}{2.72} & \cellcolor[HTML]{ffeee7} \textcolor[HTML]{000000}{0.63} & \cellcolor[HTML]{fedecf} \textcolor[HTML]{000000}{2.01} & \cellcolor[HTML]{fee4d8} \textcolor[HTML]{000000}{1.55} & \cellcolor[HTML]{feeae0} \textcolor[HTML]{000000}{1.02} & \cellcolor[HTML]{fff2eb} \textcolor[HTML]{000000}{0.30} & \cellcolor[HTML]{feeae1} \textcolor[HTML]{000000}{0.96} & \cellcolor[HTML]{fee2d5} \textcolor[HTML]{000000}{1.70} & \cellcolor[HTML]{fff2ec} \textcolor[HTML]{000000}{0.27} \\
    $\bl o_{2}$ & \cellcolor[HTML]{ffece3} \textcolor[HTML]{000000}{0.83} & \cellcolor[HTML]{67000d} \textcolor[HTML]{f1f1f1}{15.24} & \cellcolor[HTML]{ffeee6} \textcolor[HTML]{000000}{0.68} & \cellcolor[HTML]{fee3d7} \textcolor[HTML]{000000}{1.63} & \cellcolor[HTML]{fee7db} \textcolor[HTML]{000000}{1.34} & \cellcolor[HTML]{fee8dd} \textcolor[HTML]{000000}{1.22} & \cellcolor[HTML]{ffebe2} \textcolor[HTML]{000000}{0.89} & \cellcolor[HTML]{fee5d9} \textcolor[HTML]{000000}{1.42} & \cellcolor[HTML]{fff2ec} \textcolor[HTML]{000000}{0.27} & \cellcolor[HTML]{fedaca} \textcolor[HTML]{000000}{2.19} \\
    $\bl o_{3}$ & \cellcolor[HTML]{fee9df} \textcolor[HTML]{000000}{1.08} & \cellcolor[HTML]{fee0d2} \textcolor[HTML]{000000}{1.93} & \cellcolor[HTML]{d11e1f} \textcolor[HTML]{f1f1f1}{10.93} & \cellcolor[HTML]{fee3d7} \textcolor[HTML]{000000}{1.60} & \cellcolor[HTML]{ffeee7} \textcolor[HTML]{000000}{0.64} & \cellcolor[HTML]{fedecf} \textcolor[HTML]{000000}{2.02} & \cellcolor[HTML]{fca78b} \textcolor[HTML]{000000}{4.63} & \cellcolor[HTML]{fff2eb} \textcolor[HTML]{000000}{0.35} & \cellcolor[HTML]{fcb398} \textcolor[HTML]{000000}{4.14} & \cellcolor[HTML]{fdd0bc} \textcolor[HTML]{000000}{2.75} \\
    $\bl o_{4}$ & \cellcolor[HTML]{fff2ec} \textcolor[HTML]{000000}{0.29} & \cellcolor[HTML]{fdcebb} \textcolor[HTML]{000000}{2.78} & \cellcolor[HTML]{fdccb8} \textcolor[HTML]{000000}{2.92} & \cellcolor[HTML]{fff3ed} \textcolor[HTML]{000000}{0.22} & \cellcolor[HTML]{fcb79c} \textcolor[HTML]{000000}{3.98} & \cellcolor[HTML]{fb7c5c} \textcolor[HTML]{f1f1f1}{6.63} & \cellcolor[HTML]{fdcbb6} \textcolor[HTML]{000000}{2.98} & \cellcolor[HTML]{fee3d6} \textcolor[HTML]{000000}{1.65} & \cellcolor[HTML]{ffefe8} \textcolor[HTML]{000000}{0.54} & \cellcolor[HTML]{fee4d8} \textcolor[HTML]{000000}{1.56} \\
    $\bl o_{5}$ & \cellcolor[HTML]{fee7dc} \textcolor[HTML]{000000}{1.27} & \cellcolor[HTML]{fdd0bc} \textcolor[HTML]{000000}{2.70} & \cellcolor[HTML]{fcb69b} \textcolor[HTML]{000000}{4.02} & \cellcolor[HTML]{fedbcc} \textcolor[HTML]{000000}{2.14} & \cellcolor[HTML]{fc9d7f} \textcolor[HTML]{000000}{5.11} & \cellcolor[HTML]{fee7db} \textcolor[HTML]{000000}{1.29} & \cellcolor[HTML]{fcae92} \textcolor[HTML]{000000}{4.34} & \cellcolor[HTML]{fca689} \textcolor[HTML]{000000}{4.72} & \cellcolor[HTML]{fdcbb6} \textcolor[HTML]{000000}{2.95} & \cellcolor[HTML]{fed9c9} \textcolor[HTML]{000000}{2.24} \\
    $\bl o_{6}$ & \cellcolor[HTML]{fee8dd} \textcolor[HTML]{000000}{1.19} & \cellcolor[HTML]{ffefe8} \textcolor[HTML]{000000}{0.56} & \cellcolor[HTML]{fca486} \textcolor[HTML]{000000}{4.82} & \cellcolor[HTML]{fc8b6b} \textcolor[HTML]{f1f1f1}{5.95} & \cellcolor[HTML]{fff1ea} \textcolor[HTML]{000000}{0.37} & \cellcolor[HTML]{fcb99f} \textcolor[HTML]{000000}{3.83} & \cellcolor[HTML]{fee6da} \textcolor[HTML]{000000}{1.37} & \cellcolor[HTML]{ffece3} \textcolor[HTML]{000000}{0.87} & \cellcolor[HTML]{fca183} \textcolor[HTML]{000000}{4.98} & \cellcolor[HTML]{feeae0} \textcolor[HTML]{000000}{1.04} \\
    $\bl o_{7}$ & \cellcolor[HTML]{fedaca} \textcolor[HTML]{000000}{2.21} & \cellcolor[HTML]{ffeee6} \textcolor[HTML]{000000}{0.70} & \cellcolor[HTML]{fedbcc} \textcolor[HTML]{000000}{2.14} & \cellcolor[HTML]{fdd1be} \textcolor[HTML]{000000}{2.65} & \cellcolor[HTML]{fca486} \textcolor[HTML]{000000}{4.84} & \cellcolor[HTML]{fee1d4} \textcolor[HTML]{000000}{1.77} & \cellcolor[HTML]{ffefe8} \textcolor[HTML]{000000}{0.54} & \cellcolor[HTML]{feeae0} \textcolor[HTML]{000000}{1.00} & \cellcolor[HTML]{fee3d6} \textcolor[HTML]{000000}{1.65} & \cellcolor[HTML]{fdd3c1} \textcolor[HTML]{000000}{2.53} \\
    $\bl o_{8}$ & \cellcolor[HTML]{fff0e9} \textcolor[HTML]{000000}{0.46} & \cellcolor[HTML]{fff4ee} \textcolor[HTML]{000000}{0.15} & \cellcolor[HTML]{fdcbb6} \textcolor[HTML]{000000}{2.96} & \cellcolor[HTML]{fff4ee} \textcolor[HTML]{000000}{0.16} & \cellcolor[HTML]{fee3d6} \textcolor[HTML]{000000}{1.66} & \cellcolor[HTML]{fcab8f} \textcolor[HTML]{000000}{4.49} & \cellcolor[HTML]{fff3ed} \textcolor[HTML]{000000}{0.22} & \cellcolor[HTML]{fff1ea} \textcolor[HTML]{000000}{0.40} & \cellcolor[HTML]{fee1d3} \textcolor[HTML]{000000}{1.82} & \cellcolor[HTML]{fff4ef} \textcolor[HTML]{000000}{0.10} \\
    $\bl o_{9}$ & \cellcolor[HTML]{fff2eb} \textcolor[HTML]{000000}{0.30} & \cellcolor[HTML]{fff1ea} \textcolor[HTML]{000000}{0.40} & \cellcolor[HTML]{ffede5} \textcolor[HTML]{000000}{0.72} & \cellcolor[HTML]{fff5f0} \textcolor[HTML]{000000}{0.01} & \cellcolor[HTML]{fee3d7} \textcolor[HTML]{000000}{1.63} & \cellcolor[HTML]{fee6da} \textcolor[HTML]{000000}{1.39} & \cellcolor[HTML]{fcb99f} \textcolor[HTML]{000000}{3.86} & \cellcolor[HTML]{fcaa8d} \textcolor[HTML]{000000}{4.53} & \cellcolor[HTML]{fee8de} \textcolor[HTML]{000000}{1.12} & \cellcolor[HTML]{fdd3c1} \textcolor[HTML]{000000}{2.57} \\
    $\bl o_{10}$ & \cellcolor[HTML]{fee8de} \textcolor[HTML]{000000}{1.14} & \cellcolor[HTML]{fdd4c2} \textcolor[HTML]{000000}{2.51} & \cellcolor[HTML]{ffebe2} \textcolor[HTML]{000000}{0.90} & \cellcolor[HTML]{ffece4} \textcolor[HTML]{000000}{0.77} & \cellcolor[HTML]{fff1ea} \textcolor[HTML]{000000}{0.37} & \cellcolor[HTML]{ffece4} \textcolor[HTML]{000000}{0.79} & \cellcolor[HTML]{fdd5c4} \textcolor[HTML]{000000}{2.44} & \cellcolor[HTML]{fee1d3} \textcolor[HTML]{000000}{1.83} & \cellcolor[HTML]{fee1d4} \textcolor[HTML]{000000}{1.81} & \cellcolor[HTML]{fed8c7} \textcolor[HTML]{000000}{2.34} \\
    \end{tabular}
    }
\end{table}%
\begin{table}[t]%
    \caption{Benjamini-Hochberg adjusted $p$-values for alignment.}%
    \centering%
    \resizebox{0.7\linewidth}{!}{%
    \begin{tabular}{rrrrrrrrrrr}%
    & $\bl{o}_1$ & $\bl{o}_2$ & $\bl{o}_3$ & $\bl{o}_4$ & $\bl{o}_5$ \\%
    $p$ & 0.000 & 0.000 & 0.000 & 0.000 & 0.000 \\%
    & $\bl{o}_6$ & $\bl{o}_7$ & $\bl{o}_8$ & $\bl{o}_9$ & $\bl{o}_{10}$ \\%
    $p$ & 0.934 & 0.588 & 0.651 & 0.913 & 0.395 \\%
    \end{tabular}%
    }
    \label{table:final}%
\end{table}%
\subsubsection{\textcolor{black}{Residualized returns}}

\color{black}
We next consider designs consisting of residualized returns from 493 stocks retrieved from Polygon API\footnote{We remove the top 8 right singular vectors from data matrix $\bl{X}$. This is analogous to \textit{residualization} of returns, a common practice in financial modeling.  We compute the factors to delete by taking the top PCs of the covariance matrix computed using both test and train data; PCA is crude/weak as a factor model and we have only 8 factors; hence it is more comparable to compute PCs looking at both train and test data. See Appendix~\ref{app:semi-real} for more details on residualization.} \citep{Polygon}. There is one row per minute, and each contains the residualized return for each stock over the last 30 minutes. Since they share 29 minutes of returns, consecutive rows are heavily correlated. The regression corresponds to the known practice of replicating the residualized returns of an unknown asset $\bl{y}$ using a portfolio of stocks $\hat{\bs\beta}$ (with returns $\bl{X}\hat{\bs\beta}$). \textcolor{black}{The diagnostic plots in this case are tame, and we require no coupling; these are deferred to Appendix~\ref{app:verify}.} The loss curves produced by ROTI-GCV better capture the true loss curve, and that the estimated risk values (ER) provide superior estimates of the true out-of-sample risk compared to LOOCV and GCV.
\color{black}
\section{DISCUSSION AND CONCLUSION}
\textbf{Takeaways:} We give a precise analysis of ridge regression for right-rotationally invariant designs. These designs depart crucially from the existing framework of i.i.d. samples from distributions with light enough tails, instead allowing for structured dependence and heavy tails. We then characterize GCV in this setting, finding that it is inconsistent, and introduce ROTI-GCV, a consistent alternative; various synthetic and semi-real experiments illustrate the accuracy of our method.
As a byproduct of our analysis, we produce new estimators for signal-to-noise ratio ($r^2$) and noise variance ($\sigma^2$) robust to these deviations from the i.i.d. assumptions. At the core, our method relies on noticing that the relation given by Lemma~\ref{lem:gcv-as} yields a different estimating equation for every choice of $\lambda$; using this to generate a series of estimating equations, we obtain stable estimators of $r^2, \sigma^2$. This core idea applies whenever an analogue of Lemma~\ref{lem:gcv-as} can be derived, giving it reach beyond ridge regression. As future steps, extending our method to other settings such as generalized linear models, kernel ridge regression, regularized random features regression would be important directions. Furthermore, recent advances in \cite{dudeja2023spectral,wang2022universality} show that right-rotationally invariant distributions have a broad universality class -- in fact as long as the singular vectors of a design matrix satisfy certain generic structure, properties of learning algorithms under such settings are well captured by results derived under the right-rotationally invariant assumption. In particular, the universality of the in-sample risk of ridge regression is already established in these works, and we believe such universality also holds for the out-of-sample risk; thus, our analysis should extend to the entire universality class of these designs, which is remarkably broad. This promises that our approach can have wide reach, allowing reliable hyperparameter tuning for a broad class of covariate distributions. 

\textbf{Limitations:} \textcolor{black}{The well-specification assumption (\ref{item:scaling-matrix}) is crucial and seems more difficult to relax for our designs than in the i.i.d.~case due to the global dependence structure. Right-rotationally invariant designs have difficulty modeling distributions with feature covariances that are much more complicated than those covered in Section~\ref{sec:numericals}, especially when the spectrum is not dominated by eigenvalues. Finding ways to account for different types of feature covariance in data using these designs would naturally increase the broader applicability of our method, such as by modeling the design as $\bl{X} = \bl{Q}^\top \bl{D} \bl{O} \bs\Sigma^{1/2}$ and extending our analysis. Finally, as mentioned in Section~\ref{sec:intro}, our results focus on ridge regression, though we believe similar results for other penalties can be obtained.}

\FloatBarrier

\newpage

\bibliography{refs}

\begin{thebibliography}{}

\bibitem[Adamczak, 2015]{adamczak2014note}
Adamczak, R. (2015).
\newblock {A note on the Hanson-Wright inequality for random vectors with
  dependencies}.
\newblock {\em Electronic Communications in Probability}, 20(none):1 -- 13.

\bibitem[Atanasov et~al., 2024a]{atanasov2024riskcrossvalidationridge}
Atanasov, A., Zavatone-Veth, J.~A., and Pehlevan, C. (2024a).
\newblock Risk and cross validation in ridge regression with correlated
  samples.

\bibitem[Atanasov et~al., 2024b]{atanasov2024scaling}
Atanasov, A.~B., Zavatone-Veth, J.~A., and Pehlevan, C. (2024b).
\newblock Scaling and renormalization in high-dimensional regression.

\bibitem[Bigot et~al., 2024]{bigot:hal-04559313}
Bigot, J., Dabo, I.-M., and Male, C. (2024).
\newblock {High-dimensional analysis of ridge regression for non-identically
  distributed data with a variance profile}.
\newblock working paper or preprint.

\bibitem[Chazottes, 2022]{chazottesnote}
Chazottes, J.-R. (2022).
\newblock Notes on means, medians and gaussian tails.
\newblock
  \url{https://hal.science/hal-03636138v1/file/Gaussian-concentration-around-the-mean-and-the-median.pdf}.

\bibitem[Craven and Wahba, 1978]{craven1978smoothing}
Craven, P. and Wahba, G. (1978).
\newblock Smoothing noisy data with spline functions: estimating the correct
  degree of smoothing by the method of generalized cross-validation.
\newblock {\em Numerische mathematik}, 31(4):377--403.

\bibitem[Dobriban and Wager, 2018]{dobriban2015highdimensional}
Dobriban, E. and Wager, S. (2018).
\newblock High-dimensional asymptotics of prediction: ridge regression and
  classification.
\newblock {\em Ann. Statist.}, 46(1):247--279.

\bibitem[Dudeja et~al., 2023]{dudeja2023spectral}
Dudeja, R., Sen, S., and Lu, Y.~M. (2023).
\newblock Spectral universality of regularized linear regression with nearly
  deterministic sensing matrices.
\newblock {\em arXiv preprint arXiv:2208.02753}.

\bibitem[Fan, 2022]{10.1214/21-AOS2101}
Fan, Z. (2022).
\newblock {Approximate Message Passing algorithms for rotationally invariant
  matrices}.
\newblock {\em The Annals of Statistics}, 50(1):197 -- 224.

\bibitem[Fan and Wu, 2021]{fan2021replica}
Fan, Z. and Wu, Y. (2021).
\newblock The replica-symmetric free energy for ising spin glasses with
  orthogonally invariant couplings.
\newblock {\em arXiv preprint arXiv:2105.02797}.

\bibitem[Gelman et~al., 2004]{gelmanbda04}
Gelman, A., Carlin, J.~B., Stern, H.~S., and Rubin, D.~B. (2004).
\newblock {\em Bayesian Data Analysis}.
\newblock Chapman and Hall/CRC, 2nd ed. edition.

\bibitem[Gerbelot et~al., 2020]{gerbelot2020asymptotic}
Gerbelot, C., Abbara, A., and Krzakala, F. (2020).
\newblock Asymptotic errors for high-dimensional convex penalized linear
  regression beyond gaussian matrices.
\newblock In Abernethy, J. and Agarwal, S., editors, {\em Proceedings of Thirty
  Third Conference on Learning Theory}, volume 125 of {\em Proceedings of
  Machine Learning Research}, pages 1682--1713. PMLR.

\bibitem[Golub et~al., 1979]{golubgcv}
Golub, G.~H., Heath, M., and Wahba, G. (1979).
\newblock Generalized cross-validation as a method for choosing a good ridge
  parameter.
\newblock {\em Technometrics}, 21(2):215--223.

\bibitem[Grobys, 2021]{GROBYS2021101891}
Grobys, K. (2021).
\newblock What do we know about the second moment of financial markets?
\newblock {\em International Review of Financial Analysis}, 78:101891.

\bibitem[Gu and Ma, 2005]{gu2005optimal}
Gu, C. and Ma, P. (2005).
\newblock {Optimal smoothing in nonparametric mixed-effect models}.
\newblock {\em The Annals of Statistics}, 33(3):1357 -- 1379.

\bibitem[Hanin and Nica, 2020]{hanin2020products}
Hanin, B. and Nica, M. (2020).
\newblock Products of many large random matrices and gradients in deep neural
  networks.
\newblock {\em Communications in Mathematical Physics}, 376(1):287--322.

\bibitem[Hanin and Paouris, 2021]{hanin2021non}
Hanin, B. and Paouris, G. (2021).
\newblock Non-asymptotic results for singular values of gaussian matrix
  products.
\newblock {\em Geometric and Functional Analysis}, 31(2):268--324.

\bibitem[Hastie et~al., 2022]{hastie2022surprises}
Hastie, T., Montanari, A., Rosset, S., and Tibshirani, R.~J. (2022).
\newblock Surprises in high-dimensional ridgeless least squares interpolation.
\newblock {\em The Annals of Statistics}, 50(2):949--986.

\bibitem[Hubert and Verboven, 2003]{hubert2003robust}
Hubert, M. and Verboven, S. (2003).
\newblock A robust pcr method for high-dimensional regressors.
\newblock {\em Journal of Chemometrics: A Journal of the Chemometrics Society},
  17(8-9):438--452.

\bibitem[Ibrahim and Ramli, 2018]{ibrahim2018vector}
Ibrahim, N.~S. and Ramli, D.~A. (2018).
\newblock I-vector extraction for speaker recognition based on dimensionality
  reduction.
\newblock {\em Procedia Computer Science}, 126:1534--1540.

\bibitem[Jolliffe, 1982]{jolliffe1982note}
Jolliffe, I.~T. (1982).
\newblock A note on the use of principal components in regression.
\newblock {\em Journal of the Royal Statistical Society Series C: Applied
  Statistics}, 31(3):300--303.

\bibitem[Le, 2017]{OpenMLSpeech}
Le, M.-A. (2017).
\newblock Openml, dataset id: 40910.
\newblock \url{https://www.openml.org/search?type=data&status=active&id=40910}.
\newblock Accessed: 2024-10-10. This work is licensed under the CC-BY License.

\bibitem[Li, 1986]{li1986asymptotic}
Li, K.-C. (1986).
\newblock Asymptotic optimality of cl and generalized cross-validation in ridge
  regression with application to spline smoothing.
\newblock {\em The Annals of Statistics}, pages 1101--1112.

\bibitem[Li et~al., 2023]{li2023random}
Li, Y., Fan, Z., Sen, S., and Wu, Y. (2023).
\newblock Random linear estimation with rotationally-invariant designs:
  Asymptotics at high temperature.
\newblock {\em IEEE Transactions on Information Theory}, 70(30):2118--2154.

\bibitem[Li and Sur, 2023]{li2023spectrumaware}
Li, Y. and Sur, P. (2023).
\newblock Spectrum-aware adjustment: A new debiasing framework with
  applications to principal component regression.
\newblock {\em arXiv preprint arXiv:2309.07810}.

\bibitem[Liang and Sur, 2022]{liang2022precise}
Liang, T. and Sur, P. (2022).
\newblock A precise high-dimensional asymptotic theory for boosting and
  minimum-l1-norm interpolated classifiers.
\newblock {\em The Annals of Statistics}, 50(3):1669--1695.

\bibitem[Liu et~al., 2022]{liu2022memory}
Liu, L., Huang, S., and Kurkoski, B.~M. (2022).
\newblock Memory amp.
\newblock {\em IEEE Transactions on Information Theory}, 68(12):8015--8039.

\bibitem[Ma and Ping, 2017]{ma2017orthogonal}
Ma, J. and Ping, L. (2017).
\newblock Orthogonal amp.
\newblock {\em IEEE Access}, 5:2020--2033.

\bibitem[Mardia et~al., 2024]{mardia2024multivariate}
Mardia, K.~V., Kent, J.~T., and Taylor, C.~C. (2024).
\newblock {\em Multivariate analysis}, volume~88.
\newblock John Wiley \& Sons.

\bibitem[Meckes, 2019]{Meckes_2019}
Meckes, E.~S. (2019).
\newblock {\em The Random Matrix Theory of the Classical Compact Groups}.
\newblock Cambridge Tracts in Mathematics. Cambridge University Press.

\bibitem[Patil et~al., 2021]{pmlr-v130-patil21a}
Patil, P., Wei, Y., Rinaldo, A., and Tibshirani, R. (2021).
\newblock Uniform consistency of cross-validation estimators for
  high-dimensional ridge regression.
\newblock In {\em Proceedings of The 24th International Conference on
  Artificial Intelligence and Statistics}, volume 130 of {\em PMLR}, pages
  3178--3186. PMLR.

\bibitem[Polygon, 2024]{Polygon}
Polygon (2024).
\newblock Polygon api.
\newblock \url{https://polygon.io/docs/stocks/getting-started}.
\newblock Accessed: 2024-03-30.

\bibitem[Rahnama~Rad and Maleki, 2018]{rad2020scalable}
Rahnama~Rad, K. and Maleki, A. (2018).
\newblock A scalable estimate of the extra-sample prediction error via
  approximate leave-one-out.
\newblock {\em Journal of the Royal Statistical Society: Series B (Statistical
  Methodology)}, 82.

\bibitem[Rangan et~al., 2019]{rangan2019vector}
Rangan, S., Schniter, P., and Fletcher, A.~K. (2019).
\newblock Vector approximate message passing.
\newblock {\em IEEE Transactions on Information Theory}, 65(10):6664--6684.

\bibitem[Schechtman, 2014]{schechtmannote}
Schechtman, G. (2014).
\newblock Concentration, results and applications.
\newblock
  \url{https://www.weizmann.ac.il/math/gideon/sites/math.gideon/files/uploads/concentrationNov19_0.pdf}.

\bibitem[Silverstein, 1995]{SILVERSTEIN1995331}
Silverstein, J. (1995).
\newblock Strong convergence of the empirical distribution of eigenvalues of
  large dimensional random matrices.
\newblock {\em Journal of Multivariate Analysis}, 55(2):331--339.

\bibitem[Song et~al., 2024]{song2024hede}
Song, Y., Lin, X., and Sur, P. (2024).
\newblock Hede: Heritability estimation in high dimensions by ensembling
  debiased estimators.
\newblock {\em arXiv preprint arXiv:2406.11184}.

\bibitem[Sur and Cand{\`e}s, 2019]{sur2019modern}
Sur, P. and Cand{\`e}s, E.~J. (2019).
\newblock A modern maximum-likelihood theory for high-dimensional logistic
  regression.
\newblock {\em Proceedings of the National Academy of Sciences},
  116(29):14516--14525.

\bibitem[Sur et~al., 2019]{sur2019likelihood}
Sur, P., Chen, Y., and Cand{\`e}s, E.~J. (2019).
\newblock The likelihood ratio test in high-dimensional logistic regression is
  asymptotically a rescaled chi-square.
\newblock {\em Probability theory and related fields}, 175:487--558.

\bibitem[Tagliazucchi et~al., 2013]{Tagliazucchi2013}
Tagliazucchi, E., von Wegner, F., Morzelewski, A., Brodbeck, V., Jahnke, K.,
  and Laufs, H. (2013).
\newblock Breakdown of long-range temporal dependence in default mode and
  attention networks during deep sleep.
\newblock {\em Proceedings of the National Academy of Sciences},
  110(38):15419–15424.

\bibitem[Takeda et~al., 2006]{takeda2006analysis}
Takeda, K., Uda, S., and Kabashima, Y. (2006).
\newblock Analysis of cdma systems that are characterized by eigenvalue
  spectrum.
\newblock {\em Europhysics Letters}, 76(6):1193.

\bibitem[Takeuchi, 2019]{takeuchi2019rigorous}
Takeuchi, K. (2019).
\newblock Rigorous dynamics of expectation-propagation-based signal recovery
  from unitarily invariant measurements.
\newblock {\em IEEE Transactions on Information Theory}, 66(1):368--386.

\bibitem[{Wang} et~al., 2018]{wang2018approximate}
{Wang}, S., {Zhou}, W., {Maleki}, A., {Lu}, H., and {Mirrokni}, V. (2018).
\newblock {Approximate Leave-One-Out for High-Dimensional Non-Differentiable
  Learning Problems}.
\newblock {\em arXiv e-prints}, page arXiv:1810.02716.

\bibitem[Wang et~al., 2023]{wang2022universality}
Wang, T., Zhong, X., and Fan, Z. (2023).
\newblock Universality of approximate message passing algorithms and tensor
  networks.
\newblock {\em arXiv preprint arXiv:2206.13037}.

\bibitem[Wu and Xu, 2020]{wu2020optimal}
Wu, D. and Xu, J. (2020).
\newblock On the optimal weighted $\ell_2$ regularization in overparameterized
  linear regression.

\bibitem[Xu et~al., 2018]{xu2018optimal}
Xu, G., Shang, Z., and Cheng, G. (2018).
\newblock Optimal tuning for divide-and-conquer kernel ridge regression with
  massive data.
\newblock In {\em International Conference on Machine Learning}, pages
  5483--5491. PMLR.

\bibitem[Xu et~al., 2019]{xu2019distributed}
Xu, G., Shang, Z., and Cheng, G. (2019).
\newblock Distributed generalized cross-validation for divide-and-conquer
  kernel ridge regression and its asymptotic optimality.
\newblock {\em Journal of computational and graphical statistics},
  28(4):891--908.

\bibitem[Xu et~al., 2021]{xu2021consistent}
Xu, J., Maleki, A., Rad, K.~R., and Hsu, D. (2021).
\newblock Consistent risk estimation in moderately high-dimensional linear
  regression.
\newblock {\em IEEE Transactions on Information Theory}, 67(9):5997--6030.

\bibitem[Xu et~al., 2022]{xu2022approximatemessagepassingmultilayer}
Xu, Y., Hou, T., Liang, S., and Mondelli, M. (2022).
\newblock Approximate message passing for multi-layer estimation in
  rotationally invariant models.

\bibitem[Zhang et~al., 2015]{zhang2015divide}
Zhang, Y., Duchi, J., and Wainwright, M. (2015).
\newblock Divide and conquer kernel ridge regression: A distributed algorithm
  with minimax optimal rates.
\newblock {\em The Journal of Machine Learning Research}, 16(1):3299--3340.

\bibitem[Zscheischler et~al., 2021]{esd-12-1-2021}
Zscheischler, J., Naveau, P., Martius, O., Engelke, S., and Raible, C.~C.
  (2021).
\newblock Evaluating the dependence structure of compound precipitation and
  wind speed extremes.
\newblock {\em Earth System Dynamics}, 12(1):1--16.

\end{thebibliography}

\onecolumn
\aistatstitle{Appendix}
\newpage
\section{Deferred Proofs}

\subsection{Proof of Lemma~\ref{lem:riri-cov}}
\begin{lemma}\label{lem:riri-cov}
    Let $\bl{X} = \bl{Q}^\top \bl{DO}$ be right-rotationally invariant. Then $\E[\bl{X}^\top \bl{X}] = \Tr(\E[\DtD])/p \cdot \I_p$. 
\end{lemma}

\begin{proof}
    Let $\bl{O}$ have rows $\bl{o}_i$. Then
    \begin{align*}
        \E[\XtX] &= \E[\bl{O}^\top \DtD \bl{O}] = \sum_{i = 1} ^{n \wedge p} \E[ D_{ii}^2 \bl{o}_i \bl{o}_i^\top] = \sum_{i = 1} ^{n \wedge p} \E[ D_{ii}^2 ]\E[\bl{o}_i \bl{o}_i^\top] \\
        &\stackrel{(*)}{=} \frac{1}{p} \I_p \sum_{i = 1} ^{n \wedge p} \E[ D_{ii}^2 ] = \Tr(\E[\DtD])/p \cdot \I_p
    \end{align*}
    where $(*)$ follows from $\I_p = \E[\bl{O}^\top \bl{O}] = \sum_{i = 1} ^p \E[\bl{o}_i\bl{o}_i^\top]$ and the exchangeability of the $\bl{o}_i$ (which is derived from the fact that $\mathbb{O}(p)$ contains the permutation matrices). 
\end{proof}

\subsection{Useful Results}

\begin{lemma}[Hanson-Wright for Spherical Vectors]\label{lem:hw}
    Let $\bl{v} \in \R^{n}$ be a random vector distributed uniformly on $S^{n - 1}$ and let $\bl{A} \in \R^{n \times n}$ be a fixed matrix. Then
    \begin{align*}
        \P(\left|\bl{v}^\top \bl{A} \bl v - \E[\bl{v}^\top \bl{A} \bl{v}]\right| \geq t) \leq 2\exp\left(-C \min\left(\frac{n^2t^2}{2K^4\|\bl{A}\|_F^2}, \frac{nt}{K^2\|\bl{A}\|}\right)\right).
    \end{align*}
    for some absolute constants $C, K$. 
\end{lemma}
\begin{proof}
    Isoperimetric inequalities imply that vectors uniformly distributed on $S^{n - 1}$ have sub-Gaussian tails around their median (see \cite{schechtmannote}). 
    Explicitly, it is stated that for $\bl{v} \sim \mathrm{Unif}(S^{n - 1})$, one has that for any $1$-Lispchitz function $f: \R^n \to \R$, one has
    \begin{equation}
        \P(|f(\bl{v}) - M| \geq t) \leq 2e^{- nt^2 / 2},
    \end{equation}
    where $M$ denotes the median of $f(X)$ (i.e. $\min(\P(f(\bl{v}) \geq M), \P(f(\bl{v}) \leq M)) \geq \frac{1}{2}$). We now use \cite[Theorem~1.2]{chazottesnote}, which shows that having sub-Gaussian tails with respect to the median is equivalent to having sub-Gaussian tails with respect to the mean, with different constants. Thus, it holds that for any $1$-Lipschitz function $f$, one has
    \begin{equation}
        \P(|f(\bl{v}) - \E f(\bl{v})| \geq t) \leq 4e^{- nt^2 / 32}.
    \end{equation}
    Note that this bound is vacuous whenever $4e^{-nt^2/32} > 1$, and hence we can rewrite this as
    \begin{equation}
        \P(|f(\bl{v}) - \E f(\bl{v})| \geq t) \leq \min(4e^{- nt^2 / 32}, 1)
    \end{equation}
    where we now note that for sufficiently large $K$, one has $\min(4e^{- nt^2 / 32}, 1) \leq 2e^{-nt^2/K}$ for any $n \geq 1$ and $t \geq 0$. Hence, one has
    \begin{equation}\label{eq:subg-mean}
        \P(|f(\bl{v}) - \E f(\bl{v})| \geq t) \leq 2e^{-nt^2/K}.
    \end{equation}
    Finally, \cite[Theorem~2.3]{adamczak2014note} implies that random variables satisfying such types of concentration inequalities must in turn satisfy a Hanson-Wright type inequality. That is, Eq.~\eqref{eq:subg-mean} implies that, for fixed matrix $\bl{A} \in \R^{n \times n}$, one has
    \begin{equation}
        \P(\left|\bl{v}^\top \bl{A} \bl v - \E[\bl{v}^\top \bl{A} \bl{v}]\right| \geq t) \leq 2\exp\left(-C \min\left(\frac{n^2t^2}{2K^4\|\bl{A}\|_F^2}, \frac{nt}{K^2\|\bl{A}\|}\right)\right)
    \end{equation}
    as desired.
\end{proof}

\begin{lemma}[Expectations of Quadratic Forms]\label{lem:expqform}
Let $\bl{v}$ be uniformly distributed on $S^{n - 1}$, and let $\bl{A}$ be an $n \times n$ matrix. Then
\begin{equation}
    \E[\bl{v}^\top \bl{A} \bl{v}] = \Tr(\bl{A} \E[\bl{v}\bl{v}^\top]) = \frac{\Tr(\bl{A})}{n}
\end{equation}
\end{lemma}
\begin{proof}
The first equality follows from the cyclic property of trace. The second can be seen to follow from the observation used in the proof of Lemma~\ref{lem:riri-cov} that $\E[\bl{o}_i \bl{o}_i^\top] = \frac{1}{p} \bl{I}_p$ and the fact that columns of a Haar distributed matrix are uniform on the sphere (see e.g. \cite{Meckes_2019}). Alternatively, recall that for $\bl{g} \sim \N(0, \I_n)$, $\frac{\bl{g}}{\|\bl{g}\|} \indep \|\bl{g}\|$ from \cite[Theorem~2.7.1]{mardia2024multivariate}. Hence
\begin{equation*}
    \E[\bl{v}\bl{v}^\top] = \E[\bl{g}\bl{g}^\top / \|\bl{g}\|^2] = \E[\bl{g}\bl{g}^\top] / \E[\|\bl{g}\|^2] = \I_n/n. 
\end{equation*}
\end{proof}
\begin{lemma}\label{lem:vector-prod}
    Let $\bl{v}_n$ be a sequence of random vectors, where $\bl{v}_n$ is uniformly distributed on $S^{n - 1}$. Let $\bl{u}_n$ be another sequence of random vectors independent of $\bl{v}_n$, where $\limsup \|\bl{u}_n\| < C$ for some constant $C$ almost surely. Then $\bl{v}_n^\top \bl{u}_n \xrightarrow{a.s.} 0$. 
\end{lemma}
\begin{proof}
We replace $\mathbf{v}_n$ with $\frac{\mathbf{g}}{\|\mathbf{g}\|}$ where $\mathbf{g} \sim \mathsf{N}(0, \mathbf{I}_p/p)$. Then conditional on $\mathbf{u}_n$, we have $\mathbf{g}^\top \mathbf{u}_n \sim \mathsf N(0, \|\mathbf{u}_n\|^2/p)$. We then use Mill's inequality to show $\mathbf{g}^\top \mathbf{u}_n \xrightarrow{a.s.} 0$:
\begin{align*}
    \mathbb P\left(|\mathbf{g}^\top \mathbf{u}_n| > t \mid \mathbf{u}_n \right) &\leq \sqrt{\frac{2}{\pi}} \frac{\|\mathbf{u}_n\|}{t \cdot \sqrt{p}} \exp\left(-t^2/2 \cdot \frac{p}{\|\mathbf{u}_n\|^2}\right).
\end{align*}
Since $\limsup_{n \to \infty} \|\mathbf{u}_n\|$ is a.s. bounded, the bound above is summable in $n$, so by Borel-Cantelli, the convergence is almost sure. An application of Slutsky finishes, since we know $\|\mathbf{g}\| \xrightarrow{a.s.} 1$ by the strong Law of Large Numbers, and hence $\mathbf{g}^\top \mathbf{u}_n / \|\mathbf{g}\| \to 0$.
\end{proof}

\subsection{Proof of Theorem~\ref{thm:orig-gcv}}
First recall the statement of Theorem~\ref{thm:orig-gcv}:
Under the stated assumptions,
\begin{equation*}
    \GCV_n(\lambda) - \frac{r^2 (v_{\bl{D}}(-\lambda) - \lambda v'_{\bl{D}}(-\lambda)) + \sigma^2\gamma v'_{\bl{D}}(-\lambda)}{\gamma v_{\bl{D}}(-\lambda)^2} \xrightarrow{a.s.} 0.
\end{equation*}

 We analyze the numerator and denominator separately. The latter is simple: 
\begin{align*}
    (1 - \Tr(\bl{S}_{\lambda})/n)^2  &= \left(1 - \Tr((\bl{D}^\top \bl{D} + \lambda \I)^{-1}\bl{D}^\top\bl{D})/n\right)^2 = \left(\frac{1}{n}\sum_{i = 1} ^n \frac{\lambda}{\lambda + D_{ii}^2}\right)^2 = \gamma^2 \lambda^2 v_{\bl{D}}(-\lambda)^2
\end{align*}
The numerator only requires slightly more work. Recall $\bl{y} = \bl{X} \bs\beta + \bs\epsilon$. Furthermore, note that the GCV numerator can be rewritten as below:
\begin{align*}
    \frac{1}{n}\bl{y}^\top (\I - \bl{S}_{\lambda})^2 \bl{y} = \underbrace{\frac{1}{n} \bs\beta^\top \bl{X}^\top (\bl{I} - \bl{S}_{\lambda})^2 \bl{X}\bs\beta}_{T_1} + 2 \underbrace{\frac{1}{n} \bs\epsilon^\top (\I - \bl{S}_{\lambda})^2 \bl{X}\bs\beta}_{T_2} + \underbrace{\frac{1}{n}\bs\epsilon^\top (\I - \bl{S}_{\lambda})^2 \bs\epsilon}_{T_3}.
\end{align*}
We handle each term. First, $T_1$ simplifies easily into a quadratic form involving the uniformly random vector $\mathbf{O}\boldsymbol{\beta}$ which is independent of $\mathbf{D}$. This will later allow us to apply Lemma~\ref{lem:hw}:
\begin{align*}
    T_1 &= \frac{1}{n} (\bl{O}\bs\beta)^\top \bl{D}^\top (\I - \bl{D}(\bl{D}^\top \bl{D} + \lambda\I)^{-1} \bl{D}^\top)^2 \bl{D} (\bl{O}\bs\beta) \\ \intertext{The expectation of this quantity is as follows:}
    \E[T_1] &= \frac{\|\bs\beta\|^2}{n} \frac{\Tr(\bl{D}^\top (\bl{I} - \bl{D} (\DtD + \lambda \I)^{-1} \bl{D}^\top)^2\bl{D})}{p} \\
    &= \frac{\|\bs\beta\|^2}{n} \frac{1}{p} \sum_{i = 1}^{p} D_{ii}^2\left(1 - \frac{D_{ii}^2}{D_{ii}^2 + \lambda}\right)^2 \\
    &= \frac{\|\bs\beta\|^2}{n} \frac{\lambda^2}{p} \sum_{i = 1} ^p \frac{D_{ii}^2}{(D_{ii}^2 + \lambda)^2} \\
    &= \frac{\|\bs\beta\|^2}{n} \lambda^2 \left(m_{\bl{D}}(-\lambda) - \lambda m'_{\bl{D}}(-\lambda)\right) = \frac{\|\bs\beta\|^2}{n} \frac{\lambda^2}{\gamma} \left(v_{\bl{D}}(-\lambda) - \lambda v'_{\bl{D}}(-\lambda)\right)
\end{align*}
Now, letting $\bl{P} = \bl{D}^\top (\bl{I} - \bl{D} (\DtD + \lambda \I)^{-1} \bl{D}^\top)^2\bl{D}$, Hanson-Wright (Lemma~\ref{lem:hw}) implies
\begin{align*}
    \P(|T_1 - \E[T_1]| > t) \leq  2 \exp \left(- c \min\left(\frac{p^2 t^2}{2K^4 \|\bl{P}\|_F^2}, \frac{pt}{K^2 \| \bl{P} \|} \right)\right) = 2 \exp \left( p \min\left(\frac{t^2}{2K^4 \frac{1}{p} \|\bl{P}\|_F^2}, \frac{t}{K^2 \| \bl{P} \|} \right)\right)
\end{align*}
Note that
\begin{align*}
    \bl{P} = \mathrm{diag} \left( \frac{\lambda^2 D_{ii}^2}{(D_{ii}^2 + \lambda)^2} \right)_{i = 1} ^p
\end{align*}
Hence
\begin{align*}
    \frac{1}{p} \|\bl{P}\|_F^2 &= \frac{1}{p} \sum_{i = 1} ^p \left(\frac{\lambda^2 D_{ii}^2}{(D_{ii}^2 + \lambda)^2}\right)^2 = \frac{1}{p} \sum_{i = 1} ^p \left(\frac{\lambda^2}{(D_{ii}^2 + \lambda)^2}\right)^2 D_{ii}^4 \leq \|\bl{D}\|_{\op}^4 \\
    \|\bl{P}\|_{\op} &= \frac{\lambda^2 D_{11}^2}{(D_{11}^2 + \lambda)^2} \leq \|\bl{D}\|^2
\end{align*}
By assumption, both terms are almost surely bounded in the limit. Recalling that $p(n) / n \to \gamma$, the above bound is summable in $n$ and hence Borel-Cantelli implies almost sure convergence.

To handle the second term, we use a slightly different method. 
\begin{align*}
        T_2 &= \frac{1}{n} {\bs\epsilon}^\top (\bl{I} - \bl{S}_{\lambda})^2 \bl{X}\bs\beta \\
        &= \frac{1}{n} (\bl{O} \bs\beta)^\top \underbrace{\bl{D}^\top \bl{Q} (\bl{I} - \bl{S}_{\lambda})^2 {\bs\epsilon}}_{\bl{T}_4}
\end{align*}
Note that $\tilde{\bs\beta} = \bl{O} \bs\beta$ is a uniformly random vector of norm $\|\bs\beta\|$. We replace this with $\|\bs\beta\| \frac{\bl{g}}{\|\bl{g}\|}$ where $\bl{g} \sim \N(0, \I_p/p)$. Then conditional on $\bl{D},\bs{\epsilon}$, we have $\frac{1}{\sqrt{n}} \bl{g}^\top \bl{T}_4 \sim \N(0, \|\bl{T}_4\|/(pn))$. Furthermore note $\|\bl{T}_4\| \leq \|\bl{D}^\top\| \|\bl{Q}\| \|\I - \bl{S}_{\lambda}\|^2 \|\bs\epsilon\| \leq \|\bl{D}^\top\|\|\bs\epsilon\|$. Some standard tail bounds complete this.
\begin{align*}
    \frac{1}{n}  \|\bs\beta\| \frac{\bl{g}^\top}{\|\bl{g}\|} \bl{T}_4 &= \frac{\|\bs\beta\|}{\sqrt{n}} \frac{1}{\|\bl{g}\|} \left(\frac{1}{\sqrt{n}} \bl{g}^\top \bl{T}_4\right) = r \cdot \frac{1}{\|\bl{g}\|} \frac{1}{\sqrt{n}} \bl{g^\top \bl{T}_4}
\end{align*}
Hence
\begin{align*}
    \P\left(|\frac{1}{\sqrt{n}} \bl{g^\top \bl{T}_4}| > t \mid \bl{D}, \bs\epsilon\right) &\leq \sqrt{\frac{2}{\pi}} \frac{\|\bl{T}_4\|}{t \cdot \sqrt{pn}} \exp\left(-t^2/2 \cdot \frac{pn}{\|\bl{T}_4\|^2}\right).
\end{align*}
We claim that $\limsup_{n \to \infty} \|\bl{T}_4\|/\sqrt{n}$ is bounded. This follows from the bound above, the law of large numbers applied to $\bs\epsilon$, and the assumption on $\bl{D}$. Hence the bound above is summable in $n$ once more, so by Borel-Cantelli, the convergence is almost sure. An application of Slutsky finishes, since we know $\|\bl{g}\| \xrightarrow{a.s.} 1$ by the strong Law of Large Numbers. The convergence for $T_3$ follows directly from \cite[Lemma~C.3]{dobriban2015highdimensional}

\subsection{Proof of Theorem~\ref{thm:ridge}}\label{app:thm2-proof}
\begin{proof}[Proof of Theorem~\ref{thm:ridge}]
In order to understand how one should tune the regularization parameter in this problem, we first must understand the out of sample risk for a given value of $\lambda$. 
The risk admits a decomposition into three terms:
\begin{align*}
    \frac{1}{n'} \E \left[\left\|\tbl{X}(\bs\beta - \hat{\bs\beta})\right\|_2^2 \mid \bl{X}, \bl{y} \right] &= \frac{1}{n'} \cdot \frac{\E[\Tr(\tbl{D}^\top \tbl{D})]}{p} \left\|\bs\beta - \hat{\bs\beta}\right\|_2^2 \\
    &= \frac{\E[\Tr(\tbl{D}^\top \tbl{D})]}{n' p} \|(\I - (\XtX + \lambda \I)^{-1} \XtX)\bs\beta - (\XtX + \lambda\I)^{-1} \bl{X}^\top \bs\epsilon\|_2^2 \\
    &= \frac{n\E[\Tr(\tbl{D}^\top \tbl{D})]}{n' p}\left[\underbrace{ \frac{1}{n} \bs\beta (\I - (\XtX + \lambda\I)^{-1} \XtX)^2 \bs\beta}_{T_1} \right.\\
    &\qquad \left.+ \underbrace{\frac{1}{n} \bs\epsilon^\top \bl{X} (\XtX + \lambda\I)^{-2} \bl{X}^\top \bs\epsilon}_{T_2} - \underbrace{2 \bs\epsilon^\top \bl{X} (\XtX + \lambda\I)^{-2} \XtX \bs\beta}_{T_3} \right]
\end{align*}
\begin{remark}\label{rmk:scaling-explanation}
    The normalization assumption (Assumption~\ref{item:scaling}) ensures that the prefactor is finite. Without further loss of generality, we can assume it is 1. In fact it is sufficient for all proofs to simply assume that $\limsup \frac{n\E[\Tr(\tbl{D}^\top \tbl{D})]}{n' p} \leq C$, but assuming it is constant simplifies notation.

    For intuition why this term is bounded, note that $\Tr(\tbl{D}^\top \tbl{D}) = \Tr(\tbl{X}^\top \tbl{X})$, which is the covariance matrix of $n'$ samples normalized by $\sqrt{n}$ (the reason for this is given in Remark~\ref{rmk:scaling-1} -- since $\bs\beta$ is scaled up by $\sqrt{n}$ all data must be scaled down by $\sqrt{n}$). Hence we expect this trace to be order $n'p/n$, which exactly cancels. 

\end{remark}
The first term corresponds to bias, the second to the variance, and the third is negligible. 
For the first term,
\begin{align*}
    T_1 &= n^{-1}\bs\beta(\I - (\XtX + \lambda\I)^{-1} \XtX)^2\bs\beta \\
    &= n^{-1}(\bl{O}\bs\beta)^{\top} \left[\I - (\DtD + \lambda\I)^{-1}\DtD \right]^{2} (\bl{O}\bs\beta) \\
    &= n^{-1}(\bl{O}\bs\beta)^{\top} \left( \lambda (\DtD + \lambda\I)^{-1} \right)^{2} (\bl{O}\bs\beta)
\end{align*}
Again let $\bl{b} = (\bl{O}\bs\beta)/\|\bs\beta\|$ and $\bl{P} = \left( \lambda (\DtD + \lambda\I)^{-1} \right)^{2}$, and note $\Tr(\bl{P}) = \sum_{i = 1} ^p \frac{\lambda^2}{(D_{ii}^2 + \lambda)^2}$ and $\|\bl{P}\|_F^2 = \sum_{i = 1} ^p \frac{\lambda^4}{(D_{ii}^2 + \lambda)^4}$. Furthermore, one has $\E[\bl{b}^\top \bl{P} \bl{b}] = \frac{1}{p} \Tr(\bl{P})$ using \ref{lem:expqform}, and clearly $\|\bl{P}\| \leq 1$.
We now apply Lemma~\ref{lem:hw} to show concentration:
\begin{align*}
    \P\left(\left|\bl{b}^\top \bl{P} \bl b - \frac{1}{p} \Tr(\bl{P}) \right| \geq t\right) &\leq 2\exp\left(-C \min\left(\frac{p^2 t^2}{2K^4 \|\bl{P}\|^2_F}, \frac{pt}{K^2}\right)\right) \\
    &= 2\exp\left(-C p\min\left(\frac{t^2}{2K^4 \tfrac{1}{p} \|\bl{P}\|^2_F}, \frac{t}{K^2}\right)\right)
\end{align*}
Note that $\frac{1}{p} \|\bl{P}\|^2_F$ is always bounded above by $1$, so this bound can never be made vacuous by a certain setting of $\bl{D}$. This bound is summable in $n$ and thus we have almost sure convergence of the difference to zero.

This yields pointwise convergence for the bias. To strengthen this to uniform convergence, we show that the derivative of the difference $\bl{b}^\top \bl{P} \bl{b} - \frac{1}{p} \Tr(\bl{P})$ is almost surely bounded on $[\lambda_1, \lambda_2]$:
\begin{align*}
    \frac{\mathrm{d}}{\mathrm{d}\lambda}\bl{b}^\top \bl{Pb} = \frac{\d }{\d\lambda}\left[\frac{1}{n} \sum_{i = 1} ^p (\bl{o}_i^\top \bs\beta)^2 \frac{\lambda^2}{(D_{ii}^2 + \lambda)^2} \right] &= \frac{1}{n} \sum_{i = 1} ^p (\bl{o}_i^\top \bs\beta)^2 \frac{2\lambda (D_{ii}^2 + \lambda)^2 + 2(D_{ii}^2 + \lambda)(\lambda^2)}{(D_{ii}^2 + \lambda)^4} \\
    &\leq \frac{1}{n} \frac{2\lambda (D_{11}^2 + \lambda)^2 + 2(D_{11}^2 + \lambda)(\lambda^2)}{(\lambda)^4} \sum_{i = 1} ^p (\bl{o}_i^\top \bs\beta)^2 \\
    &= \frac{1}{n} \frac{2\lambda (D_{11}^2 + \lambda)^2 + 2(D_{11}^2 + \lambda)(\lambda^2)}{(\lambda)^4}
\end{align*}
which is indeed almost surely bounded  on $[\lambda_1, \lambda_2]$ in the limit by assumption on $\|\bl{D}\|$. 
Similarly, $\frac{\d }{\d \lambda} (\frac{1}{p} \Tr(\bl{P}))$ is also bounded on this interval by the same quantity, meaning the difference is Lipschitz. Hence one can discretize the interval $[\lambda_1, \lambda_2]$ into a sufficiently fine grid. Pointwise convergence on every point holds at every point on the grid, then the Lipschitz condition assures that the convergence is uniform. 

For the third term $T_3$, we rewrite it as
\begin{align*}
    2(\bs\epsilon/\sqrt{n})^\top \bl{Q}\bl{D} (\DtD + \lambda\I)^{-2} \DtD \cdot (\bl{O}\bs\beta/\sqrt{n})
\end{align*}
and then apply Lemma~\ref{lem:vector-prod} to see that this converges almost surely to zero. We can likewise control the derivative uniformly over the interval and again obtain almost-sure convergence. 

Lastly, for the second term $T_2$, we first compute 
\begin{equation}
    \E \left[\frac{1}{n} \bs\epsilon^\top \bl{X} (\XtX + \lambda\I)^{-2} \bl{X}^\top \bs\epsilon\right] = \frac{\sigma^2}{n} \sum_{i = 1} ^{n \wedge p} \frac{D_{ii}^2}{(D_{ii}^2 + \lambda)^{2}}.
\end{equation}
Then by \cite[Lemma~C.3]{dobriban2015highdimensional}, this term satisfies
\begin{equation}
    \E \left[\frac{1}{n} \bs\epsilon^\top \bl{X} (\XtX + \lambda\I)^{-2} \bl{X}^\top \bs\epsilon\right] - \frac{1}{n} \bs\epsilon^\top \bl{X} (\XtX + \lambda\I)^{-2} \bl{X}^\top \bs\epsilon \xrightarrow{a.s.} 0.
\end{equation}
The derivative is again uniformly controlled, producing uniform convergence once more.    

\end{proof}

\subsection{Proof of Corollary~\ref{cor:empirical-unif}}
\begin{proof}
Under the assumptions of Theorem~\ref{thm:orig-gcv}, for any consistent estimators $(\hat{r}^2, \hat{\sigma}^2)$ of $(r^2, \sigma^2)$, respectively, one has, for any $\lambda_0 > 0$,
\begin{equation}\label{eq:plugin-unif}
    \sup_{\lambda > \lambda_0}\left|(\hat{r}^2 - r^2) \left(\frac{\lambda^2}{\gamma} v'_{\bl{D}}(-\lambda) + \frac{\gamma - 1}{\gamma}\right) + (\hat{\sigma}^2 - \sigma^2) ( v_{\bl{D}}(-\lambda) - \lambda v_{\bl{D}}'(-\lambda))\right| \xrightarrow{p} 0.
\end{equation}
If $\hat{r}^2, \hat{\sigma}^2$ are strongly consistent, then naturally the convergence is almost sure. This is immediate from the fact that
\begin{align*}
    \lambda^2 v'_{\bl{D}} (-\lambda) = \frac{1}{n} \sum_{i = 1} ^n \frac{\lambda^2}{(D_{ii}^2 + \lambda)^2} &\leq 1 \\
    v_{\bl{D}}(-\lambda) - \lambda v'_{\bl{D}}(-\lambda) = \frac{1}{n} \sum_{i = 1} ^{n \wedge p} \frac{D_{ii}^2}{(D_{ii}^2 + \lambda)^2} &\leq \frac{1}{4\lambda}.
\end{align*}
Combining Eq.~\eqref{eq:plugin-unif} with Theorem~\ref{thm:ridge} then completes the corollary.
\end{proof}

\subsection{Proof of Lemma~\ref{lem:gcv-as}}
This was proven in the course of Theorem~\ref{thm:orig-gcv}, as the training error is nothing but the numerator of the original GCV. 

\subsection{Proof of Corollary~\ref{cor:consistent}}

\subsubsection{Nondegeneracy on the spectrum of $\bl{D}$}\label{app:nondegen}
The exact nondegeneracy condition we require is composed of two portions:
\begin{enumerate}
    \item Define the events $A_n = \bigcup_{i = 1} ^L \{a_i = 0\}$ and $B_n = \bigcup_{i = 1} ^L \{b_i = 0\}$. Then $\mathds{1}(A_n \cup B_n) \xrightarrow{a.s.} 0$. 
    \item There exists $i_* \in \{2, \dots, L\}$ such that $\liminf |a_{i_*} b_1 - a_1 b_{i_*}| > C$ almost surely for some $C > 0$. 
\end{enumerate}
The need for the first condition is clear: our estimator would not even be defined if the entire spectrum were zero; in fact, the problem would be impossible, as the entire data matrix would be zero. The second is more complex, and we give a detailed discussion on its necessity before giving an explicit proof. 

Note furthermore that each of these terms can be observed when using the procedure - this gives the user a good way to check if this assumption is likely satisfied in practice. 

As motivation, we first discuss the simpler case when $L = 2$, as this has connections to the assumptions required for consistent estimation in \cite{li2023spectrumaware}. In this setting, the estimator for $\sigma^2$ becomes
\begin{align*}
    \frac{\left( \frac{t_2}{a_2} - \frac{t_1}{a_1} \right) \left( \frac{b_2}{a_2} - \frac{b_1}{a_1}\right)}{\left( \frac{b_2}{a_2} - \frac{b_1}{a_1} \right)^2} &= \frac{\left( \left(\frac{b_2}{a_2} - \frac{b_1}{a_1}\right)\sigma^2 + o(1) \cdot \left(\frac{1}{a_2} - \frac{1}{a_1} \right) \right) \left( \frac{b_2}{a_2} - \frac{b_1}{a_1}\right)}{\left( \frac{b_2}{a_2} - \frac{b_1}{a_1} \right)^2} \\
    &= \sigma^2 + o(1)\frac{a_1 - a_2}{a_1b_2 - b_1a_2}. 
\end{align*}
where the first line comes from Lemma~\ref{lem:gcv-as}. Similarly, the estimator for $r^2$ is then $r^2 + o(1) \frac{b_2 - b_1}{a_1 b_2 - b_1 a_2}$. From this and the fact that the $a_i$ and $b_i$ are all bounded above, we see that this error term is asymptotically negligible provided $a_1b_2 - b_1a_2$ is bounded away from zero in the limit. In the limiting case where $\lambda_2 = \infty$, one has $a_2 = \frac{1}{p} \sum_{i = 1} ^p D_{ii}^2$ and $b_2 = 1$, meaning we require
\begin{equation}\label{eq:nondegen-1}
    \liminf \left| a_1 b_2 - a_2 b_1  \right| = \liminf \left| \frac{1}{p} \sum_{i = 1} ^p \frac{\lambda^2 D_{ii}^2}{(D_{ii}^2 + \lambda)^2} - \left(\frac{1}{p} \sum_{i = 1} ^p D_{ii}^2\right) \left(\frac{1}{n} \sum_{i = 1} ^n \frac{\lambda^2}{(D_{ii}^2 + \lambda)^2} \right) \right| > c > 0
\end{equation}
almost surely for some $c$. This is exactly the second part of our nondegeneracy assumption. Furthermore, this is the exact analog of \cite[Assumption~7]{li2023spectrumaware} in our setting. The above argument thus establishes that consistent estimation is possible if one assumes equation~\eqref{eq:nondegen-1}. 

We do not believe this condition to be simply technical. Consider the setting where $n \leq p$ and every $D_{ii} = 1$ for $i \leq n$ (meaning $D_{ii} = 0$ for $i > n$). Let $\tilde{\bs\beta} = \bl{O}{\bs\beta}$. Then one is given $\bl{y} = \bl{Q}^\top\tilde{\bs\beta}_{[1:n]} + \bs\epsilon$ and the goal is to infer $r^2 = \|\bs\beta\|_2^2/\sqrt{n} = \|\tilde{\bs\beta}\|_2^2/\sqrt{n}$ and $\sigma^2$. Every entry $y_i$ is now $\bl{q}_i^\top \tilde{\bs\beta}_{[1:n]} + \epsilon_i$; the first term has mean zero and variance approximately $r^2$, and the second has mean zero and variance $\sigma^2$, and it essentially becomes impossible to decouple the effects of $r^2$ and $\sigma^2$. In this setting, the term above is exactly zero, reflecting that the estimating equations become linearly dependent. That said, we believe this issue should not affect practical use cases; the term we require to be bounded away from zero can be computed by the practitioner directly, and while it is impossible to check if some quantity is bounded away in the limit, they can observe its magnitude. 

\subsubsection{Proof}
Having now overviewed a specific case of our method that yields consistent estimation through the use of an additional assumption already used in the literature, we now turn to the general setting. Note that in the general setting, the condition is weaker -- we only require that $\liminf |a_1 b_i - a_i b_1| > C > 0$ holds for \textit{one} of the $i$, meaning one loses nothing by using multiple $\lambda_i$, and in fact in practice we observe this results in improved stability. 

Here the estimators take on the following forms:
\begin{align*}
    \hat{\sigma}^2 &= \sigma^2 + o(1) \frac{\sum_{i = 1} ^L (a_i^{-1} - a_1^{-1})(b_ia_i^{-1} - b_1a_1^{-1})}{\sum_{i = 1} ^L (b_ia_i^{-1} - b_1a_1^{-1})^2} \\
    \hat{r}^2 &= r^2 + o(1) \frac{\sum_{i = 1} ^L (b_i^{-1} - b_1^{-1})(a_ib_i^{-1} - a_1b_1^{-1})}{\sum_{i = 1} ^L (a_ib_i^{-1} - a_1b_1^{-1})^2}.
\end{align*}
As a result, we will require that the coefficient of both error terms is uniformly bounded, which will yield that our estimators are consistent, meaning it suffices to show
\begin{equation}\label{eq:estimator-error-coef}
    \limsup \max \left(\left|\frac{\sum_{i = 1} ^L (a_i^{-1} - a_1^{-1})(b_ia_i^{-1} - b_1a_1^{-1})}{\sum_{i = 1} ^L (b_ia_i^{-1} - b_1a_1^{-1})^2}\right|, \left| \frac{\sum_{i = 1} ^L (b_i^{-1} - b_1^{-1})(a_ib_i^{-1} - a_1b_1^{-1})}{\sum_{i = 1} ^L (a_ib_i^{-1} - a_1b_1^{-1})^2} \right| \right) < C
\end{equation}
almost surely for some $C$. 

Now, note that
\begin{align*}
    a(\lambda) &= \frac{\lambda^2}{\gamma} \left(v_{\bl{D}}(-\lambda) - \lambda v'_{\bl{D}}(-\lambda)\right) = \frac{1}{p} \sum_{i = 1} ^p D_{ii}^2 \left(1 - \frac{D_{ii}^2}{D_{ii}^2 + \lambda} \right)^2 \\
    b(\lambda) &= \lambda^2 v'_{\bl{D}}(-\lambda) = \frac{1}{n} \sum_{i = 1} ^n \frac{\lambda^2}{(D_{ii}^2 + \lambda)^2}
\end{align*}
are both increasing in $\lambda$, and thus $a_1$ is the smallest of the $a_i$; likewise for $b_1$. Furthermore, note that the $a_i$ are bounded above by $\|\bl{D}\|_{\op}$ which is then bounded almost surely by assumption; the $b_i$ are bounded above by $1$. 

We then calculate as follows:
\begin{align*}
    \frac{\left| \sum_{i = 1} ^L (a_i^{-1} - a_1^{-1})(b_ia_i^{-1} - b_1a_1^{-1}) \right| }{\sum_{i = 1} ^L (b_ia_i^{-1} - b_1a_1^{-1})^2} &\leq \frac{\left| \sum_{i = 1} ^L (a_i^{-1} - a_1^{-1})(b_ia_i^{-1} - b_1a_1^{-1}) \right| }{(b_{i_*}a_{i_*}^{-1} - b_1a_1^{-1})^2} \cdot \frac{(a_1a_{i_*})^2}{(a_1a_{i_*})^2} \\
    &= \frac{\left| \sum_{i = 1} ^L (a_{i_*} (a_1a_{i}^{-1}) - a_{i_*})(b_ia_{i_*} (a_1a_i^{-1}) - b_1a_{i_*}) \right| }{(b_ia_{1} - b_1a_{i_*})^2} \\
    &\leq \frac{1}{(b_ia_{1} - b_1a_{i_*})^2} \sum_{i = 1} ^L \left| (a_{i_*} (a_1a_{i}^{-1}) - a_{i_*}) \right| \left| (b_ia_{i_*} (a_1a_i^{-1}) - b_1a_{i_*}) \right|
\end{align*}
Note now that since $a_1$ is the smallest of the $a_i$, $a_1a_i^{-1} \leq 1$ for all $i$; hence, combined with the fact that the $a_i$ and $b_i$ are all bounded above in the limit, the final term is also controlled in the limit.
\begin{align*}
    \limsup \frac{\left| \sum_{i = 1} ^L (a_i^{-1} - a_1^{-1})(b_ia_i^{-1} - b_1a_1^{-1}) \right| }{\sum_{i = 1} ^L (b_ia_i^{-1} - b_1a_1^{-1})^2} &\stackrel{a.s.}{\leq} \frac{1}{C'} L \cdot C'' \cdot 2 \cdot C'' 
\end{align*}
where $C'$ is the constant from the non-degeneracy assumption and $C''$ is the constant from Assumption~\ref{item:bound-op}.

The same computation can be performed for the other term in \eqref{eq:estimator-error-coef} which proceeds exactly as above, but with the $a_i$'s and $b_i$'s exchanged. The result is 
\begin{align*}
    \limsup \frac{\left| \sum_{i = 1} ^L (b_i^{-1} - b_1^{-1})(a_ib_i^{-1} - a_1b_1^{-1}) \right| }{\sum_{i = 1} ^L (a_ib_i^{-1} - a_1b_1^{-1})^2} &\stackrel{a.s.}{\leq} \frac{1}{C'} L \cdot 2 \cdot C'',
\end{align*}
as desired. 

\subsection{Proof of Theorem~\ref{thm:align-gcv}}

\begin{proof}
Most computations are analogous to the proof of Theorem~\ref{thm:ridge}. We first compute $\E[\tbl{X}^\top \tbl{X} \mid \bl{X}, \bl{y}]$ using \cite[Lemma~4]{rangan2019vector}. We find that it is equal to $\sum_{i \in \mathcal{J}_c} (\mathfrak{d}_{i}^2 - \mathfrak{d}_\bulk^2) \bl{o}_i \bl{o}_i^\top + \mathfrak{d}_\bulk^2 \I$. Call this quantity $\bl{C}$ for convenience. We then move to computing the risk. 
\begin{equation}
    \bs\beta - \hat{\bs\beta} = (\I - (\XtX + \lambda\I)^{-1} \XtX) \bs\beta - (\XtX + \lambda\I)^{-1} \bl{X}^\top \bs\epsilon.
\end{equation}
Meanwhile
\begin{equation}
    \E \left[ \|\tbl{X}(\bs\beta - \hat{\bs\beta})\|_2^2  \right]= (\bs\beta - \hat{\bs\beta})^\top \left( \sum_{i \in \mathcal{J}_c} (\mathfrak{d}_{i}^2 - \mathfrak{d}_\bulk^2) \bl{o}_i \bl{o}_i^\top + \mathfrak{d}_\bulk^2 \I \right) (\bs\beta - \hat{\bs\beta}).
\end{equation}
Thus as before, there are three main terms to handle. 

Variance term:
\begin{align*}
    &\bs\epsilon^\top \bl{X} (\XtX + \lambda\I)^{-1} \bl{C} (\XtX + \lambda\I)^{-1} \bl{X}^\top \bs\epsilon \\
    &= \mathfrak{d}_\bulk^2 \bs\epsilon^\top \bl{Q}^\top \bl{D} (\DtD + \lambda\I)^{-2} \bl{D}^\top\bl{Q}\bs\epsilon + \\
    &\qquad \bs\epsilon^\top \bl{Q}^\top \bl{D} (\DtD + \lambda\I)^{-1} \left( \sum_{i \in \mathcal{J}_c} (\mathfrak{d}_{ii}^2 - \mathfrak{d}^2_\bulk) \bl{e}_i\bl{e}_i^\top\right) (\DtD + \lambda\I)^{-1} \bl{D}^\top \bl{Q}\bs\epsilon.
\end{align*}
Note now that
\begin{equation}
    \frac{n}{n'}\mathfrak{d}_\bulk^2 = \frac{n}{n'} \frac{1}{p - |\mathcal{J}_c|} \sum_{i \not \in \mathcal{J}_c} \mathfrak{d}_i^2 = \frac{n \E\Tr(\tbl{D}^\top \tbl{D})}{n' p} + O(1/p) = 1 + O(1/p),
\end{equation}
where the penultimate equality is due to assumption~\ref{item:j-bounded-spec} and the last is due to the scaling assumption \ref{item:scaling}. Note that the $O(1/p)$ term converges almost surely to zero. 

Then
\begin{align*}
    &\frac{n}{n'}\mathfrak{d}_\bulk^2 \left|\frac{1}{n}\bs\epsilon^\top \bl{Q}^\top \bl{D} (\DtD + \lambda\I)^{-2} \bl{D}^\top\bl{Q}\bs\epsilon - \frac{1}{n} \E[\bs\epsilon^\top \bl{Q}^\top \bl{D} (\DtD + \lambda\I)^{-2} \bl{D}^\top\bl{Q}\bs\epsilon]\right|  \\
    &\leq \frac{}{} |\bs\epsilon^\top \bl{Q}^\top \bl{D} (\DtD + \lambda\I)^{-2} \bl{D}^\top\bl{Q}\bs\epsilon - \E[\bs\epsilon^\top \bl{Q}^\top \bl{D} (\DtD + \lambda\I)^{-2} \bl{D}^\top\bl{Q}\bs\epsilon]| + O(1/p) \xrightarrow{a.s.} 0.
\end{align*}
The same arguments as in Theorem~\ref{thm:ridge} then imply uniform convergence to zero.

On the other hand, the second term is order $1$ and hence vanishes when divided by $n$. The convergence is then uniform because the function is monotonic in $\lambda$. It is almost sure through, for instance, using \cite[Lemma~C.3]{dobriban2015highdimensional}. Note that the form of $\mathcal{V}_{\bl{X}}$ contains an extra
\begin{align*}
    \frac{\sigma^2}{n} \sum_{i \in \mathcal{J}_c} \frac{D_{ii}^2}{D^2_{ii} + \lambda} \mathfrak{d}_{i}^2
\end{align*}
term which is uniformly negligible by monotonicity and the assumption that $\mathfrak{d}_{i}^2$ is a.s. bounded in the limit. 

Bias term:
\begin{align*}
    &\bs\beta^\top (\I - (\XtX + \lambda\I)^{-1} \XtX)^\top \bl{C} (\I - (\XtX + \lambda\I)^{-1} \XtX))\bs\beta \\
    &= \mathfrak{d}_\bulk^2 (\bl{O}\bs\beta)^\top (\I - (\DtD + \lambda\I)^{-1} \DtD)^2 (\bl{O} \bs\beta) \\
    &\qquad + (\bl{O} \bs\beta)^\top \left[(\I - (\DtD + \lambda\I)^{-1}) \left[ \sum_{i \in \mathcal{J}_c} (\mathfrak{d}_i^2 - \mathfrak{d}_\bulk^2) \bl{e}_i \bl{e}_i^\top \right] (\I - (\DtD + \lambda\I)^{-1} \right] (\bl{O}\bs\beta). 
\end{align*}
We now need to handle this a bit more carefully than before, because $\bs\beta$ is no longer independent of $\bl{O}$. Note that $\bl{O}\bs\beta = \bl{O} \bs\beta' + \sum_{i \in \mathcal{J}_a} \alpha_i \bl{e}_i$. There are a few different types of terms we need to handle.
\begin{enumerate}
    \item Unaligned terms:
    \begin{align*}
        &\mathfrak{d}_\bulk^2 (\bl{O} \bs\beta') ^\top (\I - (\DtD + \lambda\I)^{-1} \DtD)^2 (\bl{O} \bs\beta') \\
        & \qquad + (\bl{O} \bs\beta)^\top \left[(\I - (\DtD + \lambda\I)^{-1}\DtD) \left[ \sum_{i \in \mathcal{J}_c} (\mathfrak{d}_i^2 - \mathfrak{d}_\bulk^2) \bl{e}_i \bl{e}_i^\top \right] (\I - (\DtD + \lambda\I)^{-1}\DtD) \right] (\bl{O}\bs\beta)
    \end{align*}
    The first term is treated equivalently to the bias in the ridge case in Theorem~\ref{thm:ridge}, where, as in the variance computation above, the $\mathfrak{d}_\bulk^2$ can be handled. This produces uniform convergence to its expectation. The second term is $O(|\mathcal{J}_c|)$ and hence vanishes when divided by $n$. It does so uniformly because it is monotone increasing in $\lambda$. 
    \item Aligned but uncoupled terms:
    \begin{align*}
        &\sum_{i \in \mathcal{J}_a \setminus \mathcal{J}_c} n\alpha_i^2 \mathfrak{d}_\bulk^2 \left(1 - \frac{D_{ii}^2}{D_{ii}^2 + \lambda}\right)
    \end{align*}
    \item Aligned and coupled terms:
    \begin{align*}
        &\sum_{i \in \mathcal{J}_a \cap \mathcal{J}_c} n\alpha_i^2 \mathfrak{d}_i^2 \left(1 - \frac{D_{ii}^2}{D_{ii}^2 + \lambda}\right)
    \end{align*}
    \item Cross terms:
    \begin{align*}
        \sum_{i \in \mathcal{J}_a} \bl{o}_i^\top \bs\beta' \sqrt{n}\alpha_i \left[(\mathfrak{d}_{\bulk}^2 + \1(i \in \mathcal{J}_c) (\mathfrak{d}_{ii}^2 - \mathfrak{d}_\bulk^2) )\left(1 - \frac{D_{ii}^2}{D_{ii}^2 + \lambda}\right)\right]
    \end{align*}
    These terms are only of order $O(\sqrt{n})$ and hence vanish - they do so uniformly because, after taking absolute value of the summand, the resulting term is monotone increasing in $\lambda$, and thus it suffices to control the value at $\lambda_2$. 

    Note that the cross terms between $\alpha_i \bl{e}_i$ and $\alpha_j \bl{e}_j$ are all identically zero and thus do not need to be considered. 
\end{enumerate}
Summing all of these together, one obtains uniform convergence to $\mathcal{B}$. Note that the statement of $\mathcal{B}$ contains an extra
\begin{equation}
    \frac{\lambda^2}{n} \sum_{i \in \mathcal{J}_c} \frac{r^2}{\gamma} \frac{\mathfrak{d}_i^2 - \mathfrak{d}_\bulk^2}{(D_{ii}^2 + \lambda)^2}
\end{equation}
term, which is uniformly negligible. 

Cross term: 
\begin{align*}
    &\bs\epsilon^\top \bl{X} (\XtX + \lambda\I)^{-1} \bl{C} (\I - (\XtX + \lambda\I)^{-1} \XtX)\bs\beta \\
    &= \bs\epsilon^\top \bl{X} (\XtX + \lambda\I)^{-1} \bl{C} (\I - (\XtX + \lambda\I)^{-1} \XtX)(\bs\beta' + \sum_{i \in \mathcal{J}_a} \alpha_i \bl{o}_i) \\
    &= \bs\epsilon^\top \bl{Q} \bl{D} (\DtD + \lambda\I)^{-1} \bl{O} \bl{C} \bl{O}^\top (\I - \DtD + \lambda\I)^{-1} \DtD) \bl{O} \bs\beta' + \\
    &\qquad \sum_{i \in \mathcal{J}_a \cap \mathcal{J}_c} \bl{q}_i^\top \bs\epsilon \frac{D_{ii}}{\lambda + D_{ii}^2} \cdot \mathfrak{d}_i^2 \cdot \frac{\lambda}{D_{ii}^2 + \lambda} \cdot \sqrt{n} \alpha_i
\end{align*}
For the first term, note that $\bl{O} \bl{C} \bl{O}^\top$ is in fact independent of $\bl{O}$. Moreover, the matrix in the middle has bounded operator norm, and thus this term vanishes after dividing by $1/n$ due to Lemma~\ref{lem:vector-prod}, where we apply the strong law of large numbers to bound $\|\bs\epsilon\|$. For the second, we note that this is almost surely, for sufficiently large $n$, uniformly bounded above by 
\begin{align*}
    C\sqrt{n} \cdot (\max_{i \in \mathcal{J}_a} \alpha_i) \cdot \max_{i \in \mathcal{J}_a \cap \mathcal{J}_c} |\bl{q}_i^\top \bs\epsilon|
\end{align*}
Recall we are dividing everything by $n$. Hence if we show that $|\bl{q}_i^\top \bs\epsilon|/\sqrt{n} \xrightarrow{a.s.} 0$, then we will be done with a union bound. We rewrite $|\bl{q}_i^\top \bs\epsilon| = \bs\epsilon^\top (\bl{q}_i \bl{q}_i^\top) \bs\epsilon$ and can now apply \cite[Lemma~C.3]{dobriban2015highdimensional} to find
\begin{align*}
    \frac{1}{n} \bs\epsilon^\top (\bl{q}_i \bl{q}_i^\top) \bs\epsilon - \frac{\sigma^2}{n} \xrightarrow{a.s.} 0.
\end{align*}
Thus $|\bl{q}_i^\top \bs\epsilon|^2/n$ converges a.s. to $0$, and thus by continuous mapping so does $|\bl{q}_i^\top \bs\epsilon|/\sqrt{n}$.

\end{proof}

\subsection{Proof of Lemma~\ref{lem:align-consistent}}
The following is taken from \cite{li2023spectrumaware}. We note that in this work, the errors are assumed Gaussian. We will mark which ones require this and how we believe they can be relaxed. First, the consistency of $\alpha_i$ is immediate from the first result of \cite[(40)]{li2023spectrumaware}, where we note that the proof does not require Gaussianity of the errors. The one step in the proof where care must be made is in controlling
\begin{equation}
    \frac{1}{p}\left\|\mathbf{O}_{\mathcal{J}}^{\top}\left(\mathbf{D}_{\mathcal{J}}^{\top} \mathbf{D}_{\mathcal{J}}\right)^{-1} \mathbf{D}_{\mathcal{J}}^{\top} \mathbf{Q} \varepsilon\right\|_2^2,
\end{equation}
but we note that this can be done through \cite[Lemma~C.3]{dobriban2015highdimensional}. 

Then the reduced model is right-rotationally invariant by \cite[p. 58, (124)]{li2023spectrumaware}. Showing that the reduced model is right-rotationally invariant requires Gaussianity and is what leads to the added assumption of Lemma~\ref{lem:align-consistent}. Without the Gaussian assumption, the errors in the reduced model will not be i.i.d. However, in general we only require that certain inner products concentrate. We believe this will still hold in the reduced model when one explicitly tracks the form of the error. 

\subsection{Proof of Corollary~\ref{cor:align-unif}}
We will prove uniform convergence term by term. We begin with the variance.
Define the following:
\begin{align*}
    A &= \frac{1}{n}\left[\sum_{i = 1} ^n \frac{D_{ii}^2}{(D_{ii}^2 + \lambda)^2} \left(\mathfrak{d}_\bulk^2 + \1(i \in \mathcal{J}_c)(\mathfrak{d}_i^2 - \mathfrak{d}^2_\bulk)\right)\right] \\
    \hat{A} &= \frac{1}{n}\left[\sum_{i = 1} ^n \frac{D_{ii}^2}{(D_{ii}^2 + \lambda)^2} \left(\hat{\mathfrak{d}}_\bulk^2 + \1(i \in \mathcal{J}_c)(\hat{\mathfrak{d}}_i^2 - \hat{\mathfrak{d}}^2_\bulk)\right)\right] \\
    T_1 &= \sigma^2 A \qquad T_2 = \hat{\sigma}^2 \hat{A} \qquad T_3 = \hat{\sigma}^2 {A} 
\end{align*}
Note that $T_1 = \mathcal{V}_\bl{X}$. First, note that $\hat{A}$ and $\hat{\sigma}^2/n$ are both almost surely bounded for sufficiently large $n$. This follows from the fact that the estimators are strongly consistent for bounded quantities. We can thus uniformly control $T_1 - T_3$ and then $T_3 - T_2$, producing uniform convergence. 

For the bias, we wish to uniformly control
\begin{align*}
    \mathcal{B}_\bl{X}(\hat{\bs\beta}_\lambda, \bs\beta) &= \frac{\lambda^2}{n} \sum_{i = 1} ^p \frac{\left(r^2/\gamma  + \1(i \in \mathcal{J}_a)n\alpha_i^2\right) \left(\mathfrak{d}_\bulk^2 + \1(i \in \mathcal{J}_c)(\mathfrak{d}_i^2 - \mathfrak{d}^2_\bulk)\right)}{(D_{ii}^2 + \lambda)^2}. 
\end{align*}
One can then expand the numerator into 4 constituent terms. Arguments similar to that of the variance then show uniform convergence for each of them. 

\section{Miscellaneous Details}
\subsection{Experimental details}\label{app:exp}

All experiments were run on a personal laptop. Each experiment took only on the order of 1 minute. All experiments are carried out with $n = p = 1000$, because this scenario is most sensitive to the amount of regularization needed (as with $\lambda = 0$, often times the risk is infinite). In Section~\ref{sec:gcv-vanilla}, we first sample a training set. We then repeatedly resample the training noise $\bs\epsilon$ because this is more economical. Each time, we then run each cross-validation technique and plot the loss curve over $\lambda$. The estimated risk is then the risk reported by the loss curve at its minimum. The tuned risk is then the actual mean-squared error (which has an analytical expression) produced when using the estimator with that value of $\lambda$. For ROTI-GCV, Assumption~\ref{item:scaling}, while useful theoretically, proves to provide some practical challenges, mainly that finding the correct rescaling of the data that makes this assumption hold is not always simple, since the trace does not necessarily concentrate. In the experiments in Section~\ref{sec:gcv-vanilla}, the data is drawn according to distributions satisfying this assumption, and the data is not normalized after being sampled. That said, normalizing the data does not affect any results. It is important to note that this scaling never affects the tuned value of $\lambda$ and instead only affects the risk estimate.

In Section~\ref{sec:gcv-aligned}, it becomes necessary to estimate the eigenvalues of the test set data. To do this, we now sample a test set ($n = p = 1000$ again), and then compute the eigenvalues of the test set. A similar result should hold when using eigenvalues of the training set, since here both sets are equal in distribution. 

\subsection{Semi-real experimental details}\label{app:semi-real}
\subsubsection{Speech data}
Speech data was taken directly from \cite{OpenMLSpeech} and then was standardized (demeaned and rescaled to have variance 1), as is usual when using ridge-regression. 

\subsubsection{Residualized returns}

\begin{enumerate}[itemsep=0pt, parsep=0pt]
    \item There is one row in $X$ for each minute. Each row is the residualized return over the last 30 minutes, so consecutive rows are heavily correlated, since they share 29 minutes of returns. $X$ has 493 rows and 493 columns ($n = p = 493$).
    \item We generate $\beta$ uniformly on the sphere of radius $r^2\sqrt{n}$, and the training labels $y = X\beta + \epsilon$, where $\epsilon_i \sim N(0, \sigma^2)$.
    \item We plot the predicted loss curve given by each cross-validation metric
    \item We compute the actual loss curve over the test data.
\end{enumerate}

\subsection{Proof of right-rotational invariance}\label{app:riri-proof}
Both the equicorrelated and autocorrelated cases can be written as $\bs\Sigma^{1/2} \bl{G}$, where $\bl{G}$ has i.i.d. standard Gaussian entries; the rotational invariance of the Gaussian then implies the right-rotational variance of $\bs\Sigma^{1/2}\bl{G}$ in these cases.

$t$-distributed data: A multivariate $t$ distribution with $\nu$ degrees of freedom can be generated by first sampling $\bl{y} \sim \N(\bl{0}, \bl{I}_n)$ and $u \sim \chi_\nu$; then $\bl{x} = \bl{y} / \sqrt{u / \nu}$; the rotational invariance of the Gaussian implies that this distribution is rotationally invariant. This proves rotational invariance for a single sample; stacking multiple independent samples into a matrix retains the rotational invariance. 

Products of Gaussian matrices: follows from rotational invariance of the Gaussian.

Matrix normals: Recall that $\bs\Sigma^{\mathrm{col}}$ is drawn from an inverse Wishart distribution with identity scale and $(1 + \delta)p$ degrees of freedom; this can be generated by $(\bl{G}_2^\top \bl{G}_2)^{-1}$, where $\bl{G}_2 \in \R^{(1 + \delta)p \times p}$ has i.i.d. standard Gaussian entries (since the inverse Wishart has identity scale) \cite[Appendix~A]{gelmanbda04}. A matrix normal with row covariance $\bs\Sigma^{\text{row}}$ and column covariance $\bs\Sigma^{\text{col}} = (\bl{G}_2^\top \bl{G}_2)^{-1}$ can then be generated by $(\bs\Sigma^{\text{row}})^{1/2}\bl{G} (\bl{G}_2^\top \bl{G}_2)^{-1/2}$ (with $\bl{G}$ again having i.i.d. Gaussian entries and being independent of $\bl{G}_2$). The claim then follows from the following. For any fixed $\bl{O} \in \mathbb{O}_p$,
\begin{align*}
    (\bs\Sigma^{\text{row}})^{1/2}\bl{G} (\bl{G}_2^\top \bl{G}_2)^{-1/2} \bl{O} &\stackrel{d}{=} (\bs\Sigma^{\text{row}})^{1/2}\bl{G} \bl{O} (\bl{O} \bl{G}_2^\top \bl{G}_2 \bl{O}^\top)^{-1/2} \bl{O}\\
    &= (\bs\Sigma^{\text{row}})^{1/2}\bl{G} (\bl{G}_2^\top \bl{G}_2 )^{-1/2}
\end{align*}
as required. 

Spiked matrices: It suffices to show that for any fixed $\bl{O}$ that $(\lambda \bl{V} \bl{W}^\top, \bl{G}) \stackrel{d}{=} (\lambda \bl{V} \bl{W}^\top \bl{O}, \bl{G} \bl{O})$. Since the first consists of an independent pair, it suffices to show that $\bl{V}\bl{W}^\top \stackrel{d}{=} \bl{V} \bl{W}^\top \bl{O}$, $\bl{G} \stackrel{d}{=} \bl{G}\bl{O}$ and that $\lambda \bl{V}\bl{W}^\top \bl{O}$ is independent of $\bl{G} \bl{O}$. The last two statements are automatic from the rotational invariance of the Gaussian and the independence of the $\bl{V}\bl{W}^\top$ and $\bl{G}$. The first can be seen from the fact that $\bl{W} = \bl{W'}\bl{P}$, where $\bl{P} \in \R^{p \times r}$ has $\bl{P}_{ij} = \mathds{1}(i = j)$ and $\bl{W'} \in \R^{p \times p}$ is Haar. Then $\bl{V}\bl{W}^\top \bl{O} = \bl{V} \bl{P}^\top (\bl{W'})^\top \bl{O} \stackrel{d}{=} \bl{V} \bl{P}^\top (\bl{W'})^\top = \bl{V} \bl{W}^\top $ where the distributional equality follows from the rotational invariance of Haar matrices. 

\subsection{Uniform convergence}\label{app:unif}
Suppose one has a set of (possibly random) functions $f_n$ satisfy, for some deterministic $f$,
\begin{equation}
    \sup_{x \in [x_1, x_2]} |f_n(x) - f(x)| \xrightarrow{a.s.} 0.
\end{equation}
Define
\begin{align*}
    x_{\mathsf{cv}, n} = \argmin_{x \in [x_1, x_2]} f_n(x) \qquad x_* = \argmin_{x \in [x_1, x_2]} f(x).
\end{align*}
Then
\begin{equation}
     \left|f\left(x_{\mathsf{cv}, n}\right) - f(x_*)\right| \to 0.
\end{equation}
Here the random functions $f$ represent cross validation metrics which uniformly converge to the true asymptotic out-of-sample risk $f$. This result then shows that the minimizer chosen by the cross-validation metric tends towards the asymptotically optimal value. Note that this result cannot be obtained with only pointwise convergence. 
\begin{proof}
Note $f(x_{\mathsf{cv}, n}) - f(x_*) \geq 0$ always. Hence it suffices to show an upper bound which tends to zero. 
\begin{align*}
    f(x_{\mathsf{cv}, n}) - f(x_*) &= f(x_{\mathsf{cv}, n}) - f_n(x_{\mathsf{cv}, n}) + \underbrace{f_n(x_{\mathsf{cv}, n}) - f_n(x_*)}_{\leq 0} + f_n(x_*) - f(x_*) \\
    &\leq |f(x_{\mathsf{cv}, n}) - f_n(x_{\mathsf{cv}, n})| + |f_n(x_*) - f(x_*)| \\
    &\leq 2\sup_{x \in [x_1, x_2]} |f_n(x) - f(x)|  \to 0.
\end{align*}
as required.
\end{proof}

\subsection{Distribution notation}\label{app:notation}
\begin{itemize}
    \item $\mathsf{InvWishart}(\bl{\Psi}, \nu)$ denotes an inverse Wishart distribution \cite[Appendix~A]{gelmanbda04} with scale matrix $\bs\Psi$ and degrees of freedom $\nu$. 
    \item $\mathsf{MN}(\bs\mu, \bs\Sigma^{\text{row}}, \bs\Sigma^{\text{col}})$ denotes a matrix normal with among-row covariance $\bs\Sigma^{\text{row}}$ and among-column covariance $\bs\Sigma^{\text{col}}$. 
\end{itemize}

\subsection{LOOCV Intractability}\label{app:loocv-bad}

The LOOCV in our setting is difficult to analyze 
because the objective depends on $\bl{Q}$, on which we do not place any assumptions. In particular, one can derive
\begin{align*}
    \LOOCV_n(\lambda) &= \bl{y}^\top (\bl{I} - \bl{S}_{\lambda}) \bl{D}_{\lambda}^{-2} (\bl{I} - \bl{S}_{\lambda}) \bl{y}/n \\
    &= \bl{y}^\top (\I - \bl{Q}^\top (\bl{D}(\DtD + \lambda\I)^{-1} \bl{D}^\top) \bl{Q}) \bl{D}_{\lambda}^{-2} (\I - \bl{Q}^\top (\bl{D}(\DtD + \lambda\I)^{-1} \bl{D}^\top) \bl{Q}) \bl{y} / n
\end{align*}
where $\bl{D}_{\lambda} = \mathrm{diag}((1 - (\bl{S}_{\lambda})_{ii})_{i = 1} ^n)$. We try to simplify the diagonal term first.
\begin{align*}
    (D_{\lambda})_{ii} &= 1 - \bl{x}_i^\top (\bl{X}^\top \bl{X} + \lambda \I)^{-1} \bl{x}_i \\
    &= 1 - \bl{q}_i^\top \bl{D} \bl{O} (\bl{O}^\top \bl{D}^\top \bl{D} \bl{O} + \lambda\I)^{-1} \bl{O}^\top \bl{D}^\top \bl{q}_i \\
    &= 1 - \bl{q}_i^\top \bl{D} (\DtD + \lambda\I)^{-1} \bl{D}^\top \bl{q}_i
\end{align*}
where $\bl{q}_i$ is the $i$-th row of $\bl{Q}$. 
We analyze
\begin{align*}
    \bl{Q} \bl{D}^{-2}_{\lambda} \bl{Q}^\top &= \sum_{i = 1} ^n \bl{q}_i\bl{q}_i^\top (\bl{D}_{\lambda} ^{-2})_{ii}.
\end{align*}
Then $\LOOCV$ should contribute two terms, as the cross term should vanish. They are
\begin{align*}
    (\bl{O}\bs\beta)^\top \bl{D}^\top \bl{Q} (\I - \bl{D}(\DtD + \lambda\I)^{-1}\bl{D}^\top) (\bl{Q} \bl{D}_{\lambda} ^{-2} \bl{Q}^{\top})(\I - \bl{D}(\DtD + \lambda\I)^{-1}\bl{D}^\top)\bl{O}\bs\beta
\end{align*}
which should concentrate around its trace. 

The second is 
\begin{align*}
    \bs\epsilon^\top (\I - \bl{D}(\DtD + \lambda\I)^{-1}\bl{D}^\top) (\bl{Q} \bl{D}_{\lambda} ^{-2} \bl{Q}^{\top})(\I - \bl{D}(\DtD + \lambda\I)^{-1} \bs\epsilon
\end{align*}
which also concentrates around some multiple of its trace. 

Finding this trace now involves numerous terms with $\bl{Q}$. Everything except the center term is diagonal, but immediately one arrives at the following difficulty:
\begin{align*}
    (\bl{Q} \bl{D}^{-2}_{\lambda} \bl{Q})_{ii} &= \sum_{j = 1} ^n (\bl{q}_j \bl{q}_j^\top)_{ii} (\bl{D}_{\lambda} ^{-2}) _{jj} \\
    &= \sum_{j = 1} ^n (\bl{q}_{ji})^2 (\bl{D}_{\lambda} ^{-2}) _{jj} 
\end{align*}
Assuming that the rows of $\bl{Q}$ are exchangeable conditional on $\bl{D}$ is insufficient to make progress. If $\bl{Q}$ is assumed to also be Haar, then likely closed forms can be derived, but this is not pursued. 

\subsection{Algorithm Complexity}

Assuming matrix multiplications and inversions cost $O(p^3)$, ROTI-GCV takes $O(p^3)$ time. The entire process takes only a constant factor more time than computing a single ridge solution, as the fitting process implicitly requires fitting multiple ridge solutions. 

\section{Verifying assumptions for real data}\label{app:verify}

Here, we work through applications of our method. We take a look at 5 application scenarios. Our focus will be on detecting $\mathcal{J}_c$, since that is the new concept introduced in this work. For each of the first four examples, we will generate the signal vector uniformly on the sphere, so there is no signal alignment. Afterwards, we will discuss a 5th example with signal-PC alignment. We refer to \cite[Section~4]{li2023spectrumaware} for more detailed discussion on alignment detection, and the details on hypothesis testing for signal-PC alignment, presenting a method for selecting $\mathcal{J}_a$ that controls the false discovery rate. \textbf{For each, we discuss how to check whether the most important assumptions hold.} This includes how the index sets $\mathcal{J}_a$ and $\mathcal{J}_c$ (aligned/coupled eigenvectors, definitions both restated below) can be selected, and how our estimation scheme performs. 

The five examples we discuss are as follows:
\begin{enumerate}
    \item residualized returns, where each row contains the residualized returns a 30 minute interval; this was shown in Figure~\ref{fig:sp500}, and are analyzed in more detail in Figure~\ref{fig:resid-ret-big}.
    \item the gaussian mixture, where $\bl{x}_i \sim \frac{1}{2}\N(\vec{3}, I_p) + \frac{1}{2}N(-\vec{3}, I_p)$; this was shown in Figure~\ref{fig:mixture}, and is analyzed in more detail in Figure~\ref{fig:gaussian-mixture-big}.
    \item speech data, a real dataset retrieved from OpenML, shown in \ref{fig:speech-data-result}, analyzed in more detail in Figure~\ref{fig:speech-data-big};
    \item unresidualized returns, where each row contains the unresidualized return over a 30 minute interval, shown in Figure~\ref{fig:bad};
    \item speech data with Signal-PC alignment, shown in Figure~\ref{fig:speech-data-align} shown again in Figure~\ref{fig:final};
\end{enumerate}

\textbf{As a refresher for notation}:
\begin{itemize}
    \item We study ridge regression given data $(\bl{X}, \bl{y})$, where $\bl{y} = \bl{X}\beta + \bs\epsilon$. The training set $\bl X = \bl Q^\top \bl D \bl O \in \mathbb{R}^{n \times p}$ is right-rotationally invariant, meaning $\bl O \sim \mathrm{Haar}(\mathbb{O}_p)$. The error $\bs\epsilon$ is centered with independent entries.
    \item We are interested in the test loss on a test set $\tbl{X} = \tbl{Q}^\top \tbl{D} \tbl{O} \in \mathbb{R}^{n' \times p}$, defined as 
    \[R_{\bl{X}, \bl{y}}(\hat{{\bs\beta}}(\bl{X}, \bl{y}), {\bs\beta})=\frac{1}{n^{\prime}} \mathbb{E}\left[\|\tbl{{X}} \hat{{\bs\beta}}-\tbl{{X}} {\bs\beta}\|^2 \mid \bs{X}, \bs{y}\right].\]
    \item The two key parameters are the signal strength $r^2 = \|\bs\beta\|_2^2/n$ and the noise level $\sigma^2 = \mathbb{E}[\bs\epsilon_i^2]$
    \item Let $\bl o_i^\top$ denote the $i$-th row of $\bl O$ and $\tbl{o}_i$ denote the $i$-th row of $\tbl{O}$. 
    \item \textbf{Aligned eigenvectors} refer vectors $o_i$ which align with the signal $\beta$. They are indexed by $\mathcal{J}_a$, and satisfy $\bs\beta = \sum_{i \in \mathcal J_a} \sqrt{n}\alpha_i \bl o_i + \bs\beta'$. 
    \item \textbf{Coupled eigenvectors} refer to vectors $\bl o_i$ which align with $\tbl{o}_i$, which are eigenvectors of the training covariates. They are indexed by $\mathcal{J}_c$, and we assume for each $i \in \mathcal{J}_c$, that $\bl o_i = \tbl{o}_i$. 
    \item Aligned and coupled eigenvectors allow us to model better model structures in real data that may not initially fall under the coverage of our method. We will refer to scenarios covered without the need for coupled/aligned eigenvectors as the \textbf{simpler setting}; in the main text, these are shown in Figure~1. For scenarios that require coupled/aligned eigenvectors, we refer to them as the \textbf{coupled/aligned} setting; these are covered in Figure~2 of the main text. 
\end{itemize}

\textbf{The most important assumptions}, reproduced informally from the main text, are as follows:
    \begin{enumerate}
        \item On the regularity of singular value distributions:
        \begin{itemize}
            \item In the simpler setting: we require \textbf{Assumption 2}, that the maximum singular value $\lambda_{\text{max}} (D)$ is almost surely bounded. 
            \item In the aligned/coupled setting, we further require \textbf{Assumption 11} that the expected maximum singular value of the test covariates, $\|\mathbb{E}[\tbl{{D}}^{\top} \tbl{{D}}]\|_{\textsf{op}}$, is also almost surely bounded. 
        \end{itemize}
        
        \textbf{How to check this}: One can generally verify that this assumption is satisfied by plotting a histogram of the singular values $(D_{ii})_{i = 1} ^p$ and $(\widetilde{D}_{ii})_{i = 1} ^p$. In the  examples we discuss, we only do this for the training set since we do not expect any distribution shift. As examples, see Figures~\ref{fig:resid-ret-spectrum}, \ref{fig:gaussian-mixture-spectrum} etc. 
        \item On the relation between the eigenvectors of the training set $\{\bl o_i\}_{i = 1} ^p$ and test set $\{\tbl{o}_i\}_{i = 1} ^p$:
        \begin{itemize}
            \item In the simpler setting, we have \textbf{Assumption 6}, which states the test set is independent of the training set.
            \item In the coupled setting, we have \textbf{Assumption 7}, which states the top eigenvectors of the test and training set are exactly equal, i.e. for $i \in \mathcal J_c$, we have $o_i = \tilde{o}_i$. 
        \end{itemize}
        One should think of Assumption 7 as a weakening of Assumption 6, where we allow for additional forms of structure to handle trends we see in data.
        
        \textbf{How to check this}: We compute the overlaps $\sqrt{p}\langle \tilde{o_i}, o_j \rangle$ for the top few eigenvectors (we choose 10). Note that if these two vectors were independent, then one has $\sqrt{p} \langle \tilde{o_i}, \tilde{o_j} \rangle \xrightarrow{p \to \infty} N(0, 1)$. Hence one can empirically observe if two eigenvectors possess overlap that is significantly larger than expected; if they do, then they should be coupled. For examples, see Figures~\ref{fig:resid-ret-numericals}, \ref{fig:gaussian-mixture-numericals}. 

        One limitation of this approach is that each eigenvector of the training set can only be coupled to one in the test set; what this means is that if we find eigenvectors in the training set that overlap heavily with multiple eigenvectors of the test set, then our approach is not expected to perform well. We will see this occur in setting 4. 
        
        \item On the position of $\beta$ relative to $\{o_i\}_{i = 1} ^p$, which are the rows of $O$:
        \begin{itemize}
            \item In the simpler setting, we have \textbf{Assumption 5}, which states that $\{o_i\}_{i = 1} ^p$ are independent of $\beta$.
            \item In the aligned setting, we have \textbf{Assumption 8}, which states a finite collection of eigenvectors indexed by $\mathcal J_a$ are aligned with $\beta$, so that $\beta = \sum_{i \in \mathcal J_a} \sqrt{n}\alpha_i o_i + \beta'$. 
        \end{itemize}
        \textbf{How to check this}: Identification of alignment is discussed in \cite{li2023spectrumaware}, which proposes a hypothesis testing framework for identifying $J_a$. They provide a method to compute a $p$-value per eigenvector that tests whether or not it is aligned; one can then apply the Benjamini-Hochberg procedure to control the false discovery rate. In the 5th example, we use the same approach, where the $p$-values are shown in Table~\ref{table:final-v2}.
    \end{enumerate}
\newpage
\section{Residualized Returns}
\begin{figure}
    \centering
    \begin{subfigure}[t]{0.48\textwidth}
        \centering
    \includegraphics[width=\linewidth]{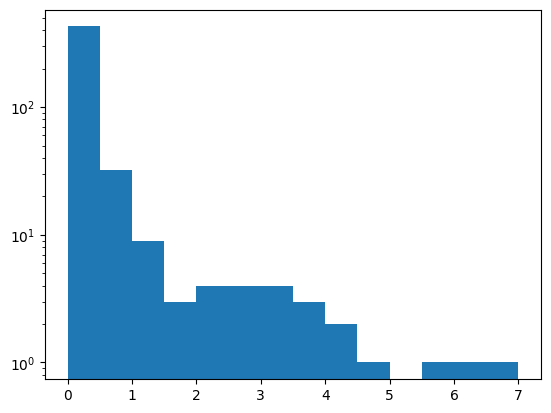}
        \caption{Histogram of singular values $(D_{ii})_{i = 1} ^n$}
        \label{fig:resid-ret-spectrum}
    \end{subfigure}%
    \hfill
    \begin{subfigure}[t]{0.48\textwidth}
        \centering
    \includegraphics[width=\linewidth]{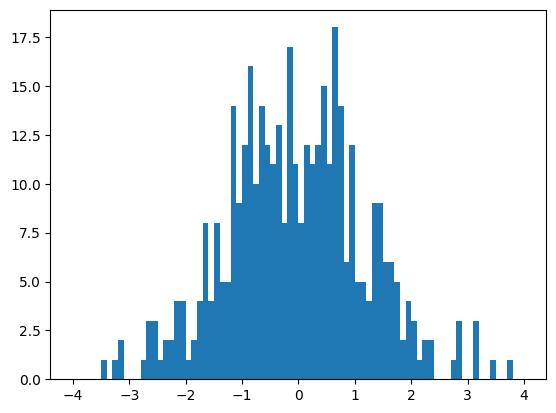}
        \caption{Histogram of overlaps $\{\sqrt{p}\langle \tilde{o}_i, o_j \rangle : 1 \leq i \leq j \leq 20 \}$}
        \label{fig:resid-ret-overlaps}
    \end{subfigure}%
    \hfill
    \begin{subfigure}[t]{0.48\textwidth}
    \centering
{\tiny
\setlength\tabcolsep{1.5pt}\renewcommand{\arraystretch}{1.5}
\sffamily
\begin{tabular}{p{1.5em}C{2.2em}C{2.2em}C{2.2em}C{2.2em}C{2.2em}C{2.2em}C{2.2em}C{2.2em}C{2.2em}C{2.2em}}
 & $\tilde{o}_{1}$ & $\tilde{o}_{2}$ & $\tilde{o}_{3}$ & $\tilde{o}_{4}$ & $\tilde{o}_{5}$ & $\tilde{o}_{6}$ & $\tilde{o}_{7}$ & $\tilde{o}_{8}$ & $\tilde{o}_{9}$ & $\tilde{o}_{10}$ \\
$o_{1}$ & \cellcolor[HTML]{fcaf93} \textcolor[HTML]{000000}{2.87} & \cellcolor[HTML]{fdcab5} \textcolor[HTML]{000000}{2.02} & \cellcolor[HTML]{fee4d8} \textcolor[HTML]{000000}{1.04} & \cellcolor[HTML]{fdc5ae} \textcolor[HTML]{000000}{2.18} & \cellcolor[HTML]{fedaca} \textcolor[HTML]{000000}{1.48} & \cellcolor[HTML]{fedccd} \textcolor[HTML]{000000}{1.37} & \cellcolor[HTML]{fcc1a8} \textcolor[HTML]{000000}{2.34} & \cellcolor[HTML]{ffece3} \textcolor[HTML]{000000}{0.58} & \cellcolor[HTML]{fdd5c4} \textcolor[HTML]{000000}{1.64} & \cellcolor[HTML]{fcc2aa} \textcolor[HTML]{000000}{2.30} \\
$o_{2}$ & \cellcolor[HTML]{fff0e9} \textcolor[HTML]{000000}{0.31} & \cellcolor[HTML]{fee4d8} \textcolor[HTML]{000000}{1.03} & \cellcolor[HTML]{fff3ed} \textcolor[HTML]{000000}{0.15} & \cellcolor[HTML]{fca588} \textcolor[HTML]{000000}{3.20} & \cellcolor[HTML]{fee8de} \textcolor[HTML]{000000}{0.78} & \cellcolor[HTML]{ffece4} \textcolor[HTML]{000000}{0.52} & \cellcolor[HTML]{fee8dd} \textcolor[HTML]{000000}{0.82} & \cellcolor[HTML]{fff0e8} \textcolor[HTML]{000000}{0.35} & \cellcolor[HTML]{ffeee6} \textcolor[HTML]{000000}{0.43} & \cellcolor[HTML]{fee8de} \textcolor[HTML]{000000}{0.76} \\
$o_{3}$ & \cellcolor[HTML]{fcb69b} \textcolor[HTML]{000000}{2.66} & \cellcolor[HTML]{fca588} \textcolor[HTML]{000000}{3.17} & \cellcolor[HTML]{fff4ef} \textcolor[HTML]{000000}{0.07} & \cellcolor[HTML]{fedbcc} \textcolor[HTML]{000000}{1.42} & \cellcolor[HTML]{fc9c7d} \textcolor[HTML]{000000}{3.47} & \cellcolor[HTML]{ffeee6} \textcolor[HTML]{000000}{0.44} & \cellcolor[HTML]{ffede5} \textcolor[HTML]{000000}{0.48} & \cellcolor[HTML]{fee5d8} \textcolor[HTML]{000000}{0.98} & \cellcolor[HTML]{fee9df} \textcolor[HTML]{000000}{0.73} & \cellcolor[HTML]{fee5d8} \textcolor[HTML]{000000}{1.01} \\
$o_{4}$ & \cellcolor[HTML]{fff2eb} \textcolor[HTML]{000000}{0.20} & \cellcolor[HTML]{fee9df} \textcolor[HTML]{000000}{0.74} & \cellcolor[HTML]{fcb499} \textcolor[HTML]{000000}{2.73} & \cellcolor[HTML]{fee3d7} \textcolor[HTML]{000000}{1.07} & \cellcolor[HTML]{fee7dc} \textcolor[HTML]{000000}{0.85} & \cellcolor[HTML]{fdd0bc} \textcolor[HTML]{000000}{1.82} & \cellcolor[HTML]{ffeee6} \textcolor[HTML]{000000}{0.46} & \cellcolor[HTML]{fff1ea} \textcolor[HTML]{000000}{0.25} & \cellcolor[HTML]{fee8de} \textcolor[HTML]{000000}{0.76} & \cellcolor[HTML]{fdd5c4} \textcolor[HTML]{000000}{1.64} \\
$o_{5}$ & \cellcolor[HTML]{fedfd0} \textcolor[HTML]{000000}{1.32} & \cellcolor[HTML]{fff3ed} \textcolor[HTML]{000000}{0.12} & \cellcolor[HTML]{fdcbb6} \textcolor[HTML]{000000}{1.99} & \cellcolor[HTML]{fdd4c2} \textcolor[HTML]{000000}{1.66} & \cellcolor[HTML]{fee3d6} \textcolor[HTML]{000000}{1.12} & \cellcolor[HTML]{fdd2bf} \textcolor[HTML]{000000}{1.75} & \cellcolor[HTML]{ffece4} \textcolor[HTML]{000000}{0.54} & \cellcolor[HTML]{ffeee7} \textcolor[HTML]{000000}{0.40} & \cellcolor[HTML]{fff0e8} \textcolor[HTML]{000000}{0.32} & \cellcolor[HTML]{fdd0bc} \textcolor[HTML]{000000}{1.81} \\
$o_{6}$ & \cellcolor[HTML]{fcbfa7} \textcolor[HTML]{000000}{2.37} & \cellcolor[HTML]{ffece3} \textcolor[HTML]{000000}{0.57} & \cellcolor[HTML]{fdd3c1} \textcolor[HTML]{000000}{1.70} & \cellcolor[HTML]{ffefe8} \textcolor[HTML]{000000}{0.38} & \cellcolor[HTML]{fee0d2} \textcolor[HTML]{000000}{1.26} & \cellcolor[HTML]{fee3d7} \textcolor[HTML]{000000}{1.07} & \cellcolor[HTML]{feeae1} \textcolor[HTML]{000000}{0.63} & \cellcolor[HTML]{ffeee6} \textcolor[HTML]{000000}{0.45} & \cellcolor[HTML]{ffefe8} \textcolor[HTML]{000000}{0.39} & \cellcolor[HTML]{fdd2bf} \textcolor[HTML]{000000}{1.72} \\
$o_{7}$ & \cellcolor[HTML]{fee6da} \textcolor[HTML]{000000}{0.93} & \cellcolor[HTML]{fdcebb} \textcolor[HTML]{000000}{1.86} & \cellcolor[HTML]{fff2eb} \textcolor[HTML]{000000}{0.23} & \cellcolor[HTML]{fee7dc} \textcolor[HTML]{000000}{0.84} & \cellcolor[HTML]{ffede5} \textcolor[HTML]{000000}{0.48} & \cellcolor[HTML]{fff5f0} \textcolor[HTML]{000000}{0.02} & \cellcolor[HTML]{feeae0} \textcolor[HTML]{000000}{0.69} & \cellcolor[HTML]{fff5f0} \textcolor[HTML]{000000}{0.02} & \cellcolor[HTML]{feeae1} \textcolor[HTML]{000000}{0.63} & \cellcolor[HTML]{fedbcc} \textcolor[HTML]{000000}{1.41} \\
$o_{8}$ & \cellcolor[HTML]{fc8666} \textcolor[HTML]{f1f1f1}{4.12} & \cellcolor[HTML]{fee8dd} \textcolor[HTML]{000000}{0.81} & \cellcolor[HTML]{fcb99f} \textcolor[HTML]{000000}{2.57} & \cellcolor[HTML]{fedfd0} \textcolor[HTML]{000000}{1.32} & \cellcolor[HTML]{fed9c9} \textcolor[HTML]{000000}{1.52} & \cellcolor[HTML]{fee4d8} \textcolor[HTML]{000000}{1.03} & \cellcolor[HTML]{fee7db} \textcolor[HTML]{000000}{0.88} & \cellcolor[HTML]{fedaca} \textcolor[HTML]{000000}{1.48} & \cellcolor[HTML]{fee9df} \textcolor[HTML]{000000}{0.72} & \cellcolor[HTML]{fee7db} \textcolor[HTML]{000000}{0.88} \\
$o_{9}$ & \cellcolor[HTML]{fff3ed} \textcolor[HTML]{000000}{0.14} & \cellcolor[HTML]{fee5d8} \textcolor[HTML]{000000}{1.01} & \cellcolor[HTML]{fff5f0} \textcolor[HTML]{000000}{0.01} & \cellcolor[HTML]{fee2d5} \textcolor[HTML]{000000}{1.17} & \cellcolor[HTML]{fee0d2} \textcolor[HTML]{000000}{1.27} & \cellcolor[HTML]{fcbda4} \textcolor[HTML]{000000}{2.43} & \cellcolor[HTML]{fee2d5} \textcolor[HTML]{000000}{1.14} & \cellcolor[HTML]{ffeee6} \textcolor[HTML]{000000}{0.45} & \cellcolor[HTML]{fff0e8} \textcolor[HTML]{000000}{0.32} & \cellcolor[HTML]{fee3d6} \textcolor[HTML]{000000}{1.12} \\
$o_{10}$ & \cellcolor[HTML]{fff5f0} \textcolor[HTML]{000000}{0.02} & \cellcolor[HTML]{fee0d2} \textcolor[HTML]{000000}{1.28} & \cellcolor[HTML]{fdd5c4} \textcolor[HTML]{000000}{1.64} & \cellcolor[HTML]{fff4ee} \textcolor[HTML]{000000}{0.09} & \cellcolor[HTML]{fee5d9} \textcolor[HTML]{000000}{0.94} & \cellcolor[HTML]{ffede5} \textcolor[HTML]{000000}{0.50} & \cellcolor[HTML]{ffeee7} \textcolor[HTML]{000000}{0.42} & \cellcolor[HTML]{fedfd0} \textcolor[HTML]{000000}{1.32} & \cellcolor[HTML]{fff0e8} \textcolor[HTML]{000000}{0.34} & \cellcolor[HTML]{feeae1} \textcolor[HTML]{000000}{0.66} \\
\end{tabular}
} 
    \caption{Numerical values of overlaps; Cell $(i, j)$ contains the value of $\sqrt{p}|\langle o_i, \tilde{o}_j \rangle|$. Note that $\sqrt{p}\langle o_i, \tilde{o}_j \rangle \xrightarrow{p \to \infty} N(0, 1)$}
        \label{fig:resid-ret-numericals}
    \end{subfigure}%
    \hfill
    \begin{subfigure}[t]{0.48\textwidth}
        \centering
    \includegraphics[width=\linewidth]{figures/sp500_single.png}
        \caption{Tuning curves for Residualized Returns; same as in Figure~2 of the main text. }
        \label{fig:resid-ret-result}
    \end{subfigure}
    \caption{Residualized returns setting, $n = p = 493$, $r^2 = \sigma^2 = 1$; Figure~\ref{fig:resid-ret-overlaps} is essentially a histogram of values in~\ref{fig:resid-ret-numericals}.}
    \label{fig:resid-ret-big}
\end{figure}

\textbf{Overview of Setting}: to recap from the main text (Figure~2d), we considered data where each row contained the 30 minute residualized returns of a collection of 493 stocks. Here, we set $n = p = 493$ and $r^2 = \sigma^2 = 1$. 

\textbf{Checking Assumptions}:
\begin{itemize}
    \item Regularity of Singular Value Distribution: In Figure~\ref{fig:resid-ret-spectrum}, we plot a histogram of the singular values and see that it is are well-behaved with no strong outliers. 
\item On coupling: Furthermore, while the histogram of singular values (shown in Figure~\ref{fig:resid-ret-overlaps}, numerical values in \ref{fig:resid-ret-numericals}) has larger tails than expected for a normal distribution, they are not overly so, and thus we find we do not require coupling either (i.e. Assumption 6 essentially holds). As a result, applying our method without coupling performs well (result shown in Figure~\ref{fig:resid-ret-result}. 
\end{itemize}

\newpage
\section{Gaussian mixture}
\begin{figure}
    \centering
    \begin{subfigure}[t]{0.45\textwidth}
        \centering
    \includegraphics[width=\textwidth]{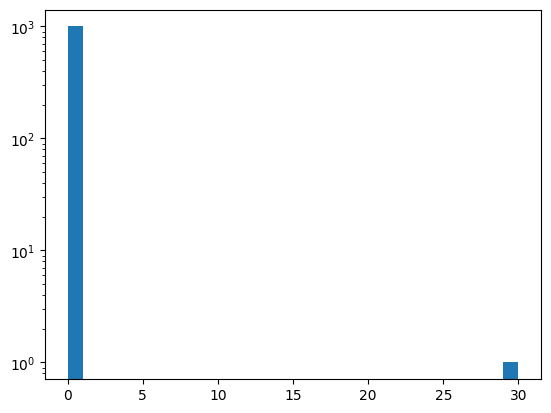}
        \caption{Histogram of singular values $(D_{ii})_{i = 1} ^n$}
        \label{fig:gaussian-mixture-spectrum}
    \end{subfigure}%
    \hfill
    \begin{subfigure}[t]{0.45\textwidth}
        \centering
    \includegraphics[width=\textwidth]{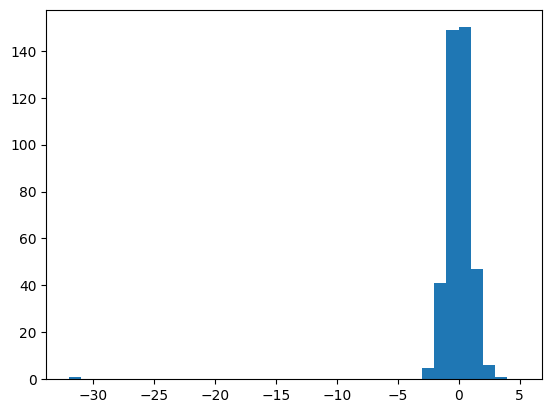}
        \caption{Histogram of overlaps $\{\sqrt{p}\langle \tilde{o}_i, o_j \rangle : 1 \leq i \leq j \leq 20 \}$.\textbf{ Note the outlier value at around $-31$.} }
        \label{fig:gaussian-mixture-overlaps}
    \end{subfigure}%
    \hfill
    \begin{subfigure}[t]{0.45\textwidth}
        \centering
{\tiny
\setlength\tabcolsep{1.5pt}\renewcommand{\arraystretch}{1.5}
\sffamily
\begin{tabular}{p{1.5em}C{2.5em}C{2.5em}C{2.5em}C{2.5em}C{2.5em}C{2.5em}C{2.5em}C{2.5em}C{2.5em}C{2.5em}}
 & $\tilde{o}_{1}$ & $\tilde{o}_{2}$ & $\tilde{o}_{3}$ & $\tilde{o}_{4}$ & $\tilde{o}_{5}$ & $\tilde{o}_{6}$ & $\tilde{o}_{7}$ & $\tilde{o}_{8}$ & $\tilde{o}_{9}$ & $\tilde{o}_{10}$ \\
$o_{1}$ & \cellcolor[HTML]{67000d} \textcolor[HTML]{f1f1f1}{31.62} & \cellcolor[HTML]{fff5f0} \textcolor[HTML]{000000}{0.01} & \cellcolor[HTML]{fff5f0} \textcolor[HTML]{000000}{0.00} & \cellcolor[HTML]{fff5f0} \textcolor[HTML]{000000}{0.01} & \cellcolor[HTML]{fff5f0} \textcolor[HTML]{000000}{0.03} & \cellcolor[HTML]{fff5f0} \textcolor[HTML]{000000}{0.00} & \cellcolor[HTML]{fff5f0} \textcolor[HTML]{000000}{0.03} & \cellcolor[HTML]{fff5f0} \textcolor[HTML]{000000}{0.02} & \cellcolor[HTML]{fff5f0} \textcolor[HTML]{000000}{0.01} & \cellcolor[HTML]{fff5f0} \textcolor[HTML]{000000}{0.01} \\
$o_{2}$ & \cellcolor[HTML]{fff5f0} \textcolor[HTML]{000000}{0.00} & \cellcolor[HTML]{ffede5} \textcolor[HTML]{000000}{0.71} & \cellcolor[HTML]{fff2eb} \textcolor[HTML]{000000}{0.35} & \cellcolor[HTML]{fff2eb} \textcolor[HTML]{000000}{0.32} & \cellcolor[HTML]{fee3d6} \textcolor[HTML]{000000}{1.67} & \cellcolor[HTML]{fff3ed} \textcolor[HTML]{000000}{0.22} & \cellcolor[HTML]{ffebe2} \textcolor[HTML]{000000}{0.92} & \cellcolor[HTML]{fff1ea} \textcolor[HTML]{000000}{0.37} & \cellcolor[HTML]{fee9df} \textcolor[HTML]{000000}{1.09} & \cellcolor[HTML]{fff1ea} \textcolor[HTML]{000000}{0.36} \\
$o_{3}$ & \cellcolor[HTML]{fff5f0} \textcolor[HTML]{000000}{0.02} & \cellcolor[HTML]{fff1ea} \textcolor[HTML]{000000}{0.38} & \cellcolor[HTML]{fee8dd} \textcolor[HTML]{000000}{1.23} & \cellcolor[HTML]{ffece4} \textcolor[HTML]{000000}{0.78} & \cellcolor[HTML]{ffeee7} \textcolor[HTML]{000000}{0.60} & \cellcolor[HTML]{fee5d9} \textcolor[HTML]{000000}{1.44} & \cellcolor[HTML]{fff5f0} \textcolor[HTML]{000000}{0.04} & \cellcolor[HTML]{ffece3} \textcolor[HTML]{000000}{0.86} & \cellcolor[HTML]{ffefe8} \textcolor[HTML]{000000}{0.56} & \cellcolor[HTML]{fff2ec} \textcolor[HTML]{000000}{0.26} \\
$o_{4}$ & \cellcolor[HTML]{fff5f0} \textcolor[HTML]{000000}{0.01} & \cellcolor[HTML]{ffede5} \textcolor[HTML]{000000}{0.74} & \cellcolor[HTML]{fff4ef} \textcolor[HTML]{000000}{0.07} & \cellcolor[HTML]{ffeee7} \textcolor[HTML]{000000}{0.61} & \cellcolor[HTML]{fff2eb} \textcolor[HTML]{000000}{0.30} & \cellcolor[HTML]{fee5d8} \textcolor[HTML]{000000}{1.50} & \cellcolor[HTML]{fff2ec} \textcolor[HTML]{000000}{0.29} & \cellcolor[HTML]{fff1ea} \textcolor[HTML]{000000}{0.41} & \cellcolor[HTML]{ffeee6} \textcolor[HTML]{000000}{0.65} & \cellcolor[HTML]{fff4ef} \textcolor[HTML]{000000}{0.11} \\
$o_{5}$ & \cellcolor[HTML]{fff5f0} \textcolor[HTML]{000000}{0.01} & \cellcolor[HTML]{ffefe8} \textcolor[HTML]{000000}{0.57} & \cellcolor[HTML]{feeae0} \textcolor[HTML]{000000}{1.05} & \cellcolor[HTML]{fff2ec} \textcolor[HTML]{000000}{0.27} & \cellcolor[HTML]{fff3ed} \textcolor[HTML]{000000}{0.21} & \cellcolor[HTML]{fff1ea} \textcolor[HTML]{000000}{0.36} & \cellcolor[HTML]{feeae1} \textcolor[HTML]{000000}{0.97} & \cellcolor[HTML]{ffefe8} \textcolor[HTML]{000000}{0.58} & \cellcolor[HTML]{feeae1} \textcolor[HTML]{000000}{0.99} & \cellcolor[HTML]{fff0e8} \textcolor[HTML]{000000}{0.47} \\
$o_{6}$ & \cellcolor[HTML]{fff5f0} \textcolor[HTML]{000000}{0.02} & \cellcolor[HTML]{fff1ea} \textcolor[HTML]{000000}{0.37} & \cellcolor[HTML]{fff0e8} \textcolor[HTML]{000000}{0.49} & \cellcolor[HTML]{fff1ea} \textcolor[HTML]{000000}{0.37} & \cellcolor[HTML]{fee7db} \textcolor[HTML]{000000}{1.31} & \cellcolor[HTML]{ffeee6} \textcolor[HTML]{000000}{0.67} & \cellcolor[HTML]{feeae1} \textcolor[HTML]{000000}{0.94} & \cellcolor[HTML]{fedccd} \textcolor[HTML]{000000}{2.08} & \cellcolor[HTML]{fff4ee} \textcolor[HTML]{000000}{0.16} & \cellcolor[HTML]{fee7dc} \textcolor[HTML]{000000}{1.24} \\
$o_{7}$ & \cellcolor[HTML]{fff5f0} \textcolor[HTML]{000000}{0.01} & \cellcolor[HTML]{fff3ed} \textcolor[HTML]{000000}{0.18} & \cellcolor[HTML]{fee8dd} \textcolor[HTML]{000000}{1.20} & \cellcolor[HTML]{fff3ed} \textcolor[HTML]{000000}{0.22} & \cellcolor[HTML]{fff5f0} \textcolor[HTML]{000000}{0.02} & \cellcolor[HTML]{fee8de} \textcolor[HTML]{000000}{1.12} & \cellcolor[HTML]{ffece4} \textcolor[HTML]{000000}{0.78} & \cellcolor[HTML]{ffebe2} \textcolor[HTML]{000000}{0.91} & \cellcolor[HTML]{ffece4} \textcolor[HTML]{000000}{0.81} & \cellcolor[HTML]{fed8c7} \textcolor[HTML]{000000}{2.33} \\
$o_{8}$ & \cellcolor[HTML]{fff5f0} \textcolor[HTML]{000000}{0.01} & \cellcolor[HTML]{fff2eb} \textcolor[HTML]{000000}{0.33} & \cellcolor[HTML]{fff5f0} \textcolor[HTML]{000000}{0.02} & \cellcolor[HTML]{ffede5} \textcolor[HTML]{000000}{0.72} & \cellcolor[HTML]{ffece3} \textcolor[HTML]{000000}{0.87} & \cellcolor[HTML]{fee5d8} \textcolor[HTML]{000000}{1.49} & \cellcolor[HTML]{fff2ec} \textcolor[HTML]{000000}{0.25} & \cellcolor[HTML]{fdd5c4} \textcolor[HTML]{000000}{2.44} & \cellcolor[HTML]{ffefe8} \textcolor[HTML]{000000}{0.58} & \cellcolor[HTML]{fee7dc} \textcolor[HTML]{000000}{1.27} \\
$o_{9}$ & \cellcolor[HTML]{fff5f0} \textcolor[HTML]{000000}{0.03} & \cellcolor[HTML]{fee4d8} \textcolor[HTML]{000000}{1.55} & \cellcolor[HTML]{ffefe8} \textcolor[HTML]{000000}{0.54} & \cellcolor[HTML]{ffebe2} \textcolor[HTML]{000000}{0.92} & \cellcolor[HTML]{fff0e9} \textcolor[HTML]{000000}{0.43} & \cellcolor[HTML]{fff4ee} \textcolor[HTML]{000000}{0.13} & \cellcolor[HTML]{fff0e9} \textcolor[HTML]{000000}{0.46} & \cellcolor[HTML]{fff0e9} \textcolor[HTML]{000000}{0.45} & \cellcolor[HTML]{fff2eb} \textcolor[HTML]{000000}{0.33} & \cellcolor[HTML]{fff5f0} \textcolor[HTML]{000000}{0.04} \\
$o_{10}$ & \cellcolor[HTML]{fff5f0} \textcolor[HTML]{000000}{0.00} & \cellcolor[HTML]{fff1ea} \textcolor[HTML]{000000}{0.36} & \cellcolor[HTML]{feeae0} \textcolor[HTML]{000000}{1.01} & \cellcolor[HTML]{fee3d6} \textcolor[HTML]{000000}{1.65} & \cellcolor[HTML]{feeae1} \textcolor[HTML]{000000}{0.95} & \cellcolor[HTML]{fee8de} \textcolor[HTML]{000000}{1.15} & \cellcolor[HTML]{ffece3} \textcolor[HTML]{000000}{0.83} & \cellcolor[HTML]{ffede5} \textcolor[HTML]{000000}{0.75} & \cellcolor[HTML]{ffece3} \textcolor[HTML]{000000}{0.83} & \cellcolor[HTML]{ffece4} \textcolor[HTML]{000000}{0.80} \\
\end{tabular}
}
    \caption{Numerical values of overlaps; Cell $(i, j)$ contains the value of $\sqrt{p}|\langle o_i, \tilde{o}_j \rangle|$. Note that $\sqrt{p}\langle o_i, \tilde{o}_j \rangle \xrightarrow{p \to \infty} N(0, 1)$}
        \label{fig:gaussian-mixture-numericals}
    \end{subfigure}%
    \hfill
    \begin{subfigure}[t]{0.45\textwidth}
        \centering
    \includegraphics[width=\textwidth]{figures/mixture.png}
        \caption{Tuning curves for Gaussian mixtures}
        \label{fig:gaussian-mixture-result}
    \end{subfigure}
    \caption{Gaussian mixture setting, $n = p = 1000$; Figure~\ref{fig:gaussian-mixture-overlaps} is essentially a histogram of values in~\ref{fig:gaussian-mixture-numericals}.}
    \label{fig:gaussian-mixture-big}
\end{figure}

To build intuition for the coupled eigenvectors condition (Assumption 7), we take a first look at how these diagnostic plots change when the data is instead drawn from the Gaussian mixture discussed in the main text (restated below):

\textbf{Overview of setting}: Each row $x_i \sim \frac{1}{2}N(\vec{3}, I_p) + \frac{1}{2}N(-\vec{3}, I_p)$. We take $n = p = 1000$, and $r^2 = \sigma^2 = 1$. 

\textbf{Checking assumptions}:
\begin{itemize}
    \item Regularity of Singular Value Distribution: As seen in Figure~\ref{fig:gaussian-mixture-spectrum}, we have one extremely large singular value, showing Assumption 2 is violated. This is usually a bad sign for our method, but in this case it manages to succeed, showing some robustness.
    \item On coupling: In Figures~\ref{fig:gaussian-mixture-overlaps} and \ref{fig:gaussian-mixture-numericals}, the huge outlier overlap is indication that the top two eigenvectors overlap with one another, meaning Assumption 6 does not hold. On the other hand, all the other eigenvectors do not align with each other. This is exactly the situation suited for our coupled eigenvectors condition (Assumption 7). As a result, coupling $o_1$ with $\tilde{o}_1$ produces solid results, shown in Figure~\ref{fig:gaussian-mixture-result}. 
\end{itemize}

\newpage
\section{Speech Data}\label{sec:speech}
\begin{figure}
    \centering
    \begin{subfigure}[t]{0.45\textwidth}
        \centering
    \includegraphics[width=\linewidth]{figures/speech-data-spectrum.png}
        \caption{Histogram of singular values $(D_{ii})_{i = 1} ^n$}
        \label{fig:speech-data-spectrum-v2}
    \end{subfigure}%
    \hfill
    \begin{subfigure}[t]{0.45\textwidth}
        \centering
    \includegraphics[width=\linewidth]{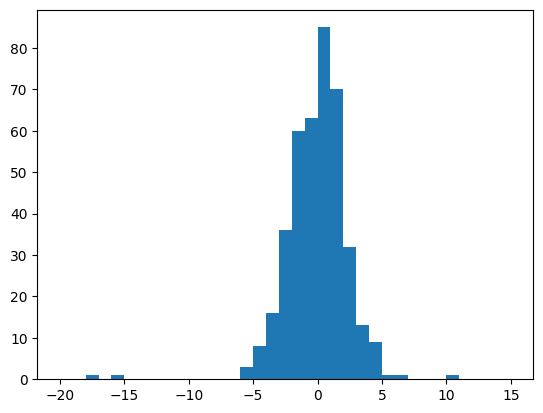}
        \caption{Histogram of overlaps $\{\sqrt{p}\langle \tilde{o}_i, o_j \rangle : 1 \leq i \leq j \leq 20 \}$}
        \label{fig:speech-data-overlaps}
    \end{subfigure}%
    \hfill
    \begin{subfigure}[t]{0.45\textwidth}
        \centering
{\tiny
\setlength\tabcolsep{1.5pt}\renewcommand{\arraystretch}{1.5}
\sffamily
\begin{tabular}{p{1.5em}C{2.5em}C{2.5em}C{2.5em}C{2.5em}C{2.5em}C{2.5em}C{2.5em}C{2.5em}C{2.5em}C{2.5em}}
 & $\tilde{o}_{1}$ & $\tilde{o}_{2}$ & $\tilde{o}_{3}$ & $\tilde{o}_{4}$ & $\tilde{o}_{5}$ & $\tilde{o}_{6}$ & $\tilde{o}_{7}$ & $\tilde{o}_{8}$ & $\tilde{o}_{9}$ & $\tilde{o}_{10}$ \\
$o_{1}$ & \cellcolor[HTML]{67000d} \textcolor[HTML]{f1f1f1}{17.30} & \cellcolor[HTML]{fdd0bc} \textcolor[HTML]{000000}{2.72} & \cellcolor[HTML]{ffeee7} \textcolor[HTML]{000000}{0.63} & \cellcolor[HTML]{fedecf} \textcolor[HTML]{000000}{2.01} & \cellcolor[HTML]{fee4d8} \textcolor[HTML]{000000}{1.55} & \cellcolor[HTML]{feeae0} \textcolor[HTML]{000000}{1.02} & \cellcolor[HTML]{fff2eb} \textcolor[HTML]{000000}{0.30} & \cellcolor[HTML]{feeae1} \textcolor[HTML]{000000}{0.96} & \cellcolor[HTML]{fee2d5} \textcolor[HTML]{000000}{1.70} & \cellcolor[HTML]{fff2ec} \textcolor[HTML]{000000}{0.27} \\
$o_{2}$ & \cellcolor[HTML]{ffece3} \textcolor[HTML]{000000}{0.83} & \cellcolor[HTML]{67000d} \textcolor[HTML]{f1f1f1}{15.24} & \cellcolor[HTML]{ffeee6} \textcolor[HTML]{000000}{0.68} & \cellcolor[HTML]{fee3d7} \textcolor[HTML]{000000}{1.63} & \cellcolor[HTML]{fee7db} \textcolor[HTML]{000000}{1.34} & \cellcolor[HTML]{fee8dd} \textcolor[HTML]{000000}{1.22} & \cellcolor[HTML]{ffebe2} \textcolor[HTML]{000000}{0.89} & \cellcolor[HTML]{fee5d9} \textcolor[HTML]{000000}{1.42} & \cellcolor[HTML]{fff2ec} \textcolor[HTML]{000000}{0.27} & \cellcolor[HTML]{fedaca} \textcolor[HTML]{000000}{2.19} \\
$o_{3}$ & \cellcolor[HTML]{fee9df} \textcolor[HTML]{000000}{1.08} & \cellcolor[HTML]{fee0d2} \textcolor[HTML]{000000}{1.93} & \cellcolor[HTML]{d11e1f} \textcolor[HTML]{f1f1f1}{10.93} & \cellcolor[HTML]{fee3d7} \textcolor[HTML]{000000}{1.60} & \cellcolor[HTML]{ffeee7} \textcolor[HTML]{000000}{0.64} & \cellcolor[HTML]{fedecf} \textcolor[HTML]{000000}{2.02} & \cellcolor[HTML]{fca78b} \textcolor[HTML]{000000}{4.63} & \cellcolor[HTML]{fff2eb} \textcolor[HTML]{000000}{0.35} & \cellcolor[HTML]{fcb398} \textcolor[HTML]{000000}{4.14} & \cellcolor[HTML]{fdd0bc} \textcolor[HTML]{000000}{2.75} \\
$o_{4}$ & \cellcolor[HTML]{fff2ec} \textcolor[HTML]{000000}{0.29} & \cellcolor[HTML]{fdcebb} \textcolor[HTML]{000000}{2.78} & \cellcolor[HTML]{fdccb8} \textcolor[HTML]{000000}{2.92} & \cellcolor[HTML]{fff3ed} \textcolor[HTML]{000000}{0.22} & \cellcolor[HTML]{fcb79c} \textcolor[HTML]{000000}{3.98} & \cellcolor[HTML]{fb7c5c} \textcolor[HTML]{f1f1f1}{6.63} & \cellcolor[HTML]{fdcbb6} \textcolor[HTML]{000000}{2.98} & \cellcolor[HTML]{fee3d6} \textcolor[HTML]{000000}{1.65} & \cellcolor[HTML]{ffefe8} \textcolor[HTML]{000000}{0.54} & \cellcolor[HTML]{fee4d8} \textcolor[HTML]{000000}{1.56} \\
$o_{5}$ & \cellcolor[HTML]{fee7dc} \textcolor[HTML]{000000}{1.27} & \cellcolor[HTML]{fdd0bc} \textcolor[HTML]{000000}{2.70} & \cellcolor[HTML]{fcb69b} \textcolor[HTML]{000000}{4.02} & \cellcolor[HTML]{fedbcc} \textcolor[HTML]{000000}{2.14} & \cellcolor[HTML]{fc9d7f} \textcolor[HTML]{000000}{5.11} & \cellcolor[HTML]{fee7db} \textcolor[HTML]{000000}{1.29} & \cellcolor[HTML]{fcae92} \textcolor[HTML]{000000}{4.34} & \cellcolor[HTML]{fca689} \textcolor[HTML]{000000}{4.72} & \cellcolor[HTML]{fdcbb6} \textcolor[HTML]{000000}{2.95} & \cellcolor[HTML]{fed9c9} \textcolor[HTML]{000000}{2.24} \\
$o_{6}$ & \cellcolor[HTML]{fee8dd} \textcolor[HTML]{000000}{1.19} & \cellcolor[HTML]{ffefe8} \textcolor[HTML]{000000}{0.56} & \cellcolor[HTML]{fca486} \textcolor[HTML]{000000}{4.82} & \cellcolor[HTML]{fc8b6b} \textcolor[HTML]{f1f1f1}{5.95} & \cellcolor[HTML]{fff1ea} \textcolor[HTML]{000000}{0.37} & \cellcolor[HTML]{fcb99f} \textcolor[HTML]{000000}{3.83} & \cellcolor[HTML]{fee6da} \textcolor[HTML]{000000}{1.37} & \cellcolor[HTML]{ffece3} \textcolor[HTML]{000000}{0.87} & \cellcolor[HTML]{fca183} \textcolor[HTML]{000000}{4.98} & \cellcolor[HTML]{feeae0} \textcolor[HTML]{000000}{1.04} \\
$o_{7}$ & \cellcolor[HTML]{fedaca} \textcolor[HTML]{000000}{2.21} & \cellcolor[HTML]{ffeee6} \textcolor[HTML]{000000}{0.70} & \cellcolor[HTML]{fedbcc} \textcolor[HTML]{000000}{2.14} & \cellcolor[HTML]{fdd1be} \textcolor[HTML]{000000}{2.65} & \cellcolor[HTML]{fca486} \textcolor[HTML]{000000}{4.84} & \cellcolor[HTML]{fee1d4} \textcolor[HTML]{000000}{1.77} & \cellcolor[HTML]{ffefe8} \textcolor[HTML]{000000}{0.54} & \cellcolor[HTML]{feeae0} \textcolor[HTML]{000000}{1.00} & \cellcolor[HTML]{fee3d6} \textcolor[HTML]{000000}{1.65} & \cellcolor[HTML]{fdd3c1} \textcolor[HTML]{000000}{2.53} \\
$o_{8}$ & \cellcolor[HTML]{fff0e9} \textcolor[HTML]{000000}{0.46} & \cellcolor[HTML]{fff4ee} \textcolor[HTML]{000000}{0.15} & \cellcolor[HTML]{fdcbb6} \textcolor[HTML]{000000}{2.96} & \cellcolor[HTML]{fff4ee} \textcolor[HTML]{000000}{0.16} & \cellcolor[HTML]{fee3d6} \textcolor[HTML]{000000}{1.66} & \cellcolor[HTML]{fcab8f} \textcolor[HTML]{000000}{4.49} & \cellcolor[HTML]{fff3ed} \textcolor[HTML]{000000}{0.22} & \cellcolor[HTML]{fff1ea} \textcolor[HTML]{000000}{0.40} & \cellcolor[HTML]{fee1d3} \textcolor[HTML]{000000}{1.82} & \cellcolor[HTML]{fff4ef} \textcolor[HTML]{000000}{0.10} \\
$o_{9}$ & \cellcolor[HTML]{fff2eb} \textcolor[HTML]{000000}{0.30} & \cellcolor[HTML]{fff1ea} \textcolor[HTML]{000000}{0.40} & \cellcolor[HTML]{ffede5} \textcolor[HTML]{000000}{0.72} & \cellcolor[HTML]{fff5f0} \textcolor[HTML]{000000}{0.01} & \cellcolor[HTML]{fee3d7} \textcolor[HTML]{000000}{1.63} & \cellcolor[HTML]{fee6da} \textcolor[HTML]{000000}{1.39} & \cellcolor[HTML]{fcb99f} \textcolor[HTML]{000000}{3.86} & \cellcolor[HTML]{fcaa8d} \textcolor[HTML]{000000}{4.53} & \cellcolor[HTML]{fee8de} \textcolor[HTML]{000000}{1.12} & \cellcolor[HTML]{fdd3c1} \textcolor[HTML]{000000}{2.57} \\
$o_{10}$ & \cellcolor[HTML]{fee8de} \textcolor[HTML]{000000}{1.14} & \cellcolor[HTML]{fdd4c2} \textcolor[HTML]{000000}{2.51} & \cellcolor[HTML]{ffebe2} \textcolor[HTML]{000000}{0.90} & \cellcolor[HTML]{ffece4} \textcolor[HTML]{000000}{0.77} & \cellcolor[HTML]{fff1ea} \textcolor[HTML]{000000}{0.37} & \cellcolor[HTML]{ffece4} \textcolor[HTML]{000000}{0.79} & \cellcolor[HTML]{fdd5c4} \textcolor[HTML]{000000}{2.44} & \cellcolor[HTML]{fee1d3} \textcolor[HTML]{000000}{1.83} & \cellcolor[HTML]{fee1d4} \textcolor[HTML]{000000}{1.81} & \cellcolor[HTML]{fed8c7} \textcolor[HTML]{000000}{2.34} \\
\end{tabular}
}

    \caption{Numerical values of overlaps; Cell $(i, j)$ contains the value of $\sqrt{p}|\langle o_i, \tilde{o}_j \rangle|$. Note that $\sqrt{p}\langle o_i, \tilde{o}_j \rangle \xrightarrow{p \to \infty} N(0, 1)$}
        \label{fig:speech-data-numericals}
    \end{subfigure}%
    \hfill
    \begin{subfigure}[t]{0.45\textwidth}
        \centering
    \includegraphics[width=\linewidth]{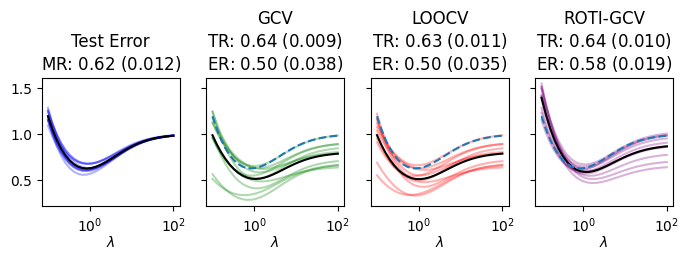}
        \caption{Tuning curves for Speech Data}
        \label{fig:speech-data-result-v2}
    \end{subfigure}
    \caption{Speech data setting, $n = p = 400$, $r^2 = \sigma^2 = 1$; Figure~\ref{fig:speech-data-overlaps} is essentially a histogram of values in~\ref{fig:speech-data-numericals}.}
    \label{fig:speech-data-big}
\end{figure}
We find a good use case for our coupled eigenvector condition on speech data taken from OpenML. This is a real dataset which is not included in the main manuscript.

\textbf{Overview of setting}: Dataset sourced is speech data from \cite{OpenMLSpeech}. $n = p = 400$, and $r^2 = \sigma^2 = 1$. 

\textbf{Checking assumptions}:
\begin{itemize}
    \item Regularity of Singular Value Distribution: Here, in Figure~\ref{fig:speech-data-spectrum-v2}, the spectrum has no outlier values.
    \item On coupling: Next, we observe that for each of the top 3 eigenvectors training set, they align strongly with the corresponding eigenvector of the test set, and none of the others. This is exactly the setting in which our coupled eigenvectors can be used. Furthermore, each of the top train eigenvectors is strongly aligned with only one of the test eigenvectors. In light of Table~\ref{fig:speech-data-numericals}, we choose to couple the first 3 eigenvectors with the respective test-set eigenvectors. This produces the result shown in Figure~\ref{fig:speech-data-result-v2}, showing that it successfully captures the required structure. This result is quite robust to the choice of $J_c$; using $6$ eigenvectors or even $10$ does not change the outcome much. 
\end{itemize}
\newpage 
\section{Poorly behaved case}
\begin{figure}
    \centering
    \begin{subfigure}[t]{0.45\textwidth}
        \centering
        \includegraphics[width=\textwidth]{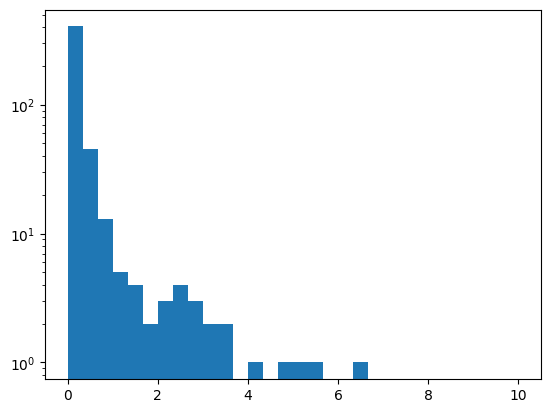}
        \caption{Histogram of singular values $(D_{ii})_{i = 1} ^n$}
        \label{fig:bad-spectrum}
    \end{subfigure}%
    \hfill
    \begin{subfigure}[t]{0.45\textwidth}
        \centering
        \includegraphics[width=\textwidth]{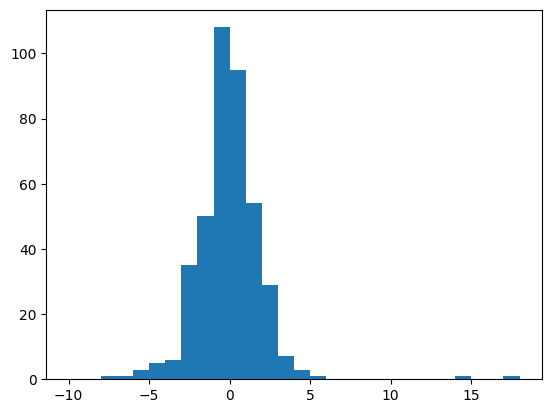}
        \caption{Histogram of overlaps $\{\sqrt{p}\langle \tilde{o}_i, o_j \rangle : 1 \leq i \leq j \leq 20 \}$}
        \label{fig:bad-overlaps}
    \end{subfigure}%
    \hfill
    \begin{subtable}[t]{0.45\textwidth}
        \centering
{\tiny
\setlength\tabcolsep{1.5pt}\renewcommand{\arraystretch}{1.5}
\sffamily
\begin{tabular}{p{1.5em}C{2.5em}C{2.5em}C{2.5em}C{2.5em}C{2.5em}C{2.5em}C{2.5em}C{2.5em}C{2.5em}C{2.5em}}
 & $\tilde{o}_{1}$ & $\tilde{o}_{2}$ & $\tilde{o}_{3}$ & $\tilde{o}_{4}$ & $\tilde{o}_{5}$ & $\tilde{o}_{6}$ & $\tilde{o}_{7}$ & $\tilde{o}_{8}$ & $\tilde{o}_{9}$ & $\tilde{o}_{10}$ \\
$o_{1}$ & \cellcolor[HTML]{67000d} \textcolor[HTML]{f1f1f1}{17.23} & \cellcolor[HTML]{fcbda4} \textcolor[HTML]{000000}{2.43} & \cellcolor[HTML]{fcb095} \textcolor[HTML]{000000}{2.84} & \cellcolor[HTML]{fee5d8} \textcolor[HTML]{000000}{0.98} & \cellcolor[HTML]{fcad90} \textcolor[HTML]{000000}{2.93} & \cellcolor[HTML]{ffece3} \textcolor[HTML]{000000}{0.55} & \cellcolor[HTML]{fff1ea} \textcolor[HTML]{000000}{0.27} & \cellcolor[HTML]{fee7dc} \textcolor[HTML]{000000}{0.84} & \cellcolor[HTML]{fc8161} \textcolor[HTML]{f1f1f1}{4.29} & \cellcolor[HTML]{fcb095} \textcolor[HTML]{000000}{2.83} \\
$o_{2}$ & \cellcolor[HTML]{fee9df} \textcolor[HTML]{000000}{0.72} & \cellcolor[HTML]{67000d} \textcolor[HTML]{f1f1f1}{14.25} & \cellcolor[HTML]{fee1d3} \textcolor[HTML]{000000}{1.22} & \cellcolor[HTML]{fff2eb} \textcolor[HTML]{000000}{0.23} & \cellcolor[HTML]{f5533b} \textcolor[HTML]{f1f1f1}{5.62} & \cellcolor[HTML]{fedecf} \textcolor[HTML]{000000}{1.36} & \cellcolor[HTML]{fcc3ab} \textcolor[HTML]{000000}{2.26} & \cellcolor[HTML]{fcbfa7} \textcolor[HTML]{000000}{2.36} & \cellcolor[HTML]{ffede5} \textcolor[HTML]{000000}{0.49} & \cellcolor[HTML]{ffede5} \textcolor[HTML]{000000}{0.50} \\
$o_{3}$ & \cellcolor[HTML]{ee3a2c} \textcolor[HTML]{f1f1f1}{6.28} & \cellcolor[HTML]{fcb89e} \textcolor[HTML]{000000}{2.60} & \cellcolor[HTML]{fc8565} \textcolor[HTML]{f1f1f1}{4.16} & \cellcolor[HTML]{fcbda4} \textcolor[HTML]{000000}{2.43} & \cellcolor[HTML]{fee5d9} \textcolor[HTML]{000000}{0.95} & \cellcolor[HTML]{fca588} \textcolor[HTML]{000000}{3.18} & \cellcolor[HTML]{fee8de} \textcolor[HTML]{000000}{0.77} & \cellcolor[HTML]{fb7d5d} \textcolor[HTML]{f1f1f1}{4.40} & \cellcolor[HTML]{fc9879} \textcolor[HTML]{000000}{3.59} & \cellcolor[HTML]{fee1d4} \textcolor[HTML]{000000}{1.18} \\
$o_{4}$ & \cellcolor[HTML]{fee7dc} \textcolor[HTML]{000000}{0.83} & \cellcolor[HTML]{fcc3ab} \textcolor[HTML]{000000}{2.23} & \cellcolor[HTML]{feeae1} \textcolor[HTML]{000000}{0.66} & \cellcolor[HTML]{fee4d8} \textcolor[HTML]{000000}{1.02} & \cellcolor[HTML]{fff2eb} \textcolor[HTML]{000000}{0.20} & \cellcolor[HTML]{fee3d6} \textcolor[HTML]{000000}{1.12} & \cellcolor[HTML]{fedccd} \textcolor[HTML]{000000}{1.39} & \cellcolor[HTML]{fcae92} \textcolor[HTML]{000000}{2.90} & \cellcolor[HTML]{fcbea5} \textcolor[HTML]{000000}{2.41} & \cellcolor[HTML]{ffefe8} \textcolor[HTML]{000000}{0.37} \\
$o_{5}$ & \cellcolor[HTML]{f34935} \textcolor[HTML]{f1f1f1}{5.87} & \cellcolor[HTML]{c8171c} \textcolor[HTML]{f1f1f1}{7.61} & \cellcolor[HTML]{fc8262} \textcolor[HTML]{f1f1f1}{4.22} & \cellcolor[HTML]{ffebe2} \textcolor[HTML]{000000}{0.59} & \cellcolor[HTML]{fc8a6a} \textcolor[HTML]{f1f1f1}{4.00} & \cellcolor[HTML]{ffeee6} \textcolor[HTML]{000000}{0.45} & \cellcolor[HTML]{fff4ee} \textcolor[HTML]{000000}{0.10} & \cellcolor[HTML]{fee1d3} \textcolor[HTML]{000000}{1.23} & \cellcolor[HTML]{fff2ec} \textcolor[HTML]{000000}{0.19} & \cellcolor[HTML]{ffece4} \textcolor[HTML]{000000}{0.54} \\
$o_{6}$ & \cellcolor[HTML]{fee8dd} \textcolor[HTML]{000000}{0.81} & \cellcolor[HTML]{fedfd0} \textcolor[HTML]{000000}{1.30} & \cellcolor[HTML]{ffece4} \textcolor[HTML]{000000}{0.53} & \cellcolor[HTML]{fcbda4} \textcolor[HTML]{000000}{2.43} & \cellcolor[HTML]{fcc2aa} \textcolor[HTML]{000000}{2.27} & \cellcolor[HTML]{ffede5} \textcolor[HTML]{000000}{0.47} & \cellcolor[HTML]{f6553c} \textcolor[HTML]{f1f1f1}{5.56} & \cellcolor[HTML]{fdd1be} \textcolor[HTML]{000000}{1.78} & \cellcolor[HTML]{fca183} \textcolor[HTML]{000000}{3.29} & \cellcolor[HTML]{fee9df} \textcolor[HTML]{000000}{0.72} \\
$o_{7}$ & \cellcolor[HTML]{fff4ee} \textcolor[HTML]{000000}{0.09} & \cellcolor[HTML]{fff2eb} \textcolor[HTML]{000000}{0.22} & \cellcolor[HTML]{fdc9b3} \textcolor[HTML]{000000}{2.07} & \cellcolor[HTML]{fcb398} \textcolor[HTML]{000000}{2.76} & \cellcolor[HTML]{ffefe8} \textcolor[HTML]{000000}{0.38} & \cellcolor[HTML]{fdd3c1} \textcolor[HTML]{000000}{1.69} & \cellcolor[HTML]{fc8262} \textcolor[HTML]{f1f1f1}{4.25} & \cellcolor[HTML]{fcc2aa} \textcolor[HTML]{000000}{2.30} & \cellcolor[HTML]{fdcebb} \textcolor[HTML]{000000}{1.87} & \cellcolor[HTML]{fee2d5} \textcolor[HTML]{000000}{1.14} \\
$o_{8}$ & \cellcolor[HTML]{fed9c9} \textcolor[HTML]{000000}{1.50} & \cellcolor[HTML]{ffeee6} \textcolor[HTML]{000000}{0.44} & \cellcolor[HTML]{fee6da} \textcolor[HTML]{000000}{0.91} & \cellcolor[HTML]{fc9777} \textcolor[HTML]{000000}{3.61} & \cellcolor[HTML]{fedecf} \textcolor[HTML]{000000}{1.33} & \cellcolor[HTML]{fb7656} \textcolor[HTML]{f1f1f1}{4.63} & \cellcolor[HTML]{fff5f0} \textcolor[HTML]{000000}{0.01} & \cellcolor[HTML]{fff0e8} \textcolor[HTML]{000000}{0.32} & \cellcolor[HTML]{fff0e9} \textcolor[HTML]{000000}{0.29} & \cellcolor[HTML]{ffeee6} \textcolor[HTML]{000000}{0.43} \\
$o_{9}$ & \cellcolor[HTML]{fcb79c} \textcolor[HTML]{000000}{2.62} & \cellcolor[HTML]{fedfd0} \textcolor[HTML]{000000}{1.31} & \cellcolor[HTML]{fff3ed} \textcolor[HTML]{000000}{0.14} & \cellcolor[HTML]{fcbea5} \textcolor[HTML]{000000}{2.42} & \cellcolor[HTML]{fdd4c2} \textcolor[HTML]{000000}{1.66} & \cellcolor[HTML]{fcbea5} \textcolor[HTML]{000000}{2.42} & \cellcolor[HTML]{fcb79c} \textcolor[HTML]{000000}{2.62} & \cellcolor[HTML]{fff4ee} \textcolor[HTML]{000000}{0.08} & \cellcolor[HTML]{fff2eb} \textcolor[HTML]{000000}{0.21} & \cellcolor[HTML]{ffefe8} \textcolor[HTML]{000000}{0.38} \\
$o_{10}$ & \cellcolor[HTML]{fee6da} \textcolor[HTML]{000000}{0.92} & \cellcolor[HTML]{fcb095} \textcolor[HTML]{000000}{2.82} & \cellcolor[HTML]{fdd2bf} \textcolor[HTML]{000000}{1.72} & \cellcolor[HTML]{fcbba1} \textcolor[HTML]{000000}{2.53} & \cellcolor[HTML]{fee5d8} \textcolor[HTML]{000000}{0.99} & \cellcolor[HTML]{fff2ec} \textcolor[HTML]{000000}{0.16} & \cellcolor[HTML]{feeae0} \textcolor[HTML]{000000}{0.70} & \cellcolor[HTML]{fee2d5} \textcolor[HTML]{000000}{1.14} & \cellcolor[HTML]{ffeee6} \textcolor[HTML]{000000}{0.46} & \cellcolor[HTML]{feeae0} \textcolor[HTML]{000000}{0.67} \\
\end{tabular}
}
    \caption{Numerical values of overlaps; Cell $(i, j)$ contains the value of $\sqrt{p}|\langle o_i, \tilde{o}_j \rangle|$. Note that $\sqrt{p}\langle o_i, \tilde{o}_j \rangle \xrightarrow{p \to \infty} N(0, 1)$}
        \label{fig:bad-numericals}
    \end{subtable}%
    \hfill
    \begin{subfigure}[t]{0.45\textwidth}
        \centering
        \includegraphics[width=\textwidth]{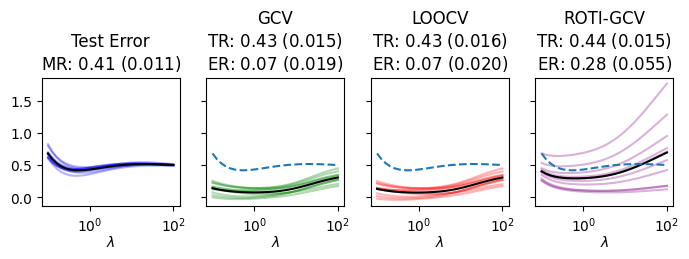}
        \caption{Tuning curves for unresidualized returns}
        \label{fig:bad-result}
    \end{subfigure}
    \caption{Poorly behaved setting, unresidualized returns, $n = p = 493$; Figure~\ref{fig:bad-overlaps} is essentially a histogram of values in~\ref{fig:bad-numericals}}
    \label{fig:bad}
\end{figure}

We will now look at a case that behaves poorly for our estimator. Here, we look at unresidualized returns. 

\textbf{Overview of Setting}: we consider data where each row contains 30 minute returns, but without residualization.

\textbf{Checking Assumptions}: 
\begin{itemize}
    \item Regularity of singular values: Here, in Figure~\ref{fig:bad-spectrum}, there are some larger singular values, but they are not extreme outliers.
    \item On coupling: We see that when checking the alignment of test and training eigenvectors in Figures~\ref{fig:bad-overlaps} and \ref{fig:bad-numericals}, the top eigenvector of the test set aligns heavily with mulitple eigenvectors of the training set. This is a scenario in which we cannot use the coupled eigenvector condition, since we can only couple one train eigenvector with one test eigenvector. As a result, we see in Figure~\ref{fig:bad-result} that our method does not do well in this scenario, which is expected. Note that this example illustrates that these diagnostics can be used ahead of time to understand whether or not our method should succeed. 
\end{itemize}
\newpage
\section{Speech data with both alignment and coupling}
\begin{figure}
    \centering
    \includegraphics[width=0.5\textwidth]{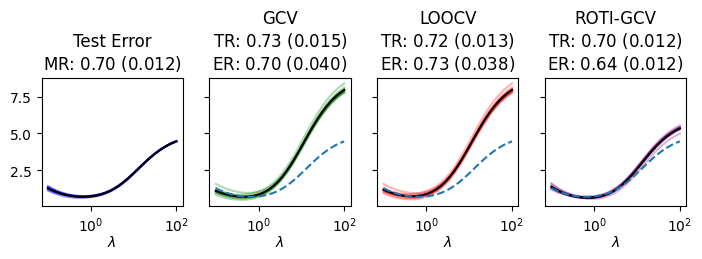}
    \caption{Speech data with both alignment and coupling}
    \label{fig:final}
\end{figure}
\begin{table}
        \centering
        \caption{BHq adjusted $p$-value for alignment.}
        \begin{tabular}{lrr}
         &  $p$ \\
        $o_1$ &  0.000 \\
        $o_{2}$ &  0.000 \\
        $o_{3}$ &  0.000 \\
        $o_{4}$ &  0.000 \\
        $o_{5}$ &  0.000 \\
        $o_{6}$ &  0.934 \\
        $o_{7}$ &  0.588 \\
        $o_{8}$ &  0.651 \\
        $o_{9}$ &  0.913 \\
        $o_{10}$ & 0.395 \\
        \end{tabular}
        \label{table:final-v2}
\end{table}

As a final example that includes both alignment and coupling, we perform another experiment with speech data that additionally contains Signal-PC alignment.

\textbf{Overview of setting}: The dataset is again the speech data from \cite{OpenMLSpeech}. Recall that an aligned signal takes the form $\beta = \sum_{i \in J_c} ^n \sqrt{n} \alpha_i + \beta'$, where the $\alpha_i$ are the weights of the aligned components, and $\beta'$ is the unaligned component of the signal. We take $\alpha_i = 1/2$ for $i \in J_a = \{1, 2, 3, 4, 5\}$.

\textbf{Checking assumptions}:
\begin{itemize}
    \item Singular value histogram and coupling detection is equivalent to that of Section~\ref{sec:speech}.
    \item Alignment: Here we next identify alignment through the hypothesis testing framework of \cite{li2023spectrumaware}. We include in Table~\ref{table:final} the Benjamini-Hochberg adjusted $p$-values of alignment of each component, finding that $p$-values for each of $i = 1, 2, 3, 4, 5$ are extremely small ($<10^{-2}$), and those of the unaligned components are large $(>0.35)$, meaning we perfectly identify the aligned portion. Our method then performs well in this setting, as shown in Figure~\ref{fig:final}.
\end{itemize}
\end{document}